\documentclass[preprint,12pt,3p]{elsarticle}
\usepackage{amsmath}
\usepackage{amssymb}
\usepackage{lineno}
\usepackage[dvipsnames]{xcolor}
\usepackage[utf8]{inputenc}
\usepackage{verbatim}
\usepackage[mathcal]{euscript}
\usepackage{graphicx}
\usepackage[english]{babel}
\usepackage{multicol}
\usepackage[utf8]{inputenc}  
\usepackage[T1]{fontenc}  
\usepackage{caption}
\usepackage[colorlinks=true ]{hyperref}
\usepackage[final]{pdfpages}
\usepackage{multicol}
\usepackage{float}
\usepackage{tikz}
\usepackage{tikz-3dplot} 
\usepackage{pdfpages}       
\usepackage{varioref} 
\usepackage{float}
 \usepackage{multicol}
\usepackage{multirow}
\usetikzlibrary{quotes,arrows.meta}
\usetikzlibrary{shapes,arrows}
\numberwithin{equation}{section}  
\usepackage{bm}
\DeclareMathOperator{\sech}{sech}
\definecolor{camel}{rgb}{0.76, 0.6, 0.42}
\newtheorem{definition}{Definition}[section]

\newtheorem{theorem}[definition]{Theorem}
\newtheorem{Claim}[definition]{Claim}
\newtheorem{Remark}[definition]{Remark}
\newtheorem{example}[definition]{Example}

\newtheorem{proof}[definition]{Proof}
\newtheorem{Corollary}[definition]{Corollary}

\include{TeXFigs}

\newcommand \bei {\begin{itemize}}
\newcommand \eei {\end{itemize}}
\newcommand \ubar u

\newcommand \del \partial

\newcommand \la \langle
\newcommand \ra \rangle 

\newcommand \auth    \textsc
\newcommand \be {\begin{equation}}
\newcommand \ee {\end{equation}}

\newcommand \bcor {\begin{Corollary}}
\newcommand \ecor {\end{Corollary}}

\newcommand \bpro {\begin{proof}}
\newcommand \epro {\end{proof}}
\newcommand \bdf {\begin{definition}}
\newcommand \edf {\end{definition}}
\newcommand \bex {\begin{example}}
\newcommand \eex {\end{example}}
\newcommand \bcl {\begin{Claim}}
\newcommand \ecl {\end{Claim}}
\newcommand \brm {\begin{Remark}}
\newcommand \erm {\end{Remark}}

\let\oldmarginpar\marginpar
\renewcommand\marginpar[1]{\-\oldmarginpar[\raggedleft\footnotesize #1]%
{\raggedright\footnotesize #1}}

\begin{document}
\begin{frontmatter}
 \title{ A numerical model preserving nontrivial steady-state solutions for predicting waves run-up on coastal areas} 
\author{Hasan Karjoun$^a$, Abdelaziz Beljadid$^{a,b,*}$ }
 \address{$^a$Mohammed VI Polytechnic University, Green City, Morocco}
 \address{$^b$University of Ottawa, Ottawa, Canada}

\begin{abstract}
In this study, a numerical model preserving a class of nontrivial steady-state solutions is proposed to predict waves propagation and waves run-up on coastal zones. The numerical model is based on the Saint-Venant system with source terms due to variable bottom topography and bed friction effects. The resulting nonlinear system is solved using a Godunov-type finite volume method on unstructured triangular grids. A special piecewise linear reconstruction of the solution is implemented with a correction technique to ensure the accuracy of the method and the positivity of the computed water depth. Efficient semi-implicit techniques for the friction terms and a well-balanced formulation for the bottom topography are used to exactly preserve stationary steady-state s solutions. Moreover, we prove that the numerical scheme preserves a class of nontrivial steady-state solutions. To validate the proposed numerical model against experiments, we first demonstrate its ability to preserve nontrivial steady-state solutions and then we model several laboratory experiments for the prediction of waves run-up on sloping beaches. The numerical simulations are in good agreement with laboratory experiments which confirms the robustness and accuracy of the proposed numerical model in predicting waves propagation on coastal areas.
\end{abstract}

\end{frontmatter}

\section{Introduction}
Coastal areas involve several complex natural processes such as surface flows, sediment transport, soil erosion and moving shorelines. The propagation of waves on coastal areas near urban zones can have negative environmental impacts and cause considerable damages \cite{clare2022assessing,satake2005tsunamis,ko2019study}. 
Understanding the dynamics of flow waves and predict its effects on coastal areas is necessary for developing solutions for sustainable water management and reducing water-related hazard including environment risks. 
Coastal wave propagation is mainly affected by the complex geometry of nearshore zones and the bottom topography. The rugged topography can cause many wave transformations such as wave refraction, diffraction, and breaking as the wave approaches the shoreline \cite{afzal2022propagation,guo2012numerical,xie2020two}. Wave run-up at the coastline is mainly depends on the offshore wave conditions such as height, length, and velocity of waves and on the coastal geometry and topography \cite{dodet2018wave}.
Several previous studies have been devoted  to predict waves propagation and determine the height of run-up along the coastlines \cite{von2022overland,varing2021spatial,bellotti2020modal,wu2021evolution,
gedik2006experimental,delis2008robust,liu2011numerical,zhu2017numerical,
dominguez2019towards,farhadi2015numerical,dodd1998numerical}.  
  Synolakis \cite{synolakis1986runup,synolakis1987runup} performed theoretical and experimental study to analyze the evolution of non-breaking and breaking solitary waves and  approximated the wave maximum run-up, as well as the breaking criterion when the wave climbs up the sloping breach.  Madsen and Mei \cite{madsen1969transformation} investigated the transformation of the solitary  wave  over an uneven bottom topography and showed that the wave height rises depending on the slope and its initial height.  Kaplan \cite{kaplan1955generalized} studied the evolution of periodic waves and derived an empirical formula for wave maximum run-up over a sloping beach.

The non-linear shallow water equations are among the commonly used models to describe the propagation of waves over variable bottom topography.  These equations are derived from a depth-averaged integration of the 3D incompressible Navier–Stokes equations under the hydrostatic pressure assumption \cite{de1871theorie}. They are widely used for modeling water flows in  lakes, rivers, and coastal areas \cite{garcia2019shallow,hernandez2016central,beljadid2012numerical,beljadid2013unstructured,delestre2017fullswof,ricchiuto2009stabilized,karjoun2022modelling}.
 Titov and Synolakis \cite{titov1995modeling} tested the efficiency of the shallow water equations for modeling the evolution of breaking and non-breaking solitary waves on sloping beaches based on data from laboratory experiments. Delis et al. \cite{delis2008robust} developed a numerical model based on shallow water equations and investigated its accuracy to simulate long waves.
  Hu et al. \cite{hu2000numerical} used nonlinear shallow water equations to perform numerical simulations of wave overtopping of coastal structures. Brocchini and Dodd \cite{brocchini2008nonlinear} modeled nearshore flows using nonlinear shallow water equations and analyzed the interdependence between physical phenomena, model equations and numerical methods.
  
Various numerical methods have been developed for solving the shallow water equations modeling flows over variable bottom topography  \cite{liang2009least,leveque2002finite,liang2009numerical,george2008augmented,zhou2002numerical,mader2004numerical}. The finite volume technique is one of the most popular tools used for solving the system of shallow water equations due to its capability to conserve mass and momentum, and accurately computes solutions with sharp gradients. Central-upwind finite volume methods are among the most efficient numerical tools developed for solving the system of conservation laws \cite{kurganov2001semidiscrete,kurganov2005central}. The central-upwind schemes are Godunov-type  free Riemann solver, which are based on the information obtained from the local speeds of wave propagation to approximate the numerical fluxes at the cell interfaces. Many extensions of these schemes have been proposed for the system of shallow water flow due to their robustness, high resolution and simplicity \cite{kurganov2002central,beljadid2016well,liu2017coupled,beljadid2017central,hernandez2016central,shirkhani2016well,hanini2021well}. Furthermore, central-upwind schemes can be well-balanced and positivity preserving, that is, they accurately compute stationary steady solutions of the model and maintain the non-negativity of the computed water depth at the discrete level \cite{kurganov2007s,chertock2015well,bryson2011well}.

In this study, the depth-averaged 2-D shallow water equations are used for modeling waves propagation and waves run-up at the coastlines. The source terms due to variable bottom topography and bed friction effect are taken into account. 
The shallow water system is solved using an unstructured finite volume method based on central-upwind techniques \cite{bryson2011well,kurganov2005central}. We used efficient semi-implicit techniques for the friction term and  a well-balanced discretization for the source term due to variable topography. Beside the well-balanced property of the lack at rest, we established that the numerical scheme preserves nontrivial steady-state solutions over a slanted surface. Furthermore, a piecewise linear approximation of the solution is implemented at the discrete level to ensure the accuracy of numerical scheme while guaranteeing the non-negativity of the computed water depth. Numerical simulations are performed to test the efficiency and the capability of the numerical model for predicting waves propagation and waves run-up on coastal areas.

The outline of this paper is as follows. We present the shallow water model with source terms due to variable bottom topography and bed friction effects in Section $\ref{S2}$. The resulting nonlinear system of the shallow water flow model is solved using unstructured finite volume central-upwind technique on triangular meshes in Section $\ref{S3}$. In Section $\ref{S4}$, we perform numerical simulations to validate the proposed numerical model and test its capability for simulating waves propagation on coastlines. The numerical model is validated against laboratory experimental data and we perform numerical simulations to predict waves propagation along a sloping beach with complex bottom topography. Finally, some concluding remarks are provided in Section  $\ref{S5}$.
\section{Governing equations} \label{S2}
\subsection{Shallow water equations}
In this study, we consider the shallow water equations for modeling waves propagation and waves run-up on the coastal areas. Let $h(x,y,t) [m]$ be the water depth above the bottom topography $B(x,y)[m]$, and $\bm{u}=(u(x,y,t),v(x,y,t))^T [m/s]$ be the velocity field of the flow as shown in Figure \ref{Fig1}. In a two-dimensional space, the shallow water equations with source terms due to variable bottom topography and bed friction effects can be written as follows:
\begin{equation}
\left\{\begin{aligned}
&\partial_t h+\partial_x hu+\partial_y hv=0,\\& \partial_t hu+\partial_x\Big(hu^2+\frac{g}{2}h^2\Big)+\partial_y\Big(huv\Big)=-gh\partial_x B-C_f u \sqrt{u^2+v^2},\\& \partial_t hv+\partial_x\Big(huv\Big)+\partial_y \Big(hv^2+\frac{g}{2}h^2\Big)=-gh\partial_y B-C_f v \sqrt{u^2+v^2},
\end{aligned}\right.
\label{Eq1}
\end{equation}
where $t$ represents the time, $x$ and $y$ are the cartesian coordinates and $g$ is the gravity acceleration. The source terms in the momentum equations are the contributions of the bottom topography $B$ and the friction terms $C_f u \sqrt{u^2+v^2}$ and $C_f v \sqrt{u^2+v^2}$, where the bed roughness coefficient $C_f$ is computed using the Manning formula:
\begin{equation}
C_f=\dfrac{g n_f^2 }{h^{1/3}},\\ 
\label{Eq2}
\end{equation}
with $n_f$ being the Manning coefficient which represents the bed hydraulic resistance to flow. The frictional resistance may have higher effects on flow velocity as the water depth decreases since the roughness coefficient $C_f$ increases by decreasing the water depth \cite{brocchini2001efficient}.
\begin{figure}[h!]
\begin{center}
\includegraphics[scale=0.45]{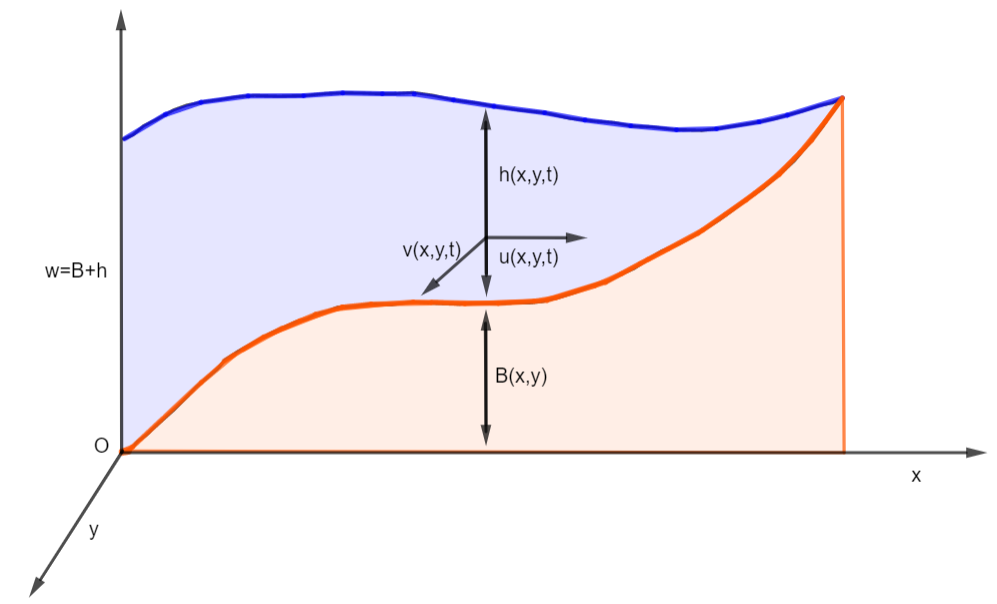}
\caption{Schematic of the shallow water model variables on a coastal area.}
\label{Fig1}
\end{center}
\end{figure}

The system of shallow water equations (\ref{Eq1}) can be expressed in the following matrix form:
\begin{equation}
\partial_t U+\partial_x F(U,B)+\partial_y G(U,B)=S(U,B)+M,
\label{Eq3}
\end{equation}
where we define new conservative variables $U:=[w,q_x  ,q_y]^T$ of the system (\ref{Eq1}) and denote by $w:=B+h$ the free-surface elevation and $q_x:=uh$ and $q_y:=vh$ the water discharges along the coordinate directions $Ox$ and $Oy$, respectively. The vectors of fluxes, bottom topography and friction term $(F,G)^T$, $S$ and $M$ respectively, are defined by:
\begin{equation}
\begin{aligned}
&F:=\left[\begin{aligned}
&q_x\\ 
&\dfrac{q_x^2}{(w-B)}+\dfrac{g}{2 }(w-B)^2\\
&\dfrac{q_x q_y}{(w-B)}
\end{aligned}\right],\, G:=\left[\begin{aligned}
&q_y\\ 
&\dfrac{q_x q_y}{ (w-B)}\\
&\dfrac{q_x^2}{(w-B)}+\dfrac{g}{2}(w-B)^2
\end{aligned}\right],\\
&S:=\left[\begin{aligned}
&0\\ 
&-g(w-B)\partial_x B\\
&-g(w-B)\partial_y B
\end{aligned}\right],\, M:=\left[\begin{aligned}
&0\\ 
&-\dfrac{g n_f^2}{h^{7/3}}q_x\sqrt{q_x^2+q_y^2}\\
&-\dfrac{g n_f^2}{h^{7/3}}q_y\sqrt{q_x^2+q_y^2} 
\end{aligned}\right],
\end{aligned}
\label{Eq4}
\end{equation}
The Jacobian matrix of the shallow water system (\ref{Eq3})-(\ref{Eq4}) has the following eigenvalues:
\begin{equation}
\bm{\lambda}_1=u n_x+v n_y -\sqrt{gh}, \,\quad\quad  \bm{\lambda}_2=u n_x+v n_y, \,\quad\quad  \bm{\lambda}_3=u n_x+v n_y +\sqrt{gh},
\label{lamda}
\end{equation}
which are used in the numerical scheme to compute the one-side local speeds of wave propagation, where $\bm{n}=[n_x,n_y]^T$ is the unit normal vector to the cell interfaces of control volumes used in our numerical methodology in Section \ref{S3}.
\subsection{Stationary steady-state solutions}
The  shallow water equations (\ref{Eq1}) with source terms is a system of balance laws. Under some particular initial conditions, this system has trivial steady-state  solutions of the “lack at rest” ($\bm{u}=0$) in the form:
\begin{equation} 
q_x\equiv q_y\equiv 0, \,\, \text{and}\,\, w=h+B\equiv constant.
\label{Eq5}
\end{equation}
Furthermore, the shallow water system has a nontrivial steady-state  solution which corresponds to the situation of steady flow with constant water depth and non-vanishing velocity over an inclined topography \cite{chertock2015well}:
\begin{equation}
\begin{aligned}
&h=h_0\equiv constant, q_x=q_0\equiv constant, q_y\equiv 0, \partial_x B=-B_0\equiv constant ,\text{and}\, \partial_y B\equiv 0, \, \text{or}\\
&h=h_0\equiv constant, q_x\equiv 0, q_y=q_0\equiv constant, \partial_x B\equiv0 ,\text{and}\, \partial_y B=-B_0\equiv constant. 
\end{aligned}
\label{Eqs6}
\end{equation}
The expression of the water depth can be obtained by replacing the solution (\ref{Eqs6}) in the momentum equations of the shallow water system (\ref{Eq1}).
\begin{equation}
h_0=\left(\dfrac{n_f^2 q_0^2}{B_0} \right)^{3/10}.
\end{equation}
The aforementioned trivial and nontrivial steady-state solutions (\ref{Eq5})-(\ref{Eqs6}) are used to test the proposed numerical model in terms of well-balanced and accuracy.   
\section{Numerical methodology}\label{S3}
\subsection{Numerical  scheme}
In this section, we briefly describe the finite volume central-upwind scheme on unstructured triangular grids applied to solve the system of shallow water equations (\ref{Eq1}) \cite{bryson2011well}. The computational domain is partitioned into  triangular cells $\bm{T}_j$ of size $|\bm{T}_j |$, and centers of mass $C_j:=(\bar{x}_{j},\bar{y}_{j} )$. We denote by $(\partial \bm{T}_j)_{k}$ the common interface of the cell $ \bm{T}_j$  and its neighboring cells $ \bm{T}_{jk}$ of length $d_{jk}$, $ k=1,2,3$. Let  $\bm{n}_{jk}:=(cos(\theta_{jk} ),sin(\theta_{jk} ) )$ be the unit vector normal to $(\partial \bm{T}_j)_{k}$ pointed towards to the cell $\bm{T}_{jk}$, and $(x_{jk},y_{jk} )$ be the coordinates of the midpoint $N_{jk}$ of the cell interface $(\partial \bm{T}_j)_{k}$ with corresponding vertices $N_{jk_i}$ of coordinates $(x_{jk_i},y_{jk_i})$, $i=1,2$ as shown in Figure \ref{fig01}.
\begin{figure}[!ht]
\centering{}
\begin{tikzpicture}[y=1.3cm, x=1.3cm,font=\scriptsize]
\draw  (3.6,2.9) -- (6.6,2.7) ;
\draw  (3.6,2.9) -- (5,5) ;
\draw  (6.6,2.7)  -- (5,5) ;
\draw [thick,dashed] (2,4.7)  -- (5,5) ;
\draw  [thick,dashed](2,4.7)  -- (3.6,2.9);
\draw  [thick,dashed](8.5,4.5)  -- (5,5);
\draw  [thick,dashed](8.5,4.5)  --  (6.6,2.7) ;
\draw [thick,dashed] (4.7,1)  --  (6.6,2.7) ;
\draw [thick,dashed] (4.7,1)  --   (3.6,2.9);
\definecolor{blue}{rgb}{0,0.6,0.4}  
\fill [blue, opacity=0.2]  (3.6,2.9) -- (6.6,2.7) -- (5,5)--cycle ;
  \draw  (4.3,4.09) node[anchor=north] {$\bullet$ };
  \draw [.-stealth](5.8,3.8) --(6.2,4.1);
  \draw  (5.83,3.95) node[anchor=north] {$\bullet$ };
    \draw  (5.1,2.93) node[anchor=north] {$\bullet$ };
\draw (4.9,1.8) node[anchor=north] { };
\draw (7.5,4.5) node[anchor=north] {$\bm{T}_{j_k}$ };
\draw (3,4.5) node[anchor=north] {};
\draw  (5,3.8) node[anchor=north] {$\bullet$ };
\draw (4.8,3.9) node[anchor=north] {$C_j$ };
\draw (6.6,4.2) node[anchor=north] {$C_{jk}$ };
\draw  (6.88,4.1) node[anchor=north] {$\bullet$ };
\draw (5,4.5) node[anchor=north] {$\bm{T}_{j}$ };
\draw (5.6,3.95) node[anchor=north] {$N_{j_k}$ };
\draw (5.8,4.6) node[anchor=north] {$(\partial \bm{T}_j)_{k}$ };
\draw (5.1,3.1) node[anchor=north] { };
\draw (4.5,4) node[anchor=north] {};
\draw (3.8 ,4.5) node[anchor=north] {};
\draw (5.09,2.3) node[anchor=north] { };
\draw (6.2,4.) node[anchor=north] {$\bm{n}_{j_k}$ };
\draw (7,2.9) node[anchor=north] {$N_{{jk}_{1}}$ };
\draw (3.3,3) node[anchor=north] {};
\draw (5,5.5) node[anchor=north] {$N_{{jk}_{2}}$};
\end{tikzpicture}
\caption{Schematic of triangular cells}
\label{fig01}
\end{figure}
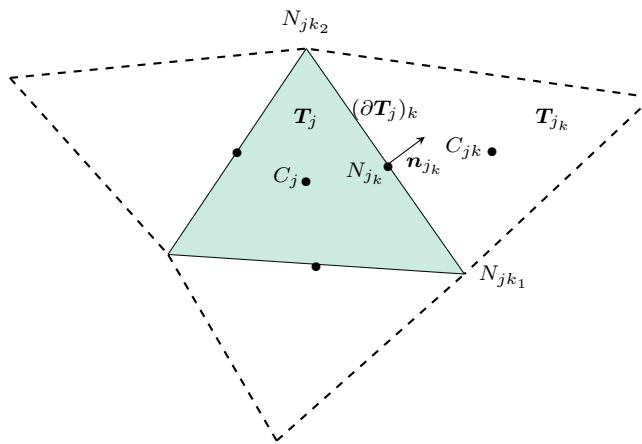

We denote by $\bm{\bar{U}}_j$ the computed solution in the finite volume framework, and $\bm{\bar{S}}_j$ and $\bm{\bar{M}}_j$ the average values of the source terms over the cell $\bm{T_j}$, which are defined as follows:
\begin{equation}
\bar{\bm{U}}_j(t)=\dfrac{1}{|\bm{T}_j|}  \int_{\bm{T}_j} {U}(x,y,t)dxdy,\, \,  \bar{\bm{S}}_j(t)=\dfrac{1}{|\bm{T}_j|}  \int_{\bm{T}_j} {S}(x,y,t)dxdy,\, \, \bar{\bm{M}}_j(t)=\dfrac{1}{|\bm{T}_j|}  \int_{\bm{T}_j} {M}(x,y,t)dxdy,
\end{equation}
The central-upwind scheme on triangular meshes is applied to the system of shallow water equations (\ref{Eq1})-(\ref{Eq4}), which is given by the following semi-discrete equation \cite{bryson2011well}:

\begin{equation}
\begin{aligned}
\dfrac{\partial \bar{\bm{U}}_j }{\partial t}=&-\dfrac{1}{|\bm{T}_j|}\sum_{k=1}^3\dfrac{d_{jk}\cos(\theta_{jk})}{b_{jk}^{in}+b_{jk}^{out}}\Big[ b_{jk}^{in}F\left(U_{jk}(N_{jk}),B_{jk}\right)  +b_{jk}^{out}F\left(U_{j}(N_{jk}),B_{jk}\right)  \Big]\\
&-\dfrac{1}{|\bm{T}_j|}\sum_{k=1}^3\dfrac{d_{jk}\sin(\theta_{jk})}{b_{jk}^{in}+b_{jk}^{out}}\Big[ b_{jk}^{in}G\left(U_{jk}(N_{jk}),B_{jk}\right)  +b_{jk}^{out}G\left(U_{j}(N_{jk}),B_{jk}\right)  \Big]\\
&+\dfrac{1}{|\bm{T}_j|}\sum_{k=1}^3\dfrac{d_{jk}b_{jk}^{in}b_{jk}^{out}}{b_{jk}^{in}+b_{jk}^{out}}\Big[ U_{jk}(N_{jk})-U_{j}(N_{jk})  \Big]+\bm{\bar{S}}_j+\bm{\bar{M}}_j.
\end{aligned}
\label{Scheme}
\end{equation}
In the numerical scheme (\ref{Scheme}), $U_ j(N_{jk})$ and $U_{jk}(N_{jk})$ are respectively the reconstructed values of the solution at the left and right of cell interfaces $(\partial \bm{T}_j)_{k}$, computed at the midpoint $N_{jk}$ and $B_{jk}$ are the corresponding reconstructed values of the bottom topography $B$. The values $b_{jk}^{in}$ and $b_{jk}^{out}$ are the one side-local speeds of wave propagation computed using the small and large eigenvalues ${\bm{\lambda}}_1$ and $\bm{\lambda}_3$ of the Jacobian matrix of the shallow water system,  given in Eq. (\ref{lamda}) \cite{bryson2011well,kurganov2005central},
\begin{equation}
\begin{aligned}
&b_{jk}^{\rm in}=-\min\big\{\bm{\lambda}_j^{\theta}-\sqrt{gh_j(N_{jk})},\bm{\lambda}_{jk}^{\theta}-\sqrt{gh_{jk}(N_{jk})},0\big\},\\
&b_{jk}^{\rm out}=\max\big\{\bm{\lambda}_j^{\theta}+\sqrt{gh_j(N_{jk})},\bm{\lambda}_{jk}^{\theta}+\sqrt{gh_{jk}(N_{jk})},0\big\},\\
\end{aligned}
\label{Eq6}
\end{equation}
where
\begin{equation}
\begin{aligned}
\bm{\lambda}^\theta_{jk}&:=\cos(\theta_{jk})u_{jk}(N_{jk})+\sin(\theta_{jk})v_{jk}(N_{jk}),\\
\bm{\lambda}^\theta_j&:=\cos(\theta_{jk})u_j(N_{jk})+\sin(\theta_{jk})v_j(N_{jk}).
\end{aligned}
\label{Eq7}
\end{equation}
The reconstructed values of the velocity at the cell interfaces, $\bm{u}_{j}$ and $\bm{u}_{jk}$ are computed using the following  desingularization  formula to prevent division by small computed values of water depth, where we omit the indices $j$ and $jk$  \cite{karjoun2022structure}
\begin{equation}
\bm{u}=\left\{\begin{aligned}
&\dfrac{1}{h} \times\left[\begin{array}{c}
q_x\\ 
q_y
\end{array} \right],\,\,\quad\quad\quad\quad\quad\quad \quad\quad\text{if}\, \, h\geq\eta , \\\
& \dfrac{\sqrt{2} h}{\sqrt{h^4+\max\lbrace h^4, \epsilon\rbrace}}\times\left[\begin{array}{c}
q_x\\ 
q_y
\end{array} \right],\quad\text{if} \, \, h< \eta,
\end{aligned}\right.
\label{EqU}
\end{equation}
where the prescribed tolerances  $\epsilon =\max_j\lbrace |\bm{T}_j |^2\rbrace$ and $\eta=10^{-6}$ are used in our numerical simulations.
\subsection{ Linear reconstructions }
In this section, we define the reconstructions of the bottom topography $B$ and the conservative variables $U$ of the system.  The bottom topography is known at the cell vertices $N_{jk_i}$, since we use a linear reconstruction over the cell $\bm{T}_j$ its values at the midpoint $N_{jk}$ are obtained using the following linear approximation \cite{bryson2011well,beljadid2016well}:



\begin{equation}
B_{jk}=\dfrac{B_{jk_1}+B_{jk_2}}{2}.
\end{equation}
Similarly, the values of  the topography $\bm{B}_j$  at the center of mass $C_j$  are  computed as follows:
\begin{equation}
\bm{B}_j=\dfrac{1}{|\bm{T}_j|}\int_{\bm{T}_j}B(x,y)\,dxdy =\dfrac{B_{j1}+B_{j2}+B_{j3}}{3}.
\end{equation}
To design the second-order numerical model, the  values of vector variables at the midpoints of cell interfaces, ${U}_j(N_{jk})$ and ${U}_{jk}(N_{jk})$  are computed using the following piecewise linear reconstruction:
\begin{equation}
{U}_j(x,y):=\bar{\bm{U}}_j+(\partial \bm{U}_j)_x(x-\bar{x}_j)+(\partial \bm{U}_j)_y(y-\bar{y}_j),
\quad\label{Eq8}
\end{equation}
where $(\partial \bm{U}_j)_x$ and $(\partial \bm{U}_j)_y$ are the two components of the numerical gradient of the vector $U$ at the cell $\bm{T}_j$, computed following the approach developed in \cite{jawahar2000high}:
\begin{equation}
\begin{aligned}
(\partial\bm{U}_j^{(i)})_x=&\mathcal{W}_{j1}^{(i)}(\partial \mathcal{U}_{j1}^{(i)})_x+\mathcal{W}_{j2}^{(i)}(\partial \mathcal{U}_{j2}^{(i)})_x+\mathcal{W}_{j3}^{(i)}(\partial \mathcal{U}_{j3}^{(i)})_x,\\
(\partial \bm{U}_j^{(i)})_y=&\mathcal{W}_{j1}^{(i)}(\partial\mathcal{U}_{j1}^{(i)})_y+\mathcal{W}_{j2}^{(i)}(\partial \mathcal{U}_{j2}^{(i)})_y+\mathcal{W}_{j3}^{(i)}(\partial \mathcal{U}_{j3}^{(i)})_y,\, \text{for} \, i=1,\, 2,\,3, 
\end{aligned}
\label{Grd1}
\end{equation}
where the initial gradients $(\partial \mathcal{U}_j^{(i)})_x$ and $(\partial \mathcal{U}_j^{(i)})_y$ at the cell $\bm{T_j}$ are computed by
constructing the plane passing through the centers of mass of  three neighboring cells, $(\bar{x}_{j1},\bar{y}_{j1}, \bm{\bar{U}}_{j1}^{(i)})$, $(\bar{x}_{j2},\bar{y}_{j2}, \bm{\bar{U}}_{j2}^{(i)})$, and $(\bar{x}_{j3},\bar{y}_{j3}, \bm{\bar{U}}_{j3}^{(i)})$ for i$th$ component of the vector $U$, $i=1,\,2,\,3$.

\begin{equation}
\begin{aligned}
(\mathcal{\partial \bm{U}}_j^{(i)})_x=&\dfrac{(\bar{y}_{j3}-\bar{y}_{j1})(\bm{\bar{U}}_{j2}^{(i)}-\bm{\bar{U}}_{j1}^{(i)})-(\bar{y}_{j2}-\bar{y}_{j1})(\bm{\bar{U}}_{j3}^{(i)}-\bm{\bar{U}}_{j1}^{(i)}) }{(\bar{y}_{j3}-\bar{y}_{j1})(\bar{x}_{j2}-\bar{x}_{j1})-(\bar{y}_{j2}-\bar{y}_{j1})(x_{j3}-x_{j1})},\\
(\mathcal{\partial \bm{U}}_j^{(i)})_y=&\dfrac{(\bar{x}_{j2}-\bar{x}_{j1})(\bm{\bar{U}}_{j3}^{(i)}-\bm{\bar{U}}_{j1}^{(i)})-(\bar{x}_{j3}-\bar{x}_{j1})(\bm{\bar{U}}_{j2}^{(i)}-\bm{\bar{U}}_{j1}^{(i)}) }{(\bar{x}_{j2}-\bar{x}_{j1})(\bar{y}_{j3}-\bar{y}_{j1})-(\bar{x}_{j3}-\bar{x}_{j1})(\bar{y}_{j2}-\bar{y}_{j1})}.
\end{aligned}
\label{Grd}
\end{equation}
The weights  $\mathcal{W}_{j1}^{(i)}$, $\mathcal{W}_{j2}^{(i)}$, and $\mathcal{W}_{j3}^{(i)}$ are given by \cite{jawahar2000high}:
\begin{equation}
\begin{aligned}
\mathcal{W}_{j1}^{(i)}=&\dfrac{|| \partial \mathcal{U}_{j3}^{(i)} ||_2^2 || \partial \mathcal{U}_{j2}^{(i)} ||_2^2+\xi^2}{|| \partial \mathcal{U}_{j1}^{(i)} ||_2^4+|| \partial \mathcal{U}_{j2}^{(i)} ||_2^4+|| \partial \mathcal{U}_{j3}^{(i)} ||_2^4+3\xi^2},\\
\mathcal{W}_{j2}^{(i)}=&\dfrac{|| \partial \mathcal{U}_{j3}^{(i)} ||_2^2 || \partial \mathcal{U}_{j1}^{(i)} ||_2^2+\xi^2}{|| \partial\mathcal{U}_{j1}^{(i)} ||_2^4+|| \partial \mathcal{U}_{j2}^{(i)} ||_2^4+|| \partial \mathcal{U}_{j3}^{(i)} ||_2^4+3\xi^2},\\
\mathcal{W}_{j3}^{(i)}=&\dfrac{|| \partial \mathcal{U}_{j1}^{(i)} ||_2^2 ||  \mathcal{U}_{j2}^{(i)} ||_2^2+\xi^2}{|| \partial \mathcal{U}_{j1}^{(i)} ||_2^4+|| \partial \mathcal{U}_{j2}^{(i)} ||_2^4+|| \partial \mathcal{U}_{j3}^{(i)} ||_2^4+3\xi^2},
\end{aligned}
\label{Grd3}
\end{equation}
where $\xi=10^{-7}$ is considered to avoid division by  zero.

\subsection{Positivity reconstruction of the water depth}
In this section, we will correct the linear reconstruction ($\ref{Eq8}$) for the free-surface elevation variable $w$ to ensure the non-negativity of the reconstructed values $h_j$ and $h_{jk}$ of the water depth.
Indeed, the correction technique is applied for triangles where we have $w_j(x_{jk_1},y_{jk_1})<B_{jk_1}$ in some vertices. 
Following the methodology developed by Bryson et al.  \cite{bryson2011well} there are two cases where the correction technique is needed. The first case is where we have  $w_j(x_{jk_1},y_{jk_1})<B_{j_{k1}}$ and $w_j(x_{jk_2},y_{jk_2})<B_{j_{k2}}$ at the cell interface $\partial (\bm{T}_j)_k$ and the second case is where only one vertex for which $w_j(x_{jk_1},y_{jk_1})<B_{j_{k1}}$. \\
We set $w_{jk_1}^{corr}=B_{jk_1}$ at the vertices for which $w_j(x_{jk_1},y_{jk_1})<B_{j_{k1}}$ and for the rest of the vertices we use $w_{jk_1}^{corr}=\tilde{h}_j+B_{jk_1}$, where $\tilde{h}_j$ is computed using the conservation requirement:\\
\begin{equation}
\tilde{h}_j=\dfrac{3\bm{\bar{h}}_j}{\sum_{k=1}^3 {\beta_k}},\, \text{with}\, \, \, \, \beta_k=\left\{ 
\begin{array}{lll}
1, & \text{if} \,\, w_{jk_1}\geq B_{jk_1}, \\ 
0,&  \text{if} \, \,w_{jk_1}<B_{jk_1},\\

\end{array}%
\right.
\end{equation}

\begin{figure}[!ht]
\begin{center}
\includegraphics[scale=0.45]{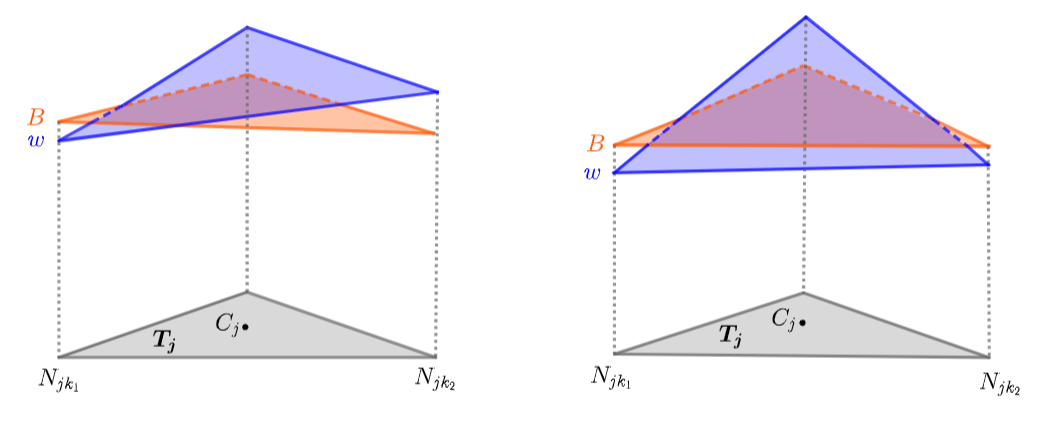}
\caption{Schematic of the two cases where a correction procedure is applied: case 1 (right) and case 2 (left).}
\label{Fig12}
\end{center}
\end{figure} 

The corrected reconstruction $h_j^{corr}$ for the water depth is conservative over the cell $\bm{T_j}$.
Indeed, denote by $\mathcal{V}^{-}$ and $\mathcal{V}^{+}$ the sets of vertices $N_{jk_1}$ for which $w_j(x_{jk_1},y_{jk_1})<B_{j_{k1}}$ and $w_j(x_{jk_1},y_{jk_1})\geq B_{j_{k1}}$ of cardinals $n^-$ and $n^+$, respectively. Using the linear approximation we have:
\begin{equation}
\begin{aligned}
\bm{\bar{h}^{corr}}_j=\dfrac{1}{|\bm{T_j}|}\int_{\bm{T_j}} h_j^{corr}(x,y)\,dxdy=&\dfrac{1}{|\bm{T_j}|}\int_{\bm{T_j}}\left[ w_j^{corr}(x,y) -B(x,y)\right]\,dxdy,\\
=&\dfrac{1}{3}\sum_{k=1}^{3}\left[ w_{jk_1}^{corr} -B_{jk_1}\right],\\
\end{aligned}
\end{equation}
where we use the values at the cell vertices $N_{jk_1}$
\begin{equation}
\begin{aligned}
\bm{\bar{h}^{corr}}_j
=&\dfrac{1}{3}\sum_{N_{jk_1}\in \mathcal{V}^-}\left[ w_{jk_1}^{corr} -B_{jk_1}\right]+\dfrac{1}{3}\sum_{N_{jk_1}\in\mathcal{V}^+}\left[  w_{jk_1}^{corr} -B_{jk_1}\right],\\
=&\dfrac{1}{3}\sum_{N_{jk_1}\in \mathcal{V}^-}h_{jk_1}^{corr}+\dfrac{1}{3}\sum_{N_{jk_1}\in \mathcal{V}^+}h_{jk_1}^{corr}.
\end{aligned}
\end{equation}
Since $h_{jk_1}^{corr}=0$ in $\mathcal{V}^-$ and $h_{jk_1}^{corr}=\tilde{h}_j$ in $\mathcal{V}^+$, then
\begin{equation}
\begin{aligned}
\bm{\bar{h}^{corr}}_j=&\dfrac{1}{3}\sum_{N_{jk_1}\in \mathcal{V}^+}h_{jk_1}^{corr}=\dfrac{1}{3}{n^+}\tilde{h}_j=\dfrac{1}{3}{n^+}\dfrac{3\bm{\bar{h}}_j}{\sum_{k=1}^3 {\beta_k}}=\bm{\bar{h}}_j.
\end{aligned}
\end{equation}
\subsection{Discretization of the source terms}
A suitable discretization of the bottom topography source term should be used to satisfy the well-balanced property of the numerical model in which the steady-state  solutions (\ref{Eq5})-(\ref{Eqs6}) are preserved. In our study, we use the approach developed in \cite{bryson2011well} where the discretization of the topography source term balances the numerical fluxes.  The well-balanced quadrature for the non-vanishing components of $\bm{\bar{S}}_j$ are: 
\begin{equation}
\begin{aligned}
\bar{\bm{S}}_j^{(2)}=&\frac{1}{2|\bm{T}_j|}\sum_{k=1}^{3}g d_{jk}[B_{jk}-w_j(N_{jk})]^2\cos(\theta_{jk})+g(\partial  \bm{{w}}_j)_x[\bm{B}_j-\bm{\bar{w}}_j],&\\
\bar{\bm{S}}_j^{(3)}=&\frac{1}{2|\bm{T}_j|}\sum_{k=1}^{3}g d_{jk}[B_{jk}-w_j(N_{jk})]^2\sin(\theta_{jk})+g(\partial \bm{{w}}_j)_y[\bm{B}_j-\bm{\bar{w}}_j].&
\end{aligned}
\label{Eq11}
\end{equation}

The bed friction term is discretized using the following semi-implicit scheme proposed in \cite{karjoun2022structure}:
\begin{equation}
 \bar{\bm{M}}=\dfrac{1}{|\bm{T}_j|}  \int_{\bm{T}_j} \bm{M}(x,y,t)dxdy=-\dfrac{gn_f^2}{2}\left(\dfrac{\sqrt{\bm{\bar{q}}_{xj}^2+\bm{\bar{q}}_{yj}^2}}{\bm{\bar{h}}_j^{7/3}}\right)^n\left[ \begin{array}{c}
 0 \\ 
 \bm{\bar{q}}_{xj}^{n}+\bm{\bar{q}}_{xj}^{n+1} \\ 
  \bm{\bar{q}}_{yj}^{n}+\bm{\bar{q}}_{yj}^{n+1} 
 \end{array} \right].
 \label{Fri}
\end{equation}
We use the following desingularization technique to prevent the division by small computed values of water depth: 
\begin{equation}
 \bar{\bm{M}}=-\dfrac{gn_f^2}{2}\left(\dfrac{\bm{\bar{h}}_j^{5/3}\sqrt{\bm{\bar{q}}_{xj}^2+\bm{\bar{q}}_{yj}^2}}{\bm{\bar{h}}_j^{4}+\max\lbrace\bm{\bar{h}}_j^{4}, \epsilon\rbrace}\right)^n\left[\begin{array}{c}
 0 \\ 
 \bm{\bar{q}}_{xj}^{n}+\bm{\bar{q}}_{xj}^{n+1} \\ 
  \bm{\bar{q}}_{yj}^{n}+\bm{\bar{q}}_{yj}^{n+1} 
 \end{array} \right],
 \label{Fric}
\end{equation}
where  we use the same value of $\epsilon$ as in Eq. (\ref{EqU}). The Euler temporal method is applied to discretize  the semi-discrete equation of the central-upwind scheme (\ref{Scheme}), which yields to the following explicit expression for the discharge components $\bm{\bar{q}}_{xj}^{n+1}$ and $\bm{\bar{q}}_{yj}^{n+1}$:
\begin{equation}
\begin{aligned}
\bm{\bar{q}}_{xj}^{n+1}=&\dfrac{\bm{\bar{q}}_{xj}^{n}(1-\Delta t \Phi_j)+\Delta t\Pi_j^{q_x} +\Delta t\bar{\bm{S}}_j^{(2)}}{1+\Delta t \Phi_j},&\\
\bm{\bar{q}}_{yj}^{n+1}=&\dfrac{\bm{\bar{q}}_{yj}^{n}(1-\Delta t \Phi_j)+\Delta t\Pi_j^{q_y} +\Delta t\bar{\bm{S}}_j^{(3)}}{1+\Delta t \Phi_j},&
\end{aligned}
\end{equation}
with, $$\Phi_j=\dfrac{gn_f^2}{2}\left(\dfrac{\bm{\bar{h}}_j^{5/3}\sqrt{\bm{\bar{q}}_{xj}^2+\bm{\bar{q}}_{yj}^2}}{\bm{\bar{h}}_j^{4}+\max\lbrace\bm{\bar{h}}_j^{4}, \epsilon\rbrace}\right)^n,$$ and $\Pi_j^{q_x}=\frac{1}{|\bm{T}_j|}\sum_{k=1}^{3}d_{jk}\bm{A}_{jk}^{(2)}$ and $\Pi_j^{q_y}=\frac{1}{|\bm{T}_j|}\sum_{k=1}^{3}d_{jk} \bm{A}_{jk}^{(3)}$ with $\bm{A}_{jk}$ are the central-upwind fluxes:
\begin{equation}
\begin{aligned}
\bm{A}_{jk}=&-\dfrac{\cos(\theta_{jk})}{b_{jk}^{in}+b_{jk}^{out}}\Big[ b_{jk}^{in}F\left(U_{jk}(N_{jk}),B_{jk}\right)  +b_{jk}^{out}F\left(U_{j}(N_{jk}),B_{jk}\right)  \Big]\\
&-\dfrac{\sin(\theta_{jk})}{b_{jk}^{in}+b_{jk}^{out}}\Big[ b_{jk}^{in}G\left(U_{jk}(N_{jk}),B_{jk}\right)  +b_{jk}^{out}G\left(U_{j}(N_{jk}),B_{jk}\right)  \Big]\\
&+\dfrac{b_{jk}^{in}b_{jk}^{out}}{b_{jk}^{in}+b_{jk}^{out}}\Big[ U_{jk}(N_{jk})-U_{j}(N_{jk})  \Big].
\end{aligned}
\end{equation}
\subsection{Nontrivial well-balanced property}
In this section, we prove the nontrivial well-balanced property of the central-upwind scheme (\ref{Scheme}) through the following theorem.
\begin{theorem}
The numerical scheme (\ref{Scheme})-(\ref{Fri}) and the forward Euler time discretization for solving the shallow water equations (\ref{Eq1})-(\ref{Eq4}) preserves the nontrivial steady-state solutions (\ref{Eqs6}).
\end{theorem}
\begin{proof} We consider a rectangular domain and an inclined topography in the $x-$direction with $(\partial\bm{B}_j)_x=-B_0\, \, \text{and} \,\, (\partial\bm{B}_j)_y= 0$, where $B_0$ is a constant. We assume that at time $t^n$:
\begin{equation}
\bm{\bar{h}}_{j}^n=h_0=\left(\dfrac{n_f^2 q_0^2}{B_0} \right)^{3/10}, \,\bm{\bar{q}}_{xj}^{n}=q_0,\,\bm{\bar{q}}_{yj}^{n}=0,
\label{Cnd}
\end{equation}
and the following boundary conditions: $u_{jk}=u_{j}$,  $v_{jk}=v_{j}$ and $h_{jk}=h_{j}$ in $x-$ direction and $u_{jk}=u_{j}$, $v_{jk}=0$ and $h_{jk}=h_{j}$ in $y-$direction.\\
We will prove that $\bm{\bar{h}}_{j}^{n+1}=h_0$, $\bm{\bar{q}}_{xj}^{n+1}=q_0$ and $\bm{\bar{q}}_{yj}^{n+1}=0$ at time $t^{n+1}$.  By applying the Euler time discretization to the semi-discrete scheme (\ref{Scheme}), we obtain:

\begin{equation}
\begin{aligned}
 \bar{\bm{U}}_j^{n+1} = \bar{\bm{U}}_j^n&-\dfrac{\Delta t}{|\bm{T}_j|}\sum_{k=1}^3\dfrac{d_{jk}\cos(\theta_{jk})}{b_{jk}^{in}+b_{jk}^{out}}\Big[ b_{jk}^{in}F\left(U_{jk}(N_{jk}),B_{jk}\right)  +b_{jk}^{out}F\left(U_{j}(N_{jk}),B_{jk}\right)  \Big]\\
&-\dfrac{\Delta t}{|\bm{T}_j|}\sum_{k=1}^3\dfrac{d_{jk}\sin(\theta_{jk})}{b_{jk}^{in}+b_{jk}^{out}}\Big[ b_{jk}^{in}G\left(U_{jk}(N_{jk}),B_{jk}\right)  +b_{jk}^{out}G\left(U_{j}(N_{jk}),B_{jk}\right)  \Big]\\
&+\dfrac{\Delta t}{|\bm{T}_j|}\sum_{k=1}^3\dfrac{d_{jk}b_{jk}^{in}b_{jk}^{out}}{b_{jk}^{in}+b_{jk}^{out}}\Big[ U_{jk}(N_{jk})-U_{j}(N_{jk})  \Big]+\Delta t\bm{\bar{S}}_j+\Delta t\bm{\bar{M}}_j.
\end{aligned}
\label{Steady}
\end{equation} 
 
The proof is presented in three steps as follows:\\
\emph{-Step 1: Constant water depth}. Under the conditions (\ref{Cnd}) the reconstruction (\ref{Eq8}) leads to $h_j=h_{jk}=h_0$, $q_{xj}=q_{xjk}=q_0$ and  $q_{yj}=q_{yjk}=0$ at cell interfaces. Then, the scheme (\ref{Steady}) for the continuity equation yields to:
\begin{equation}
\begin{aligned}
 \bar{\bm{w}}_j^{n+1} = \bar{\bm{w}}_j^n&-\dfrac{\Delta t}{|\bm{T}_j|}\sum_{k=1}^3\dfrac{d_{jk}\cos(\theta_{jk})}{b_{jk}^{in}+b_{jk}^{out}}\Big[ b_{jk}^{in}q_0  +b_{jk}^{out}q_0   \Big]\\
&+\dfrac{\Delta t}{|\bm{T}_j|}\sum_{k=1}^3\dfrac{d_{jk}b_{jk}^{in}b_{jk}^{out}}{b_{jk}^{in}+b_{jk}^{out}}\Big[ w_{jk}(N_{jk})-w_{j}(N_{jk})  \Big].
\end{aligned}
\label{sth}
\end{equation} 
Since ${w}={h}+{B}$ and the reconstruction of the bottom topography $B$ is continuous, then $w_{jk}(N_{jk})-w_j(N_{jk})=h_{jk}(N_{jk})-h_j(N_{jk})=0$ and Eq. (\ref{sth}) reduces to: 
\begin{equation}
\begin{aligned}
 \bar{\bm{h}}_j^{n+1} =&\bar{\bm{h}}_j^n-\dfrac{\Delta t}{|\bm{T}_j|}\sum_{k=1}^3\dfrac{d_{jk}\cos(\theta_{jk})}{b_{jk}^{in}+b_{jk}^{out}}\Big[ b_{jk}^{in}+b_{jk}^{out}\Big] q_0,\\
=& h_0-\dfrac{\Delta t}{|\bm{T}_j|}\sum_{k=1}^3 d_{jk}\cos(\theta_{jk})q_0.
\end{aligned}
\label{sth2}
\end{equation} 
Thanks to $\sum_{k=1}^3 d_{jk}\cos(\theta_{jk})=0$, we have $\bar{\bm{h}}_j^{n+1}=h_0$ which shows that the computed water depth remains constant over the domain at time $t^{n+1}$.\\
\emph{-Step 2: Constant water discharge $q_x$}. Similarly to the water depth, by incorporating the conditions (\ref{Cnd}) in the numerical  scheme (\ref{Steady}) we obtain the following expression for the discharge variable $q_x$:
\begin{equation}
\begin{aligned}
 \bar{\bm{q}}_{xj}^{n+1} =& q_0-\dfrac{\Delta t}{|\bm{T}_j|}\sum_{k=1}^3\dfrac{d_{jk}\cos(\theta_{jk})}{b_{jk}^{in}+b_{jk}^{out}}\Big[ b_{jk}^{in}\left(\frac{q_0^2}{h_0}+\frac{1}{2}gh_0^2\right)+b_{jk}^{out}\left(\frac{q_0^2}{h_0}+\frac{1}{2}gh_0^2\right) \Big]\\
&+\frac{\Delta t g}{2|\bm{T}_j|}\sum_{k=1}^{3}d_{jk}\cos(\theta_{jk})h_0^2-g\Delta t h_0(\partial  \bm{{w}}_j)_x \\
&-\dfrac{\Delta t g n_f^2 q_0}{2h_0^{7/3}}[q_0+\bar{\bm{q}}_{xj}^{n+1}],\\ 
 =& q_0-\dfrac{\Delta t}{|\bm{T}_j|}\sum_{k=1}^3d_{jk}\cos(\theta_{jk})\left(\frac{q_0^2}{h_0}+\frac{1}{2}gh_0^2\right)  \\
&+\frac{\Delta t g}{2|\bm{T}_j|}\sum_{k=1}^{3}d_{jk}\cos(\theta_{jk})h_0^2-g\Delta t h_0(\partial  \bm{{w}}_j)_x \\
&-\dfrac{\Delta t g n_f^2 q_0}{2h_0^{7/3}}[q_0+\bar{\bm{q}}_{xj}^{n+1}].
\end{aligned}
\label{Stqx}
\end{equation} 
Due to $\sum_{k=1}^3 d_{jk}\cos(\theta_{jk})=0$ at each cell, we obtain:
\begin{equation}
\begin{aligned}
 \bar{\bm{q}}_{xj}^{n+1} = q_0-g\Delta t h_0(\partial  \bm{{w}}_j)_x-\dfrac{\Delta t g n_f^2 q_0}{2h_0^{7/3}}[q_0+\bar{\bm{q}}_{xj}^{n+1}].
\end{aligned}
\label{Stqx1}
\end{equation}
From the definition of the numerical gradient in Eqs. (\ref{Grd1})-(\ref{Grd3}), we have $(\partial  \bm{{w}}_j)_x=(\partial  \bm{{B}}_j)_x=-B_0$. Indeed, recall that the topography is known at the cell vertices and the reconstruction of the topography is linear over each computational cell. Furthermore, in our case the topography is linear over the entire domain, then the following Taylor expansions are exact:
\begin{equation}
\begin{aligned}
\bm{B}_{j1}=&\bm{B}_j+(\partial  \bm{{B}}_j)_x(\bar{x}_{j1}-\bar{x}_{j})+(\partial  \bm{{B}}_j)_y(\bar{y}_{j1}-\bar{y}_{j}),\\
\bm{B}_{j2}=&\bm{B}_j+(\partial  \bm{{B}}_j)_x(\bar{x}_{j2}-\bar{x}_{j})+(\partial  \bm{{B}}_j)_y(\bar{y}_{j2}-\bar{y}_{j}),\\
\bm{B}_{j3}=&\bm{B}_j+(\partial  \bm{{B}}_j)_x(\bar{x}_{j3}-\bar{x}_{j})+(\partial  \bm{{B}}_j)_y(\bar{y}_{j3}-\bar{y}_{j}).
\end{aligned}
\label{BB}
\end{equation}
Since $(\partial  \bm{{B}}_j)_y=0$, then 
\begin{equation}
\begin{aligned}
\bm{B}_{j1}=&\bm{B}_j+(\partial  \bm{{B}}_j)_x(\bar{x}_{j1}-\bar{x}_{j}),\\
\bm{B}_{j2}=&\bm{B}_j+(\partial  \bm{{B}}_j)_x(\bar{x}_{j2}-\bar{x}_{j}),\\
\bm{B}_{j3}=&\bm{B}_j+(\partial  \bm{{B}}_j)_x(\bar{x}_{j3}-\bar{x}_{j}).
\end{aligned}
\end{equation}  
According to Eq. (\ref{Grd}), we have:
\begin{equation}
\begin{aligned}
(\mathcal{\partial \bm{W}}_j)_x=&\dfrac{(\bar{y}_{j3}-\bar{y}_{j1})(\bm{\bar{w}}_{j2}-\bm{\bar{w}}_{j1})-(\bar{y}_{j2}-\bar{y}_{j1})(\bm{\bar{w}}_{j3}-\bm{\bar{w}}_{j1}) }{(\bar{y}_{j3}-\bar{y}_{j1})(\bar{x}_{j2}-\bar{x}_{j1})-(\bar{y}_{j2}-\bar{y}_{j1})(x_{j3}-x_{j1})},\\
=&\dfrac{(\bar{y}_{j3}-\bar{y}_{j1})(\bm{{B}}_{j2}-\bm{{B}}_{j1})-(\bar{y}_{j2}-\bar{y}_{j1})(\bm{{B}}_{j3}-\bm{{B}}_{j1}) }{(\bar{y}_{j3}-\bar{y}_{j1})(\bar{x}_{j2}-\bar{x}_{j1})-(\bar{y}_{j2}-\bar{y}_{j1})(x_{j3}-x_{j1})},\\
=&\dfrac{(\bar{y}_{j3}-\bar{y}_{j1})(\bar{x}_{j2}-\bar{x}_{j1})(\partial  \bm{{B}}_j)_x-(\bar{y}_{j2}-\bar{y}_{j1})(\bar{x}_{j3}-\bar{x}_{j1})(\partial  \bm{{B}}_j)_x }{(\bar{y}_{j3}-\bar{y}_{j1})(\bar{x}_{j2}-\bar{x}_{j1})-(\bar{y}_{j2}-\bar{y}_{j1})(x_{j3}-x_{j1})},\\
=&(\partial  \bm{{B}}_j)_x=-B_0,
\end{aligned}
\label{Grdb}
\end{equation}
where $(\mathcal{\partial \bm{W}}_j)_x$ is the first component of the unlimited numerical gradient of the variable $w$.  According to Eq. (\ref{Grd1}), we obtain   $(\partial  \bm{{w}}_j)_x=(\partial  \bm{{B}}_j)_x$. By replacing the numerical gradient $ (\partial  \bm{{w}}_j)_x$ by $(\partial  \bm{{B}}_j)_x=-B_0$ and using $h_0=\left({n_f^2 q_0^2}/{B_0} \right)^{3/10}$ in Eq. (\ref{Stqx1}), we get:

\begin{equation}
\begin{aligned}
 \bar{\bm{q}}_{xj}^{n+1}=&\dfrac{q_0\left(h_0^{7/3}-\dfrac{\Delta t g n_f^2}{2}q_0\right)+\Delta t gB_0h_0^{10/3}}{h_0^{7/3}+\dfrac{\Delta t g n_f^2}{2}q_0},\\
 =&\dfrac{q_0\left(h_0^{7/3}-\dfrac{\Delta t g n_f^2}{2}q_0\right)+\Delta t gn_f^2 q_0^2}{h_0^{7/3}+\dfrac{\Delta t g n_f^2}{2}q_0},\\
 =&\dfrac{q_0\left(h_0^{7/3}+\dfrac{\Delta t g n_f^2}{2}q_0\right)}{h_0^{7/3}+\dfrac{\Delta t g n_f^2}{2}q_0},\\
 =&q_0.\\
\end{aligned}
\label{Stqx2}
\end{equation}
This confirms that the computed discharge $q_x$ remains constant over the domain at time $t^{n+1}$.  \\
\emph{-Step 3: Zero-water discharge $q_y$}. Here, we will follow the same techniques used for the water discharge $q_x$ to prove that $\bar{\bm{q}}_{yj}^{n+1}=0$. Due to (\ref{Cnd}), the numerical scheme (\ref{Steady}) for the discharge variable $q_y$ reduces to:
\begin{equation}
\begin{aligned}
\bar{\bm{q}}_{yj}^{n+1} =&-\dfrac{\Delta t}{|\bm{T}_j|}\sum_{k=1}^3\dfrac{d_{jk}\sin(\theta_{jk})}{b_{jk}^{in}+b_{jk}^{out}}\Big[\frac{g}{2}h_0^2 (b_{jk}^{in} +b_{jk}^{out})  \Big]\\
&+\frac{\Delta t g}{2|\bm{T}_j|}\sum_{k=1}^{3}d_{jk}\sin(\theta_{jk})h_0^2-g\Delta t h_0(\partial  \bm{{w}}_j)_y.
\end{aligned}
\label{Steay}
\end{equation}
Using $\sum_{k=1}^3 d_{jk}\sin(\theta_{jk})=0$ at each cell, Eq. (\ref{Steay}) is simply: 
\begin{equation}
\begin{aligned}
\bar{\bm{q}}_{yj}^{n+1} =-g\Delta t h_0(\partial  \bm{{w}}_j)_y.
\end{aligned}
\label{Steay2}
\end{equation}
By incorporating Eq. (\ref{BB}) in the second equation of (\ref{Grd}) and using some straightforward calculations, we obtain $(\bm{{w}}_j)_y=(\bm{{B}}_j)_y=0$. Then, Eq. (\ref{Steay2}) becomes $\bar{\bm{q}}_{yj}^{n+1} =0$, that is the water discharge in y-direction remains zero over the domain at time $t^ {n+1} $. We conclude that, the numerical scheme (\ref{Scheme})-(\ref{Fri}) preserves the nontrivial steady-state  solutions.\\
The same procedure can be applied for the nontrivial steady-state solutions in the $y-$direction.
\end{proof}
\section{Numerical experiments}\label{S4}
In this section, we perform numerical experiments to investigate the performance of the proposed numerical model for simulating waves propagation and waves run-up on coastal areas. The results of the proposed numerical model are compared with available experimental data from laboratory experiments. In all numerical examples, we used $g=9.81 m/s^2$ for gravity acceleration. We start our numerical experiments by demonstrating the nontrivial well-balanced property of the numerical model in Example $\ref{Ex0}$. In Example $\ref{Ex1}$, we perform numerical simulations of dam break flow over a sloping bed. In Example $\ref{Ex2}$, we simulate the propagation and run-up of breaking and non-breaking solitary waves on a sloping beach. In Example $\ref{Ex3}$, we test the ability of the numerical model for simulating periodic waves over a sloping beach. In Example $\ref{Ex4}$, we study the evolution of solitary wave over a conical island. Finally in Example $\ref{Ex5}$, the proposed numerical model is applied to predict waves propagation along a  sloping beach with complex bottom topography.
\subsection{Steady flow over a sloping bed} \label{Ex0}
In this first numerical example, we validate the model's ability to preserve nontrivial steady-state solutions over a sloping bed. We consider a computational domain $[0, 2.5 ]\times[0, 0.2]$ which is discretized using  $12863$ triangular cells. The bottom topography is as follows:
\begin{equation}
B(x,y)=-0.015(x-2.53).
\end{equation}
The water depth is initially constant everywhere  over the domain $h(x,y,0)=h_0=\left(\dfrac{n_f^2q_0^2 }{0.015}\right)^{3/10}$ with constant flow discharges $q_x(x,y,0)=q_0$ and $q_y(x,y,0)=0$. In our experiments, we study the case for supercritical and subcritical flows. In the supercritical flow case, we use the initial discharge $q_0=0.02\,m^{2}/s$ and Manning's coefficient $n = 0.01\, m^{-1/3}s$ with Froude number $Fr=2.058$. While in the subcritical flow case, the initial discharge $q_0=0.1\, m^{2}/s$ and Manning’s coefficient $n = 0.05 \, m^{-1/3}s$  are used and  Froude number is $Fr=0.568$. We set the following boundary conditions; $u_{jk}=u_{j}$,  $v_{jk}=v_{j}$ and $h_{jk}=h_{j}$ in $x-$direction and $u_{jk}=u_{j}$, $v_{jk}=0$ and $h_{jk}=h_{j}$ in $y-$direction.\\
To assess the accuracy of the numerical model for preserving nontrivial steady-state solutions, we compute  errors for $L_1$, $L_2$ and $L_{\infty}$ norms for water depth $h$ and water discharges $q_x$ and $q_y$ variables at time $t=150\, s$ for two case studies shown in Table \ref{Tbl}. In Figure \ref{FF2}, we show the computed free-surface elevation and water discharges compared with exact solution at time $t=150\,s$.  We obtain accurate results and the computed solution remains steady for a large simulation time.

\begin{center}
\captionof{table}{Computed errors $L_1$, $L_2$ and $L_{\infty}$ for the water depth $h$ and water discharges $q_x$ and $q_y$.}
\begin{tabular}{ c c c c c}
\hline 
Flow type& Variable  & $L_1$-error & $L_2$-error & $L_{\infty}$-error \\ 
\hline 
 & $h$ & $7.337\times 10^{-14}$ & $1.94\times 10^{-12}$ & $3.1\times 10^{-15}$ \\ 

supercritical  & $q_x$ &  $1.32\times 10^{-14}$ &  $3.20\times 10^{-13}$ &  $5.72\times 10^{-16}$\\ 

 & $q_y$ &  $3.85\times 10^{-18}$ &  $4.96\times 10^{-18}$ &  $1.94\times 10^{-17}$\\ 
\hline 
 & $h$ &  $2.64\times 10^{-15}$ &  $1.93\times 10^{-13}$ &  $1.61\times 10^{-15}$ \\ 

subcritical & $q_x$ &  $7.51\times 10^{-14}$ &  $3.58\times 10^{-12}$ &  $1.74\times 10^{-14}$ \\ 

 & $q_y$ &  $1.17\times 10^{-17}$ &  $9.41\times 10^{-17}$ &  $3.22\times 10^{-16}$ \\ 
\hline 
\end{tabular}
\label{Tbl}

 \end{center} 

\begin{figure}[!ht]
\begin{center}
(a)\\ 
\includegraphics[scale=0.28]{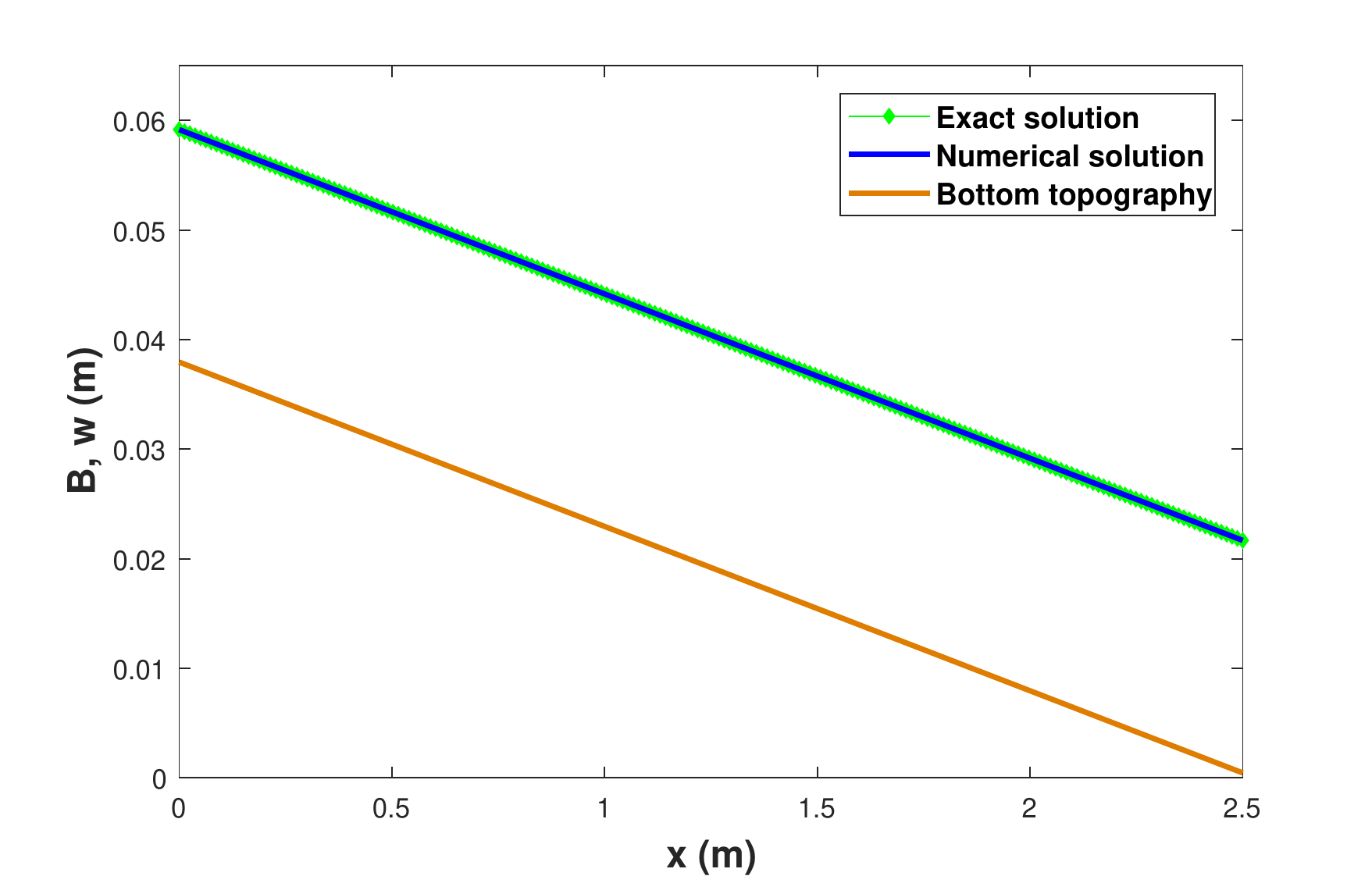}\quad \includegraphics[scale=0.28]{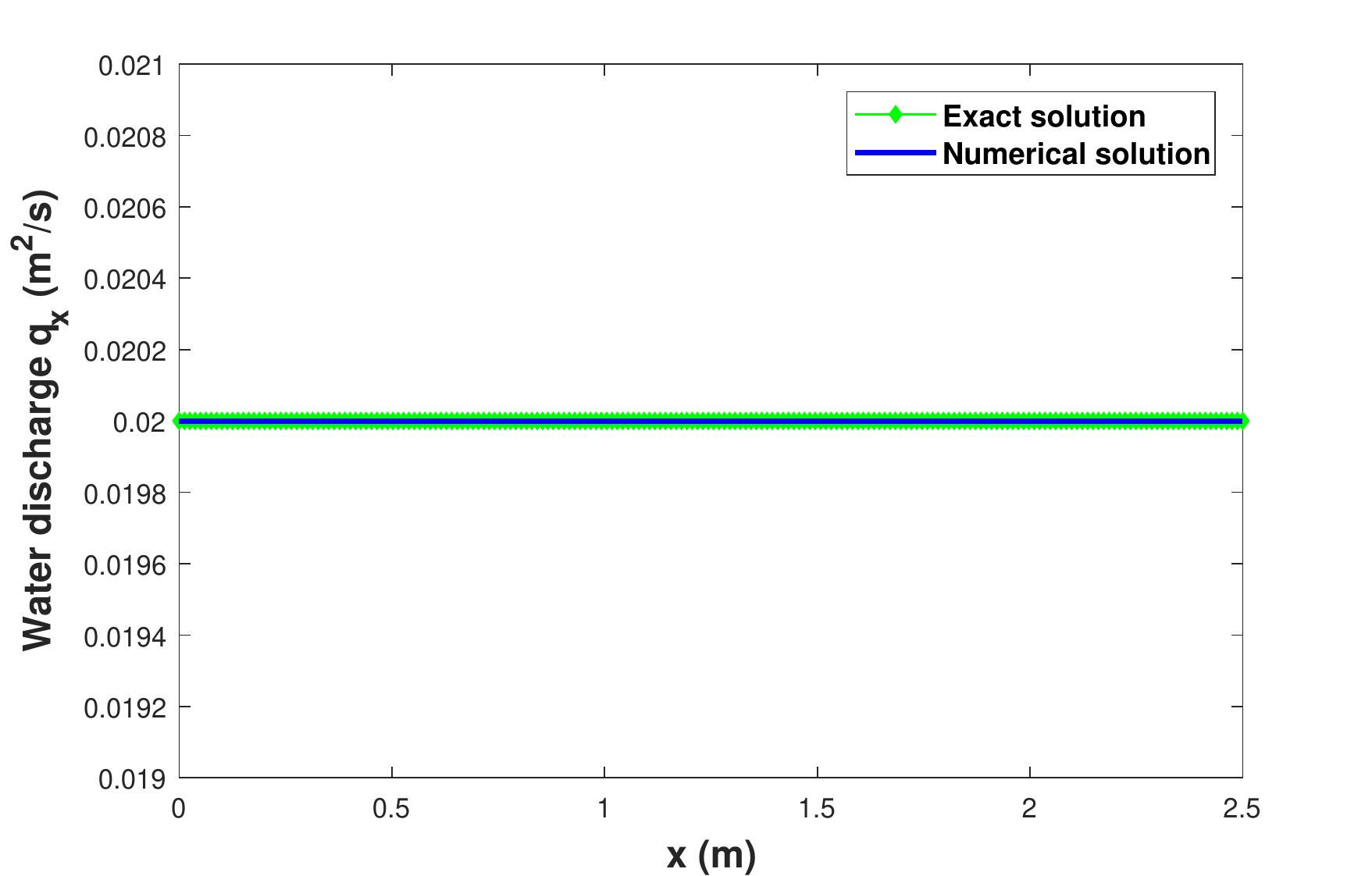}\quad \includegraphics[scale=0.28]{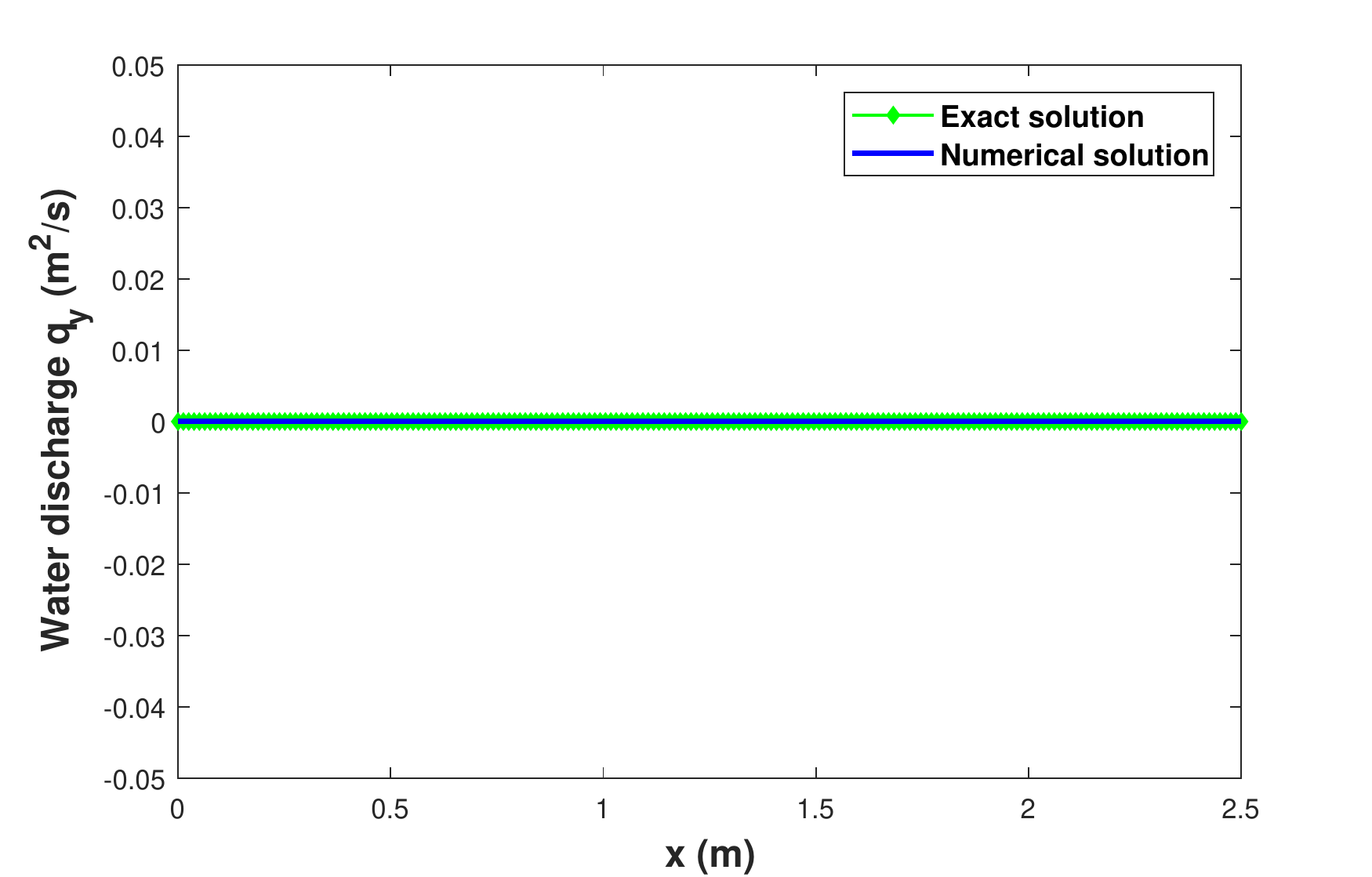}

(b)\\ 
\includegraphics[scale=0.28]{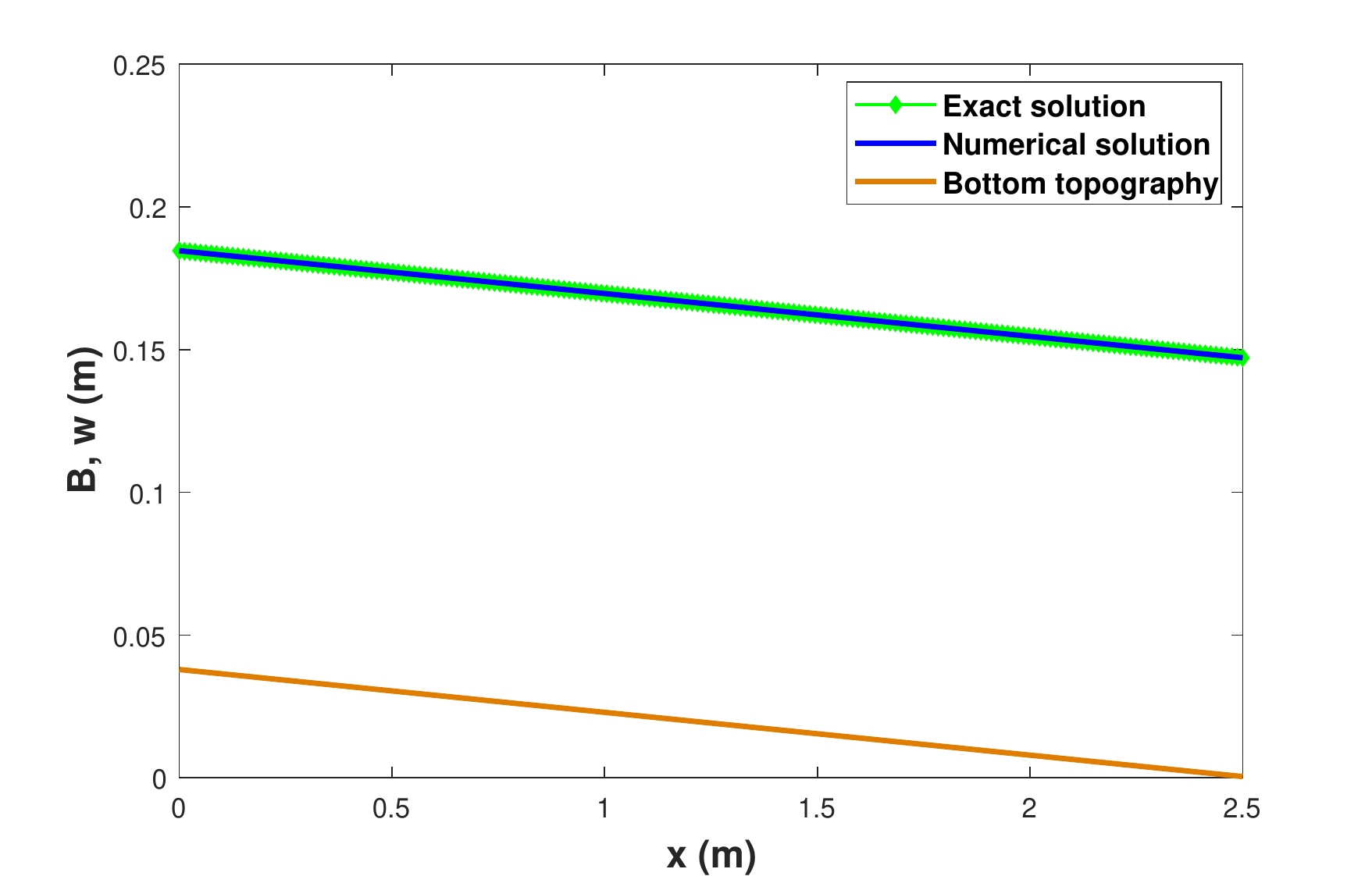}\quad \includegraphics[scale=0.28]{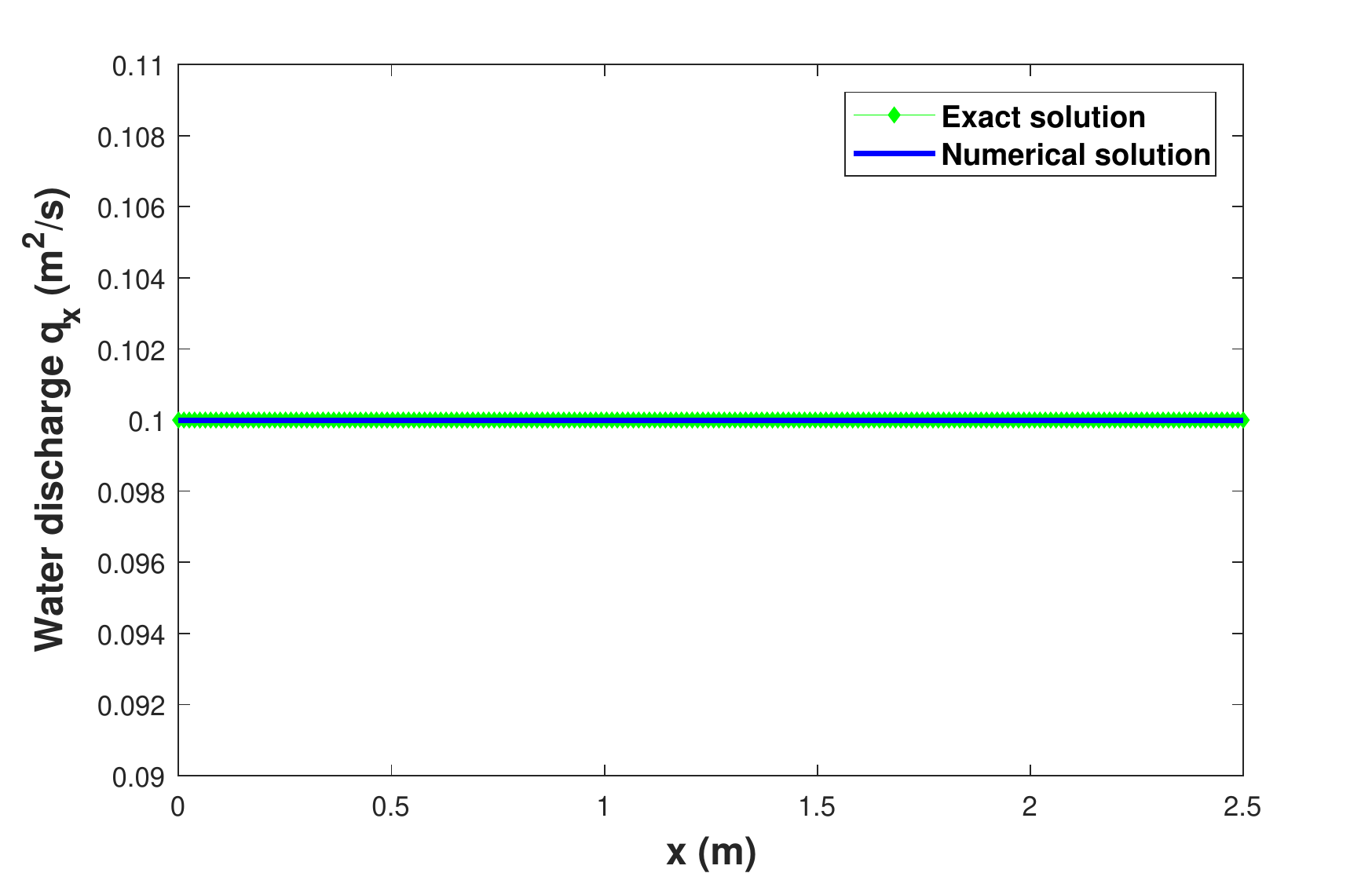}\quad \includegraphics[scale=0.28]{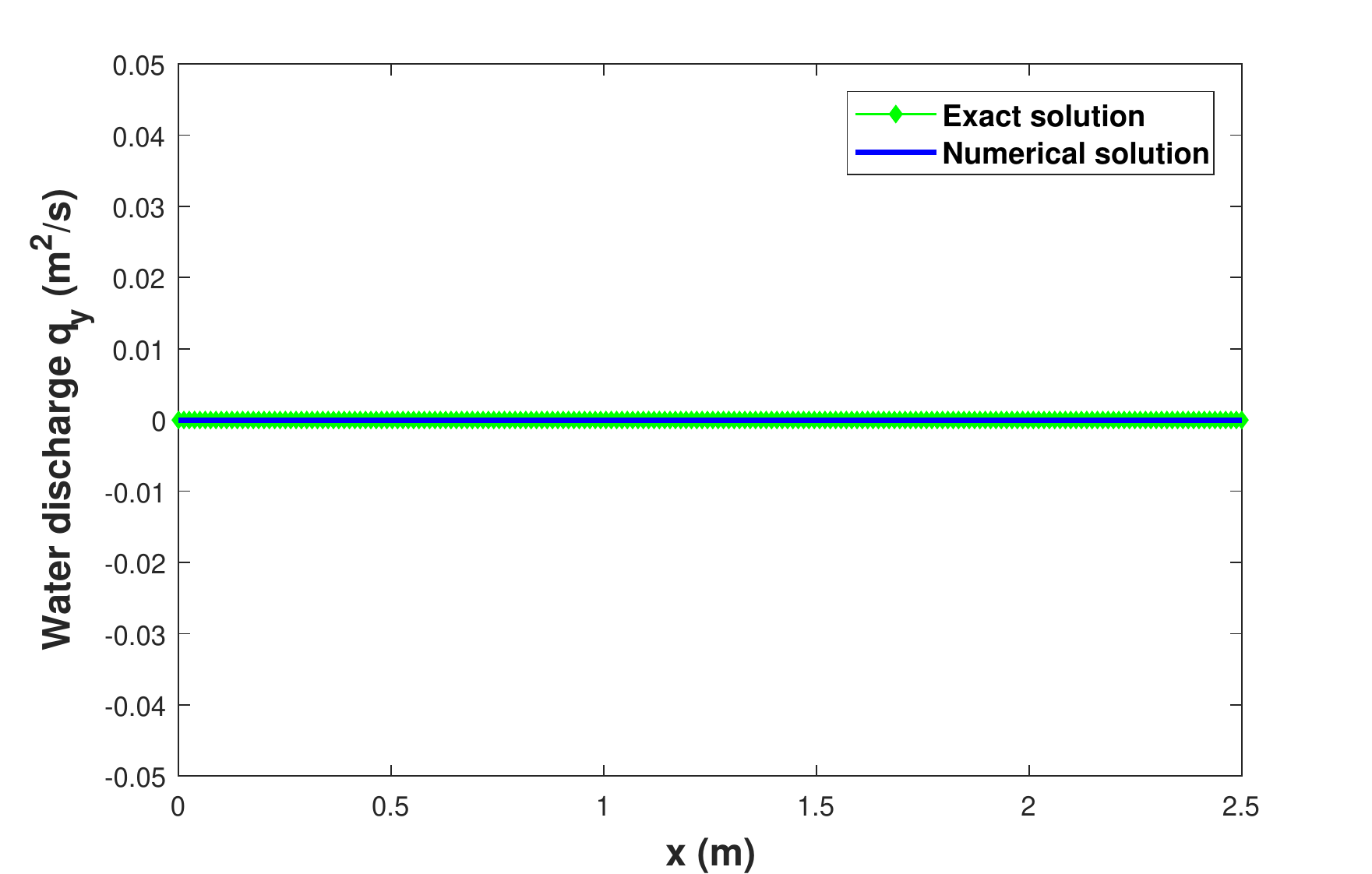} \quad
\caption{Computed water free-surface elevation and water discharge over a slanted surface at time $t=150\, s$ in the supercritical case (a) and  in the subcritical case (b).  }
\label{FF2}
\end{center}
\end{figure}
\subsection{Dam break flow on a sloping bed}\label{Ex1}
In this second numerical example, we consider a laboratory experiment of dam break flow on a sloping bed, carried out by Aureli et al.\cite{aureli2000numerical} at Department of Civil Engineering, Parma University. The experiment was conducted on a rectangular flume of $7\,m $ long, $1\, m$ wide, and $0.5\, m$ high with $1:10$ sloping bed located at $x=3.4\, m$ downstream of the dam as shown in Figure \ref{Fi31}. The dam gate was controlled by an oil pressurized circuit located at $x=2.25\, m$ \cite{aureli2000numerical}. Initially, upstream of the dam, the reservoir has a water depth of $0.25\, m$ while the water depth is zero at the downstream of the dam. In the experiment, complex hydrodynamic behaviors are observed such as run-up and run-down between wet and dry areas. The measurements of the water depth have been performed at four locations $G_1$ $(1.4, 0.5)$, $G_2$ $(2.25, 0.5)$, $G_3$ $(3.4, 0.5)$, and $G_4$ $(4.5, 0.5)$ along the flume. In the numerical simulations, the computational domain $[0, 7]\times [0, 1]$ is discretized using $18082$ triangular cells. The Manning friction coefficient $ n_f=0.01\, m^{-1/3}s$ is used, and wall boundary conditions are considered at all sides of the domain except at the outlet side where outflow condition is imposed. 

\begin{figure}[!ht]
\begin{center}
\includegraphics[scale=0.55]{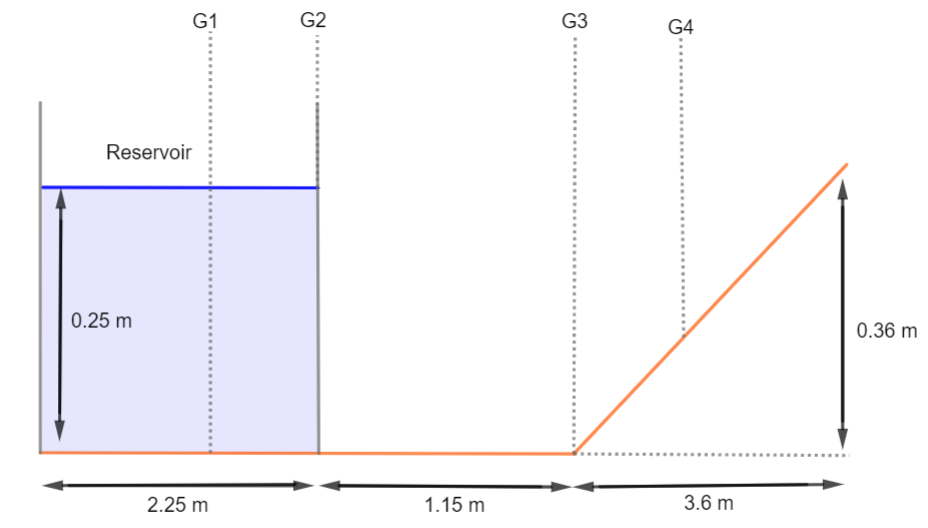}
\caption{Schematic of the dam-break flow experiment Aureli et al.\cite{aureli2000numerical}. }
\label{Fi31}
\end{center}
\end{figure}
The evolution of the computed water depth using the proposed numerical model and the experimental measurements at the four gauges are shown in Figure \ref{Figg31}. Initially, once the dam gate is removed the water flooded from the reservoir to downstream generating run-up and run-down motions at the sloping bed. The dam-break wave propagates forward and backward over the domain and wave run-up and run-down motions occur several times. Due to the presence of bed friction effects, the flow velocity decreases and the surface water stabilizes for large time. The simulation results show that the time evolution of water depth at the gauges in different locations in space is well predicted compared to experimental measurements which confirms that the proposed numerical model is accurate for modeling waves run-up and run-down between wet and dry zones.
\begin{figure}[!ht]
\begin{center}
\includegraphics[scale=0.38]{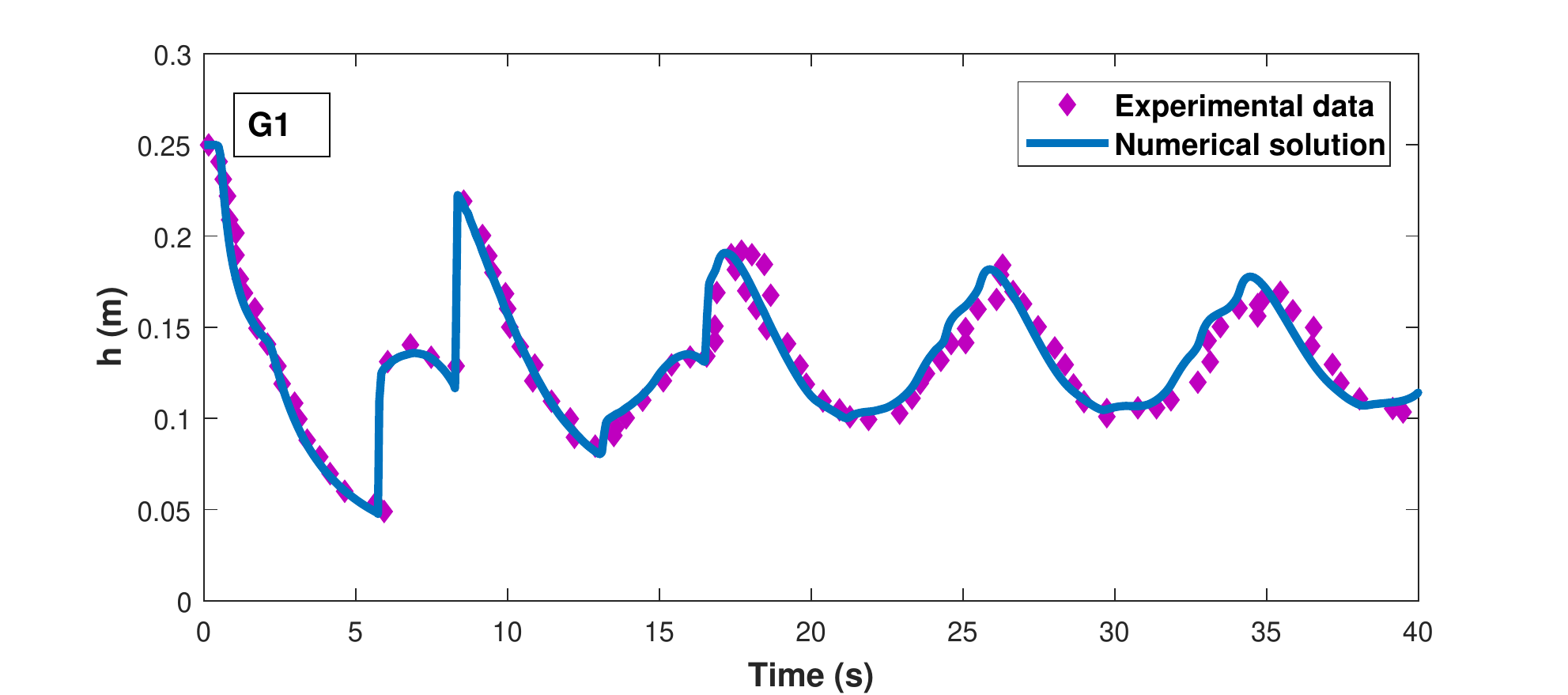}\includegraphics[scale=0.38]{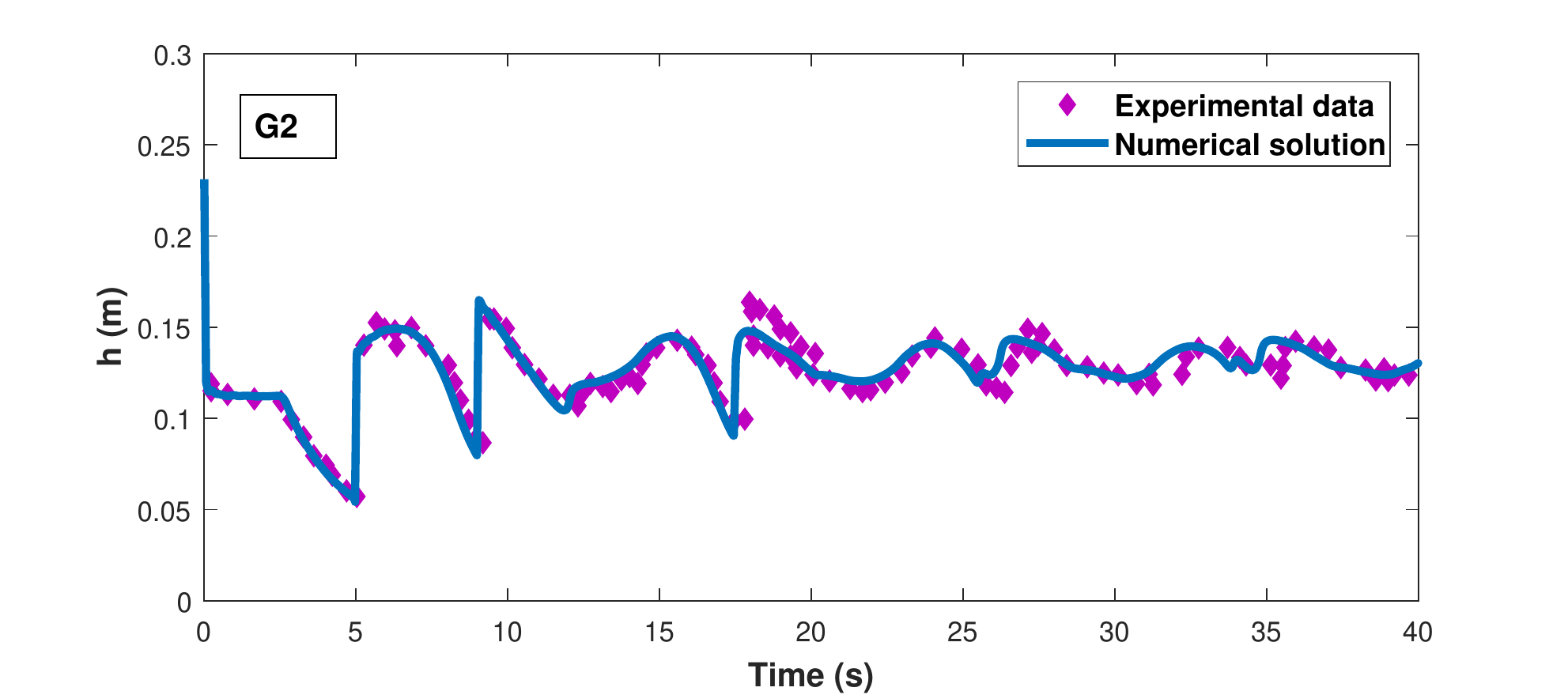}\quad 
\includegraphics[scale=0.38]{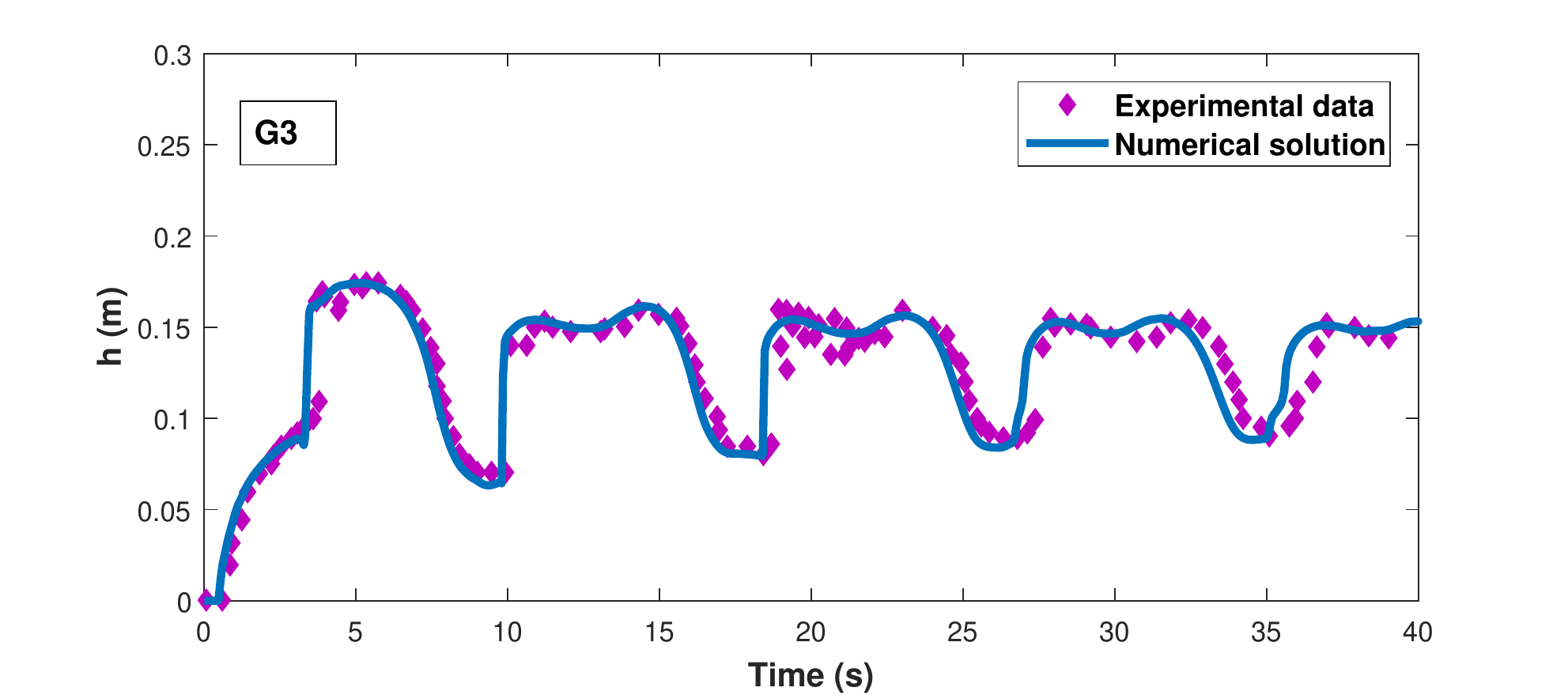}\includegraphics[scale=0.38]{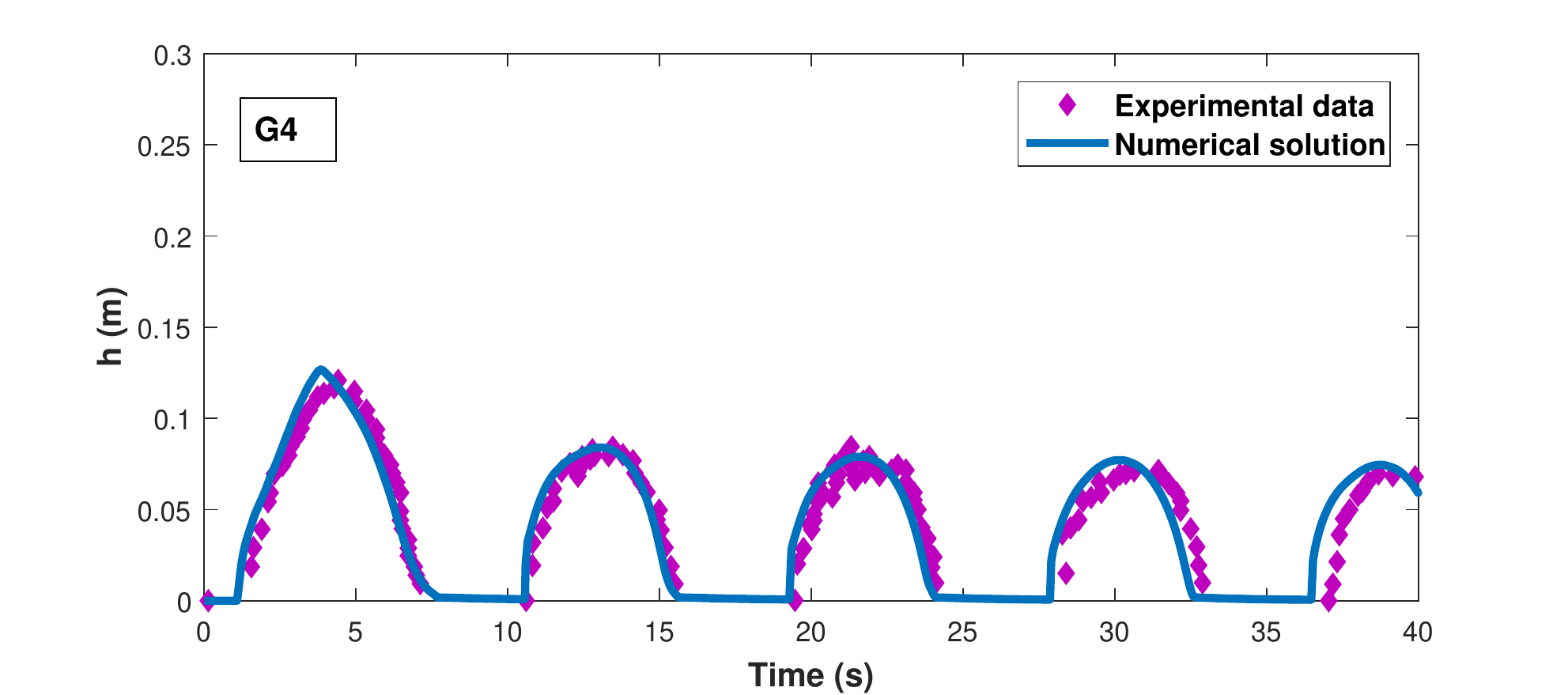}\quad 

\caption{Computed water depth compared with experimental measurements at the four gauges.}
\label{Figg31}
\end{center}
\end{figure}
\subsection{Solitary wave run-up on a sloping beach}\label{Ex2}
In this numerical example, we test the ability of the numerical model for predicting the propagation and run-up of breaking and non-breaking solitary waves on a sloping beach. The numerical results are validated using laboratory experimental data provided by Synolakis \cite{synolakis1986runup,synolakis1987runup}. The experiment consists of $1:19.85$ plane sloping beach of angle $\varphi$ with constant bed connected to a constant depth region. The initial water surface profile is a solitary wave of height $H$ centered at $x_0$ which propagates over a still water of constant depth $h_0$ as shown in Figure \ref{Fig2}. In our simulations, we consider the computational domain $[-20, 60]\times [0, 1]$ which is discretized using $41377$ triangular cells. The initial condition for free-surface elevation and velocity is \cite{que2006numerical,mahdavi2009modeling}:

\begin{equation}
w(x,y,0)=H\sech^2\left[\sqrt{\dfrac{3 H}{4h_0^3}}(x-x_0) \right],\,\, \bm{u}(x,y,0)=\left[\dfrac{c w(x,y,0)}{w(x,y,0)+h_0},0\right]^T, 
\end{equation}
where $c=\sqrt{g(h_0+H)}$ being the wave velocity. The still water depth $h_0=1\, m$ and the wave height $H=0.3\,m$ for breaking  and  $H= 0.0185\,m$ for non-breaking waves are used. The Manning friction coefficient is set to $n_f=0.01\, m^{-1/3} s$ and  wall boundary conditions are used at the lateral parts of the computational domain while inflow and outflow conditions are imposed at the right and left boundaries, respectively.
\begin{figure}[!ht]
\begin{center}
\includegraphics[scale=0.55]{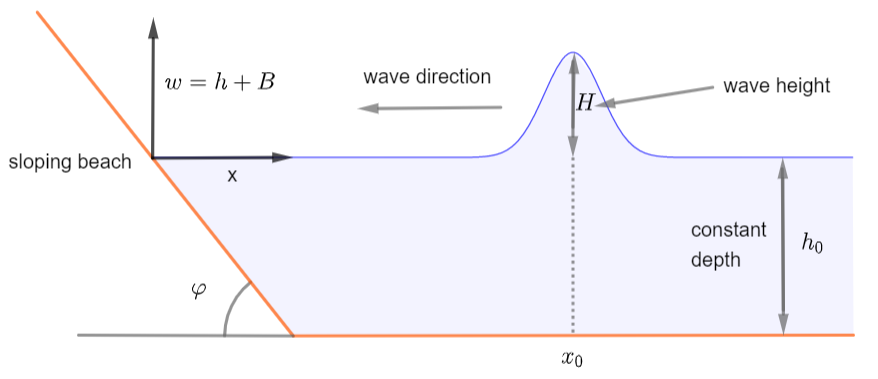}
\caption{1D representation of the solitary wave over a sloping beach.}
\label{Fig2}
\end{center}
\end{figure}

Figure \ref{Fig3} shows the evolution of breaking solitary wave over the sloping beach simulated at different times $t=0,\,1.6,\,4.79,\,11.18$, and $17.59\,s$. The cross section along the x-axis of the computed free-surface elevation is compared with experimental measurements at each time as shown in Figure \ref{Fig3} (left). In Figure \ref{Fig3} (right), we show the three-dimensional view of the computed solution. The wave is initially centered at $x_0=14\, m$ and propagates first over the region of constant depth towards the beach. Once the wave arrives the shoreline,  it  climbs up the sloping bed to reach its maximum run-up at about time $t=11.18\, s$. The results show that breaking occurs as the wave advances over the region of constant depth, when vertical front of the wave is reached as shown in Figure \ref{Fig3}.\\
For the non-breaking case, the wave is initially centered at $x_0=38.5\, m$.
Figure \ref{Fig4} shows the computed free surface elevation of non-breaking solitary wave at times $t=8,\,11.17,\,14.37,\,17.56$, and $20.75\,s$, and the results are compared with the experimental measurements.  As in the breaking case, the incident wave advances first over the region of constant depth and climbs up the sloping bed when it  arrives the shoreline.  The wave reaches its maximum  run-up at about time $t =17.56\, s$. We observe a good agreement between the numerical simulations and experimental measurements which demonstrates  the accuracy of the proposed numerical model for predicting the evolution of breaking and non-braking solitary waves on a sloping beach.

\begin{figure}[!ht]
\begin{center}

\includegraphics[scale=0.4]{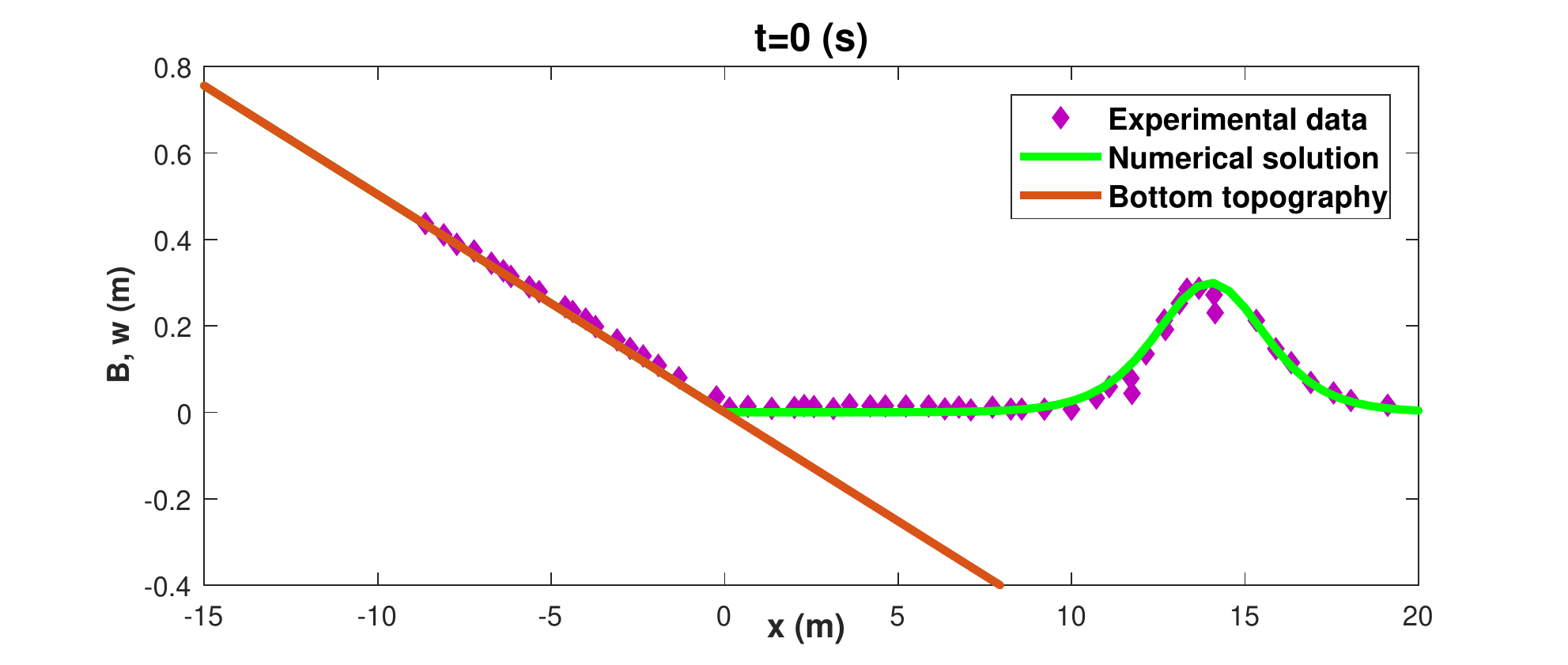}\quad \includegraphics[scale=0.23]{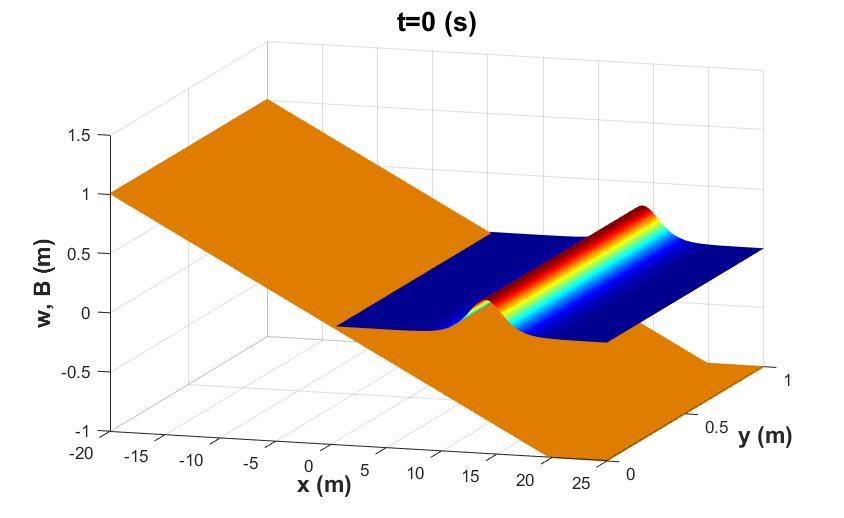}

\includegraphics[scale=0.38]{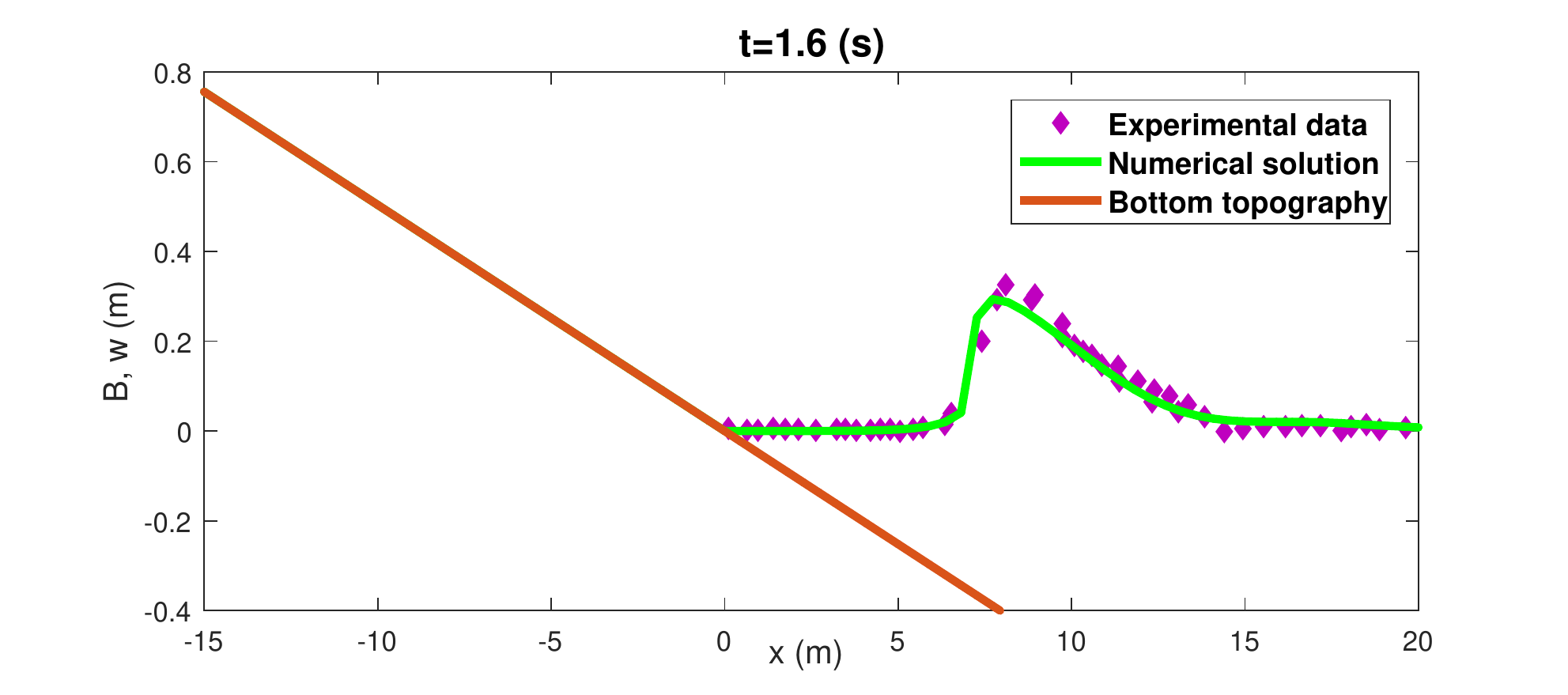}\quad \includegraphics[scale=0.22]{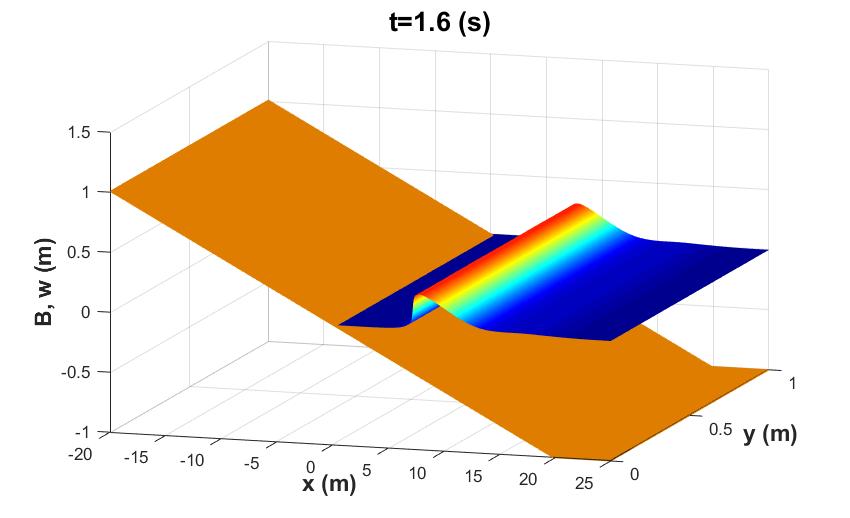}
\includegraphics[scale=0.38]{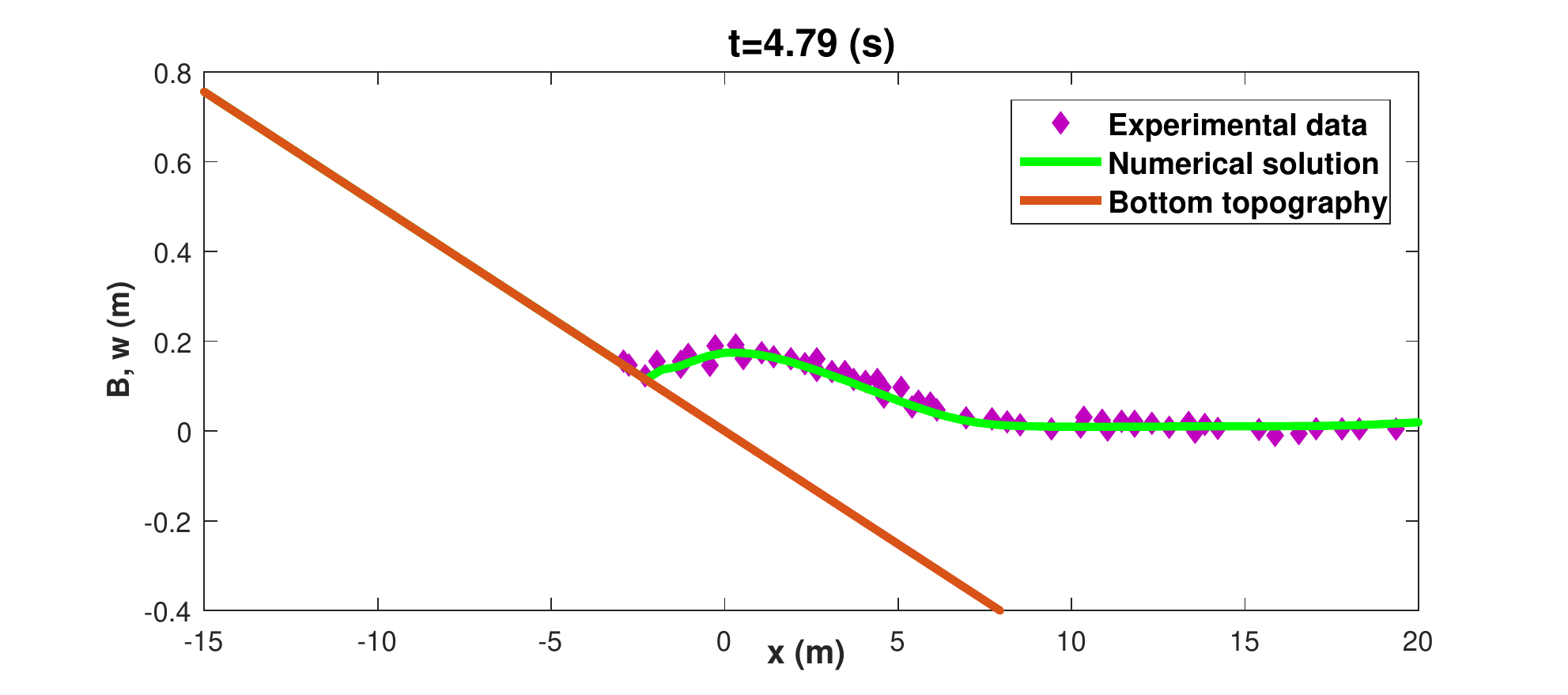}\quad \includegraphics[scale=0.22]{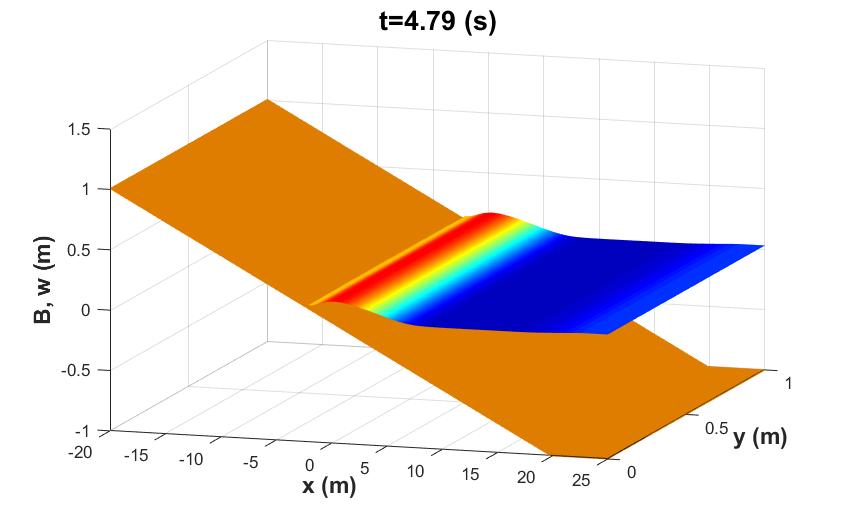}
\includegraphics[scale=0.38]{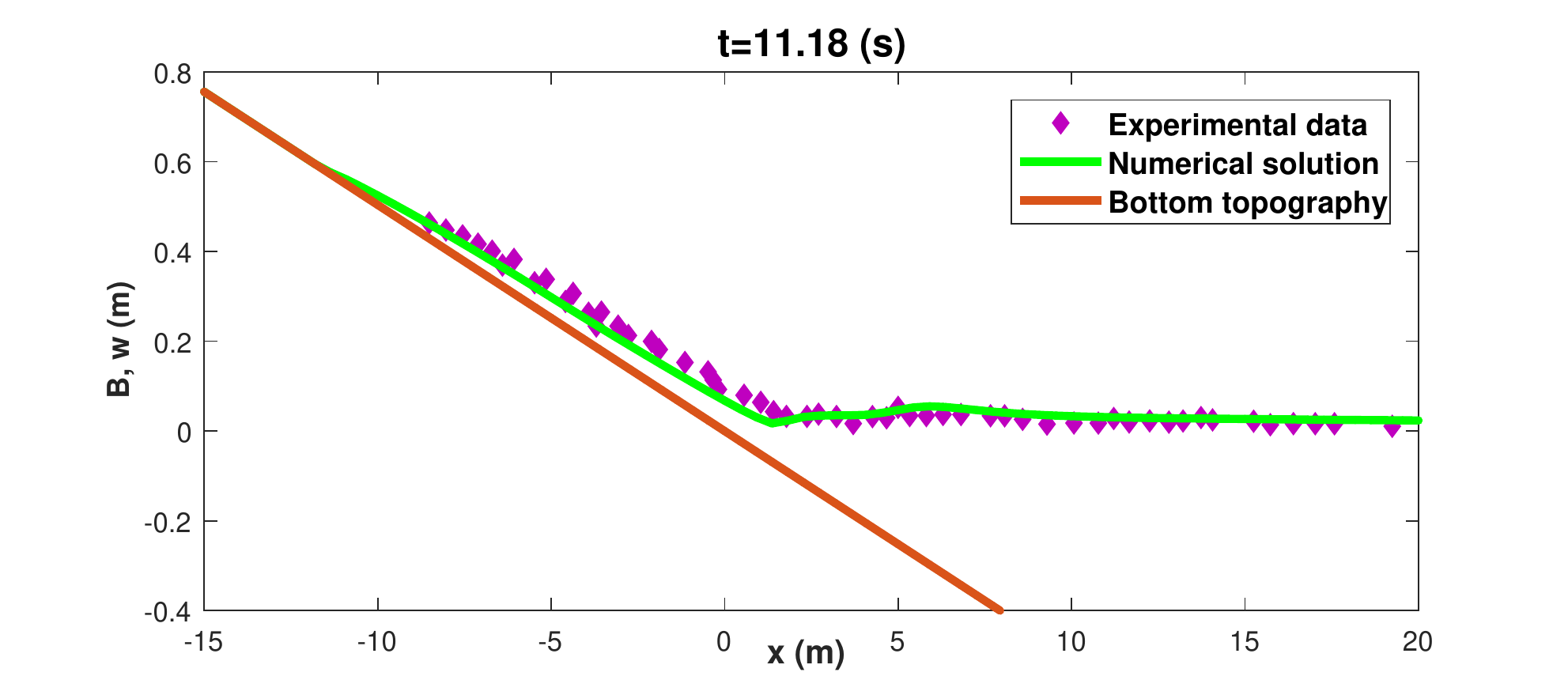}\quad \includegraphics[scale=0.22]{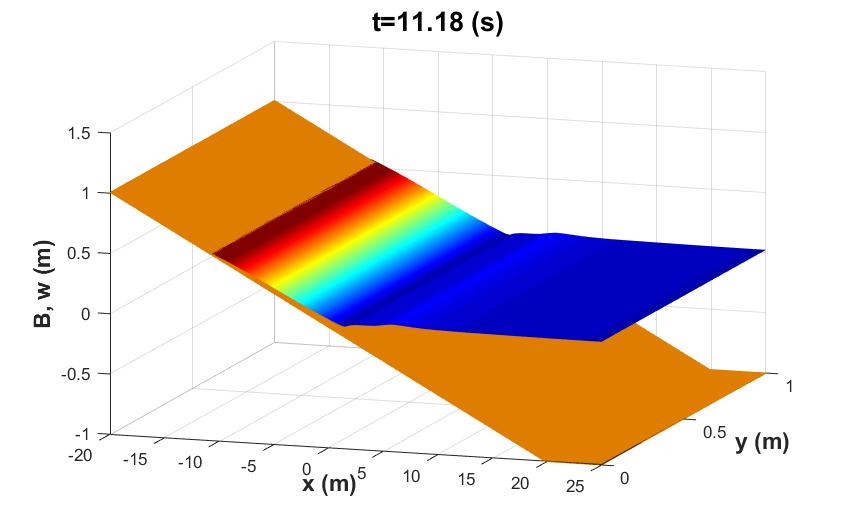}
\includegraphics[scale=0.38]{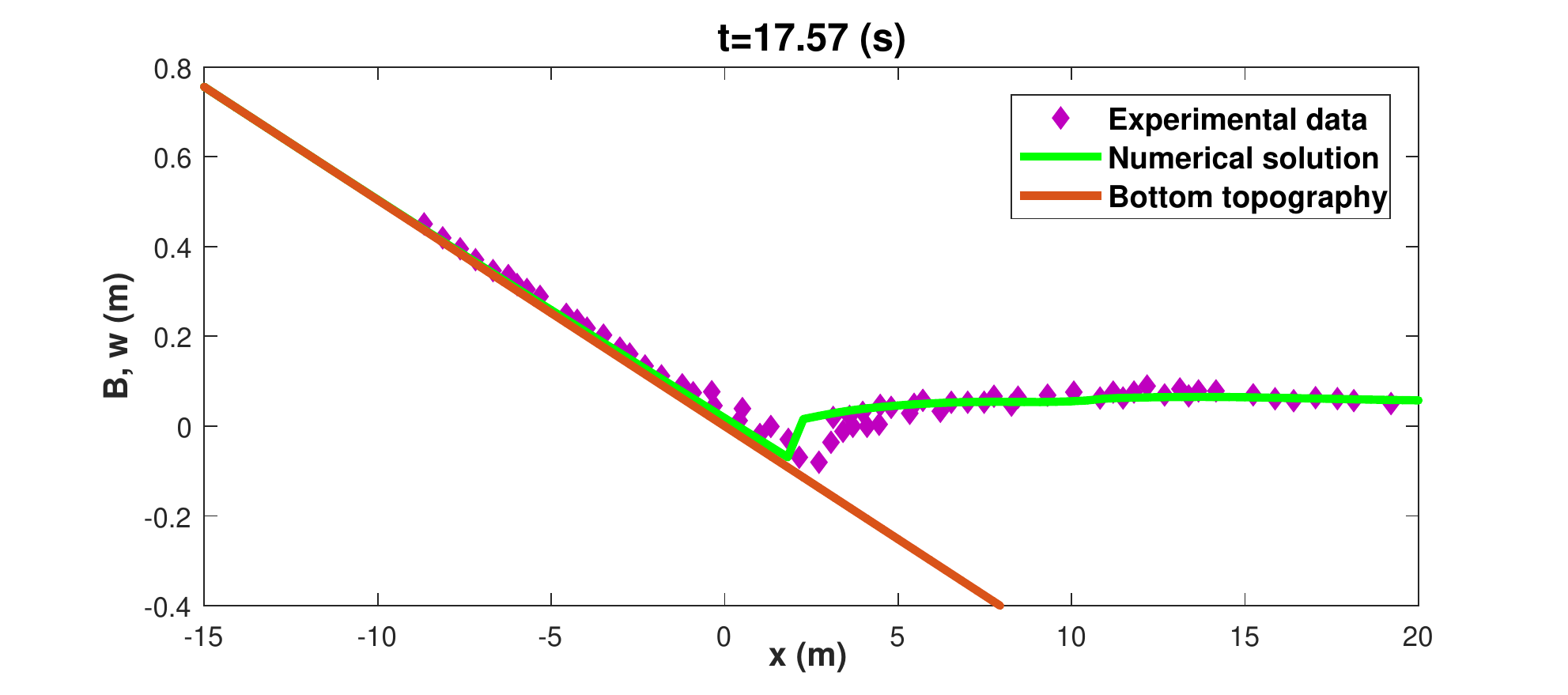}\quad \includegraphics[scale=0.22]{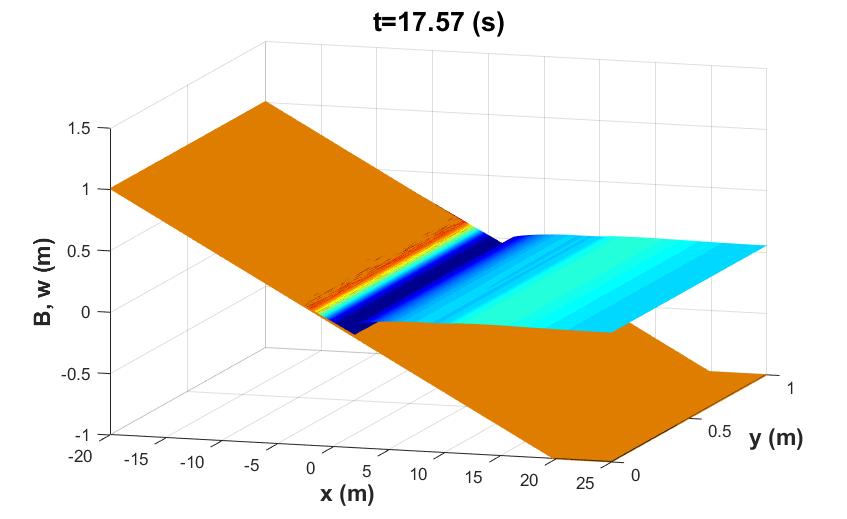}
\caption{Modeling of breaking solitary wave on a sloping beach:  Predicted free-surface elevation compared with experimental measurements at different times (left) and the corresponding three-dimensional view (right).}
\label{Fig3}
\end{center}
\end{figure}
\begin{figure}[!ht]
\begin{center}

\includegraphics[scale=0.38]{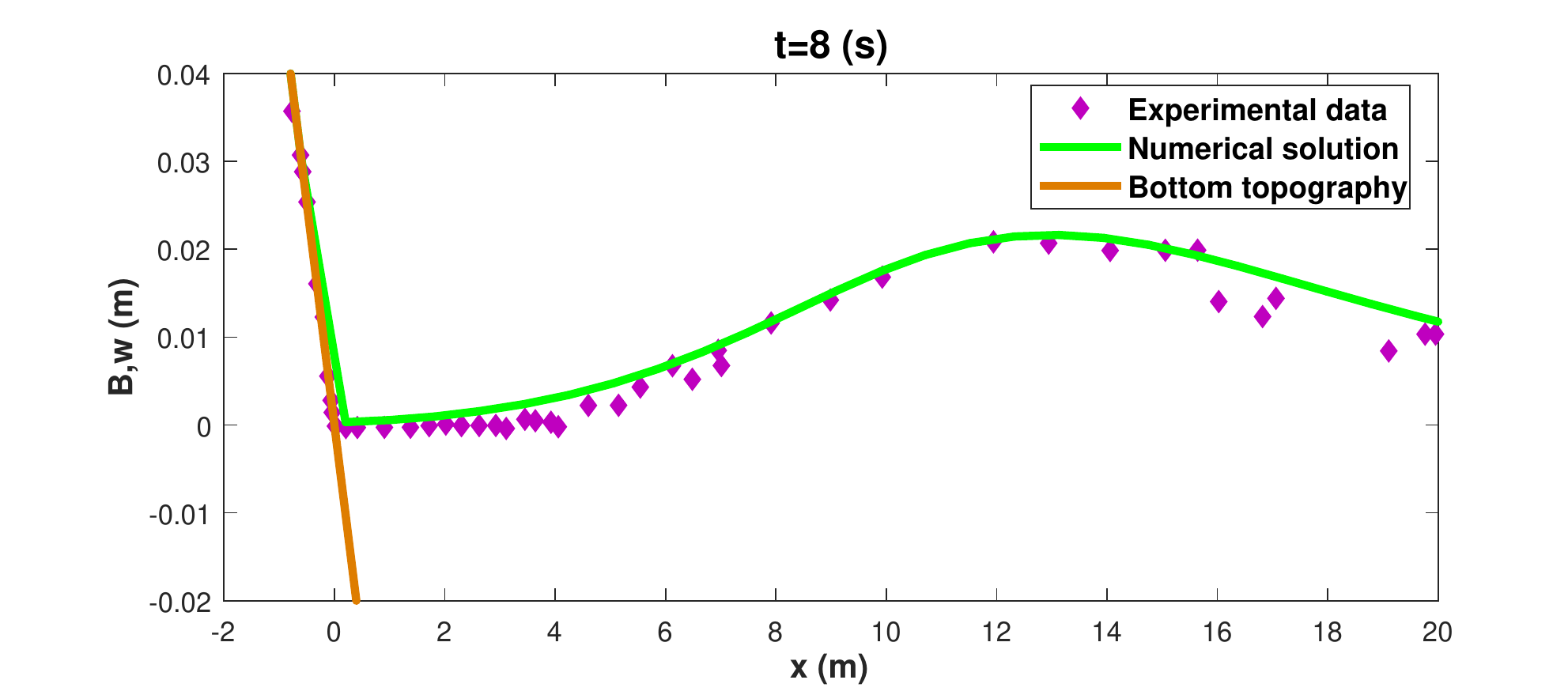}\quad \includegraphics[scale=0.23]{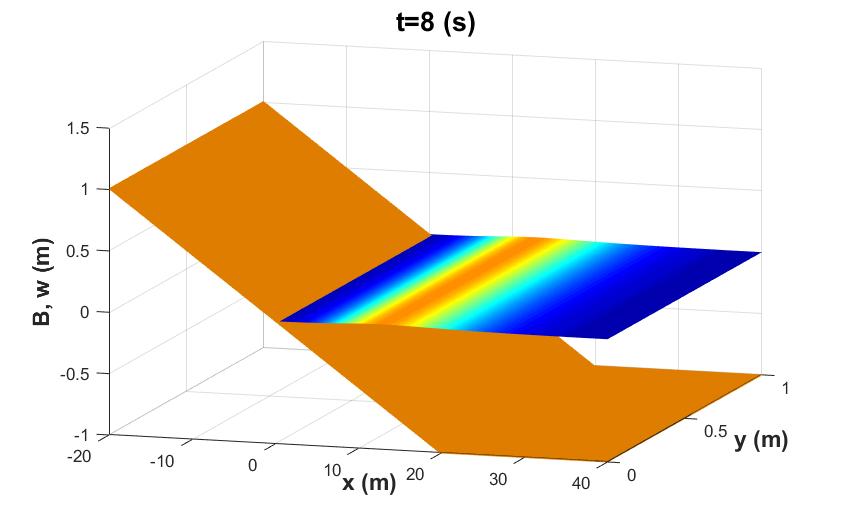}

\includegraphics[scale=0.38]{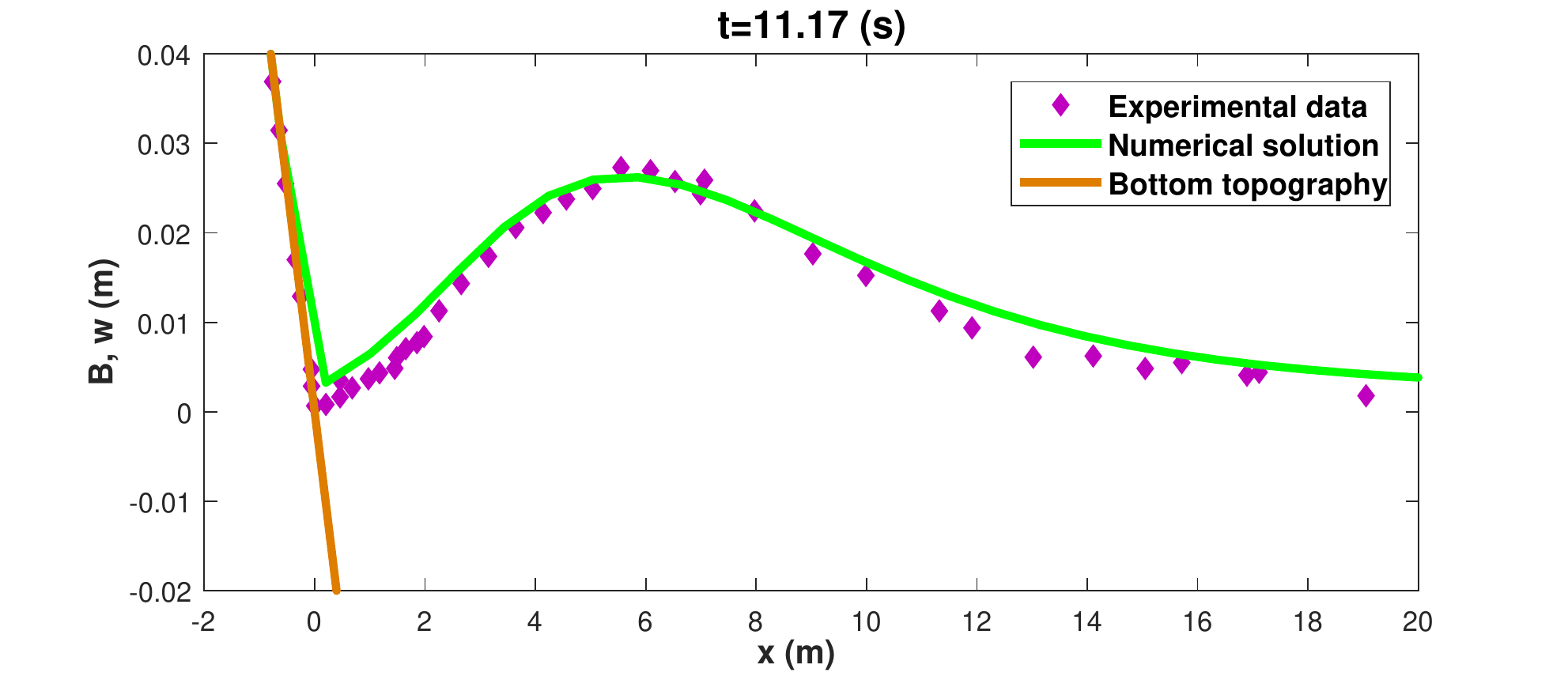}\quad \includegraphics[scale=0.22]{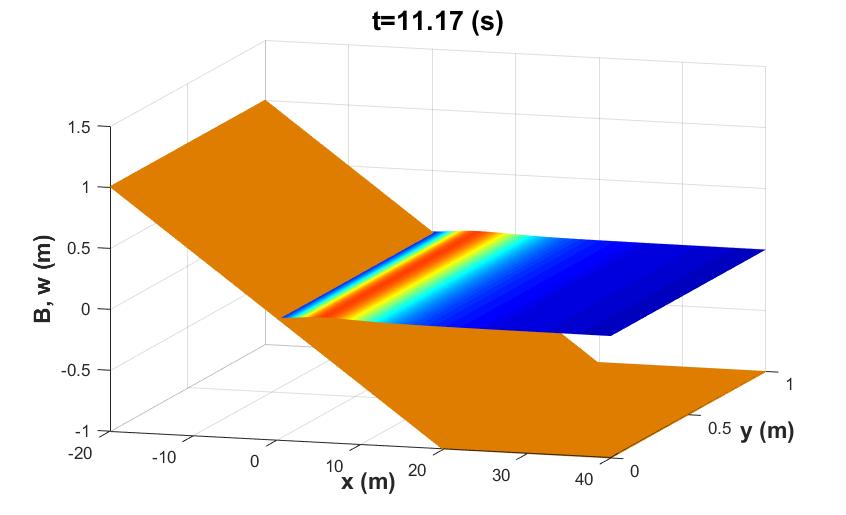}
\includegraphics[scale=0.38]{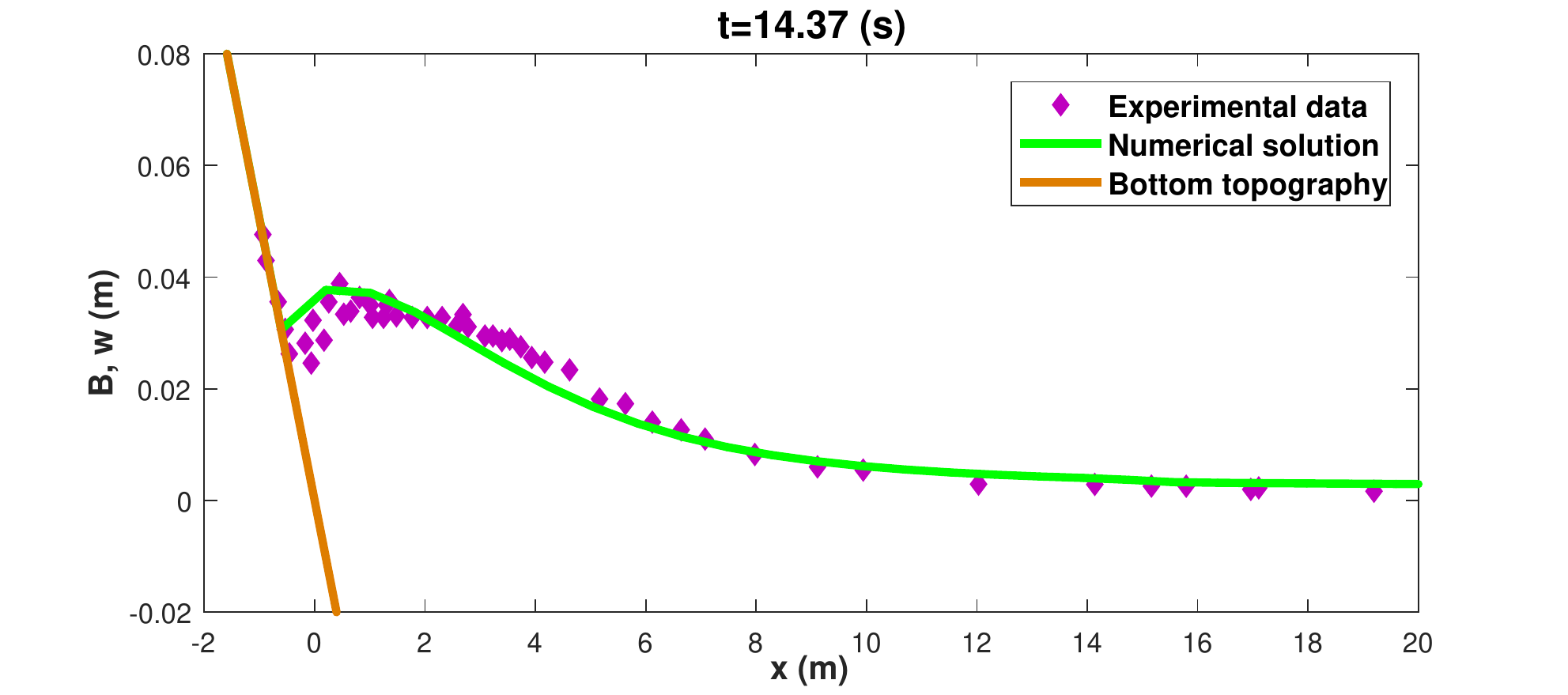}\quad \includegraphics[scale=0.22]{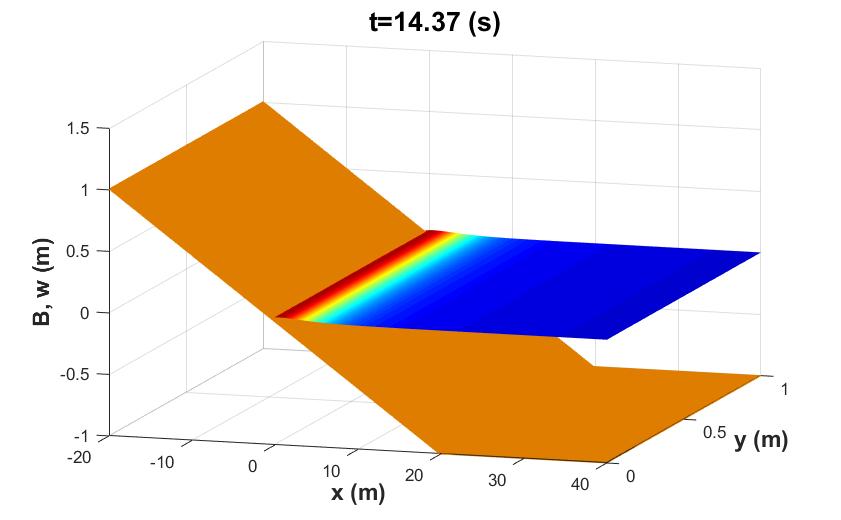}
\includegraphics[scale=0.38]{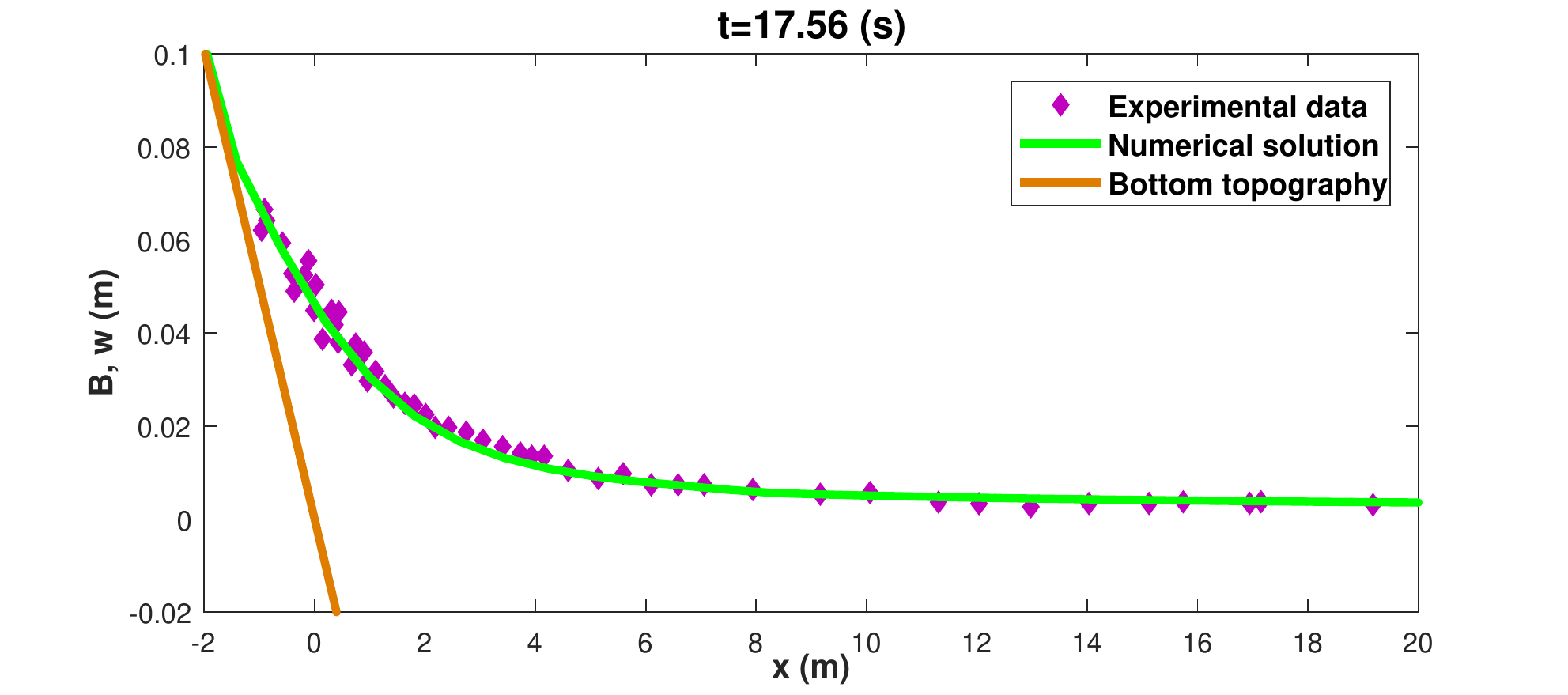}\quad \includegraphics[scale=0.22]{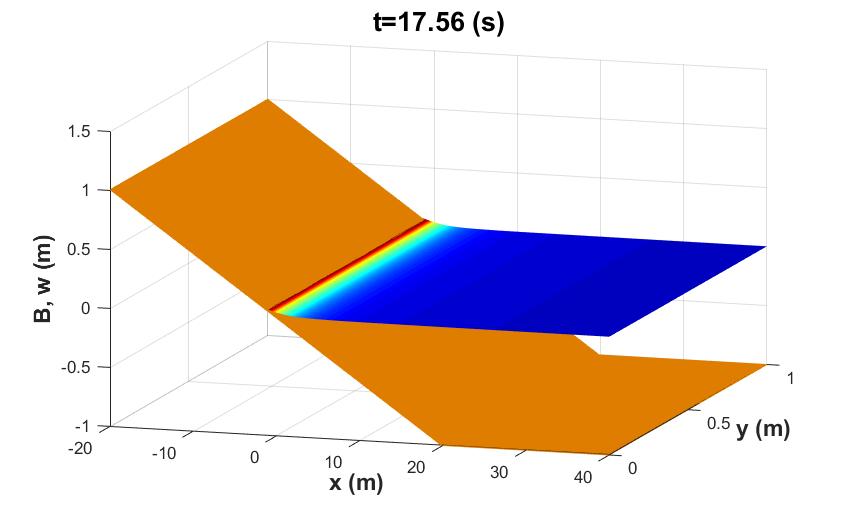}
\includegraphics[scale=0.38]{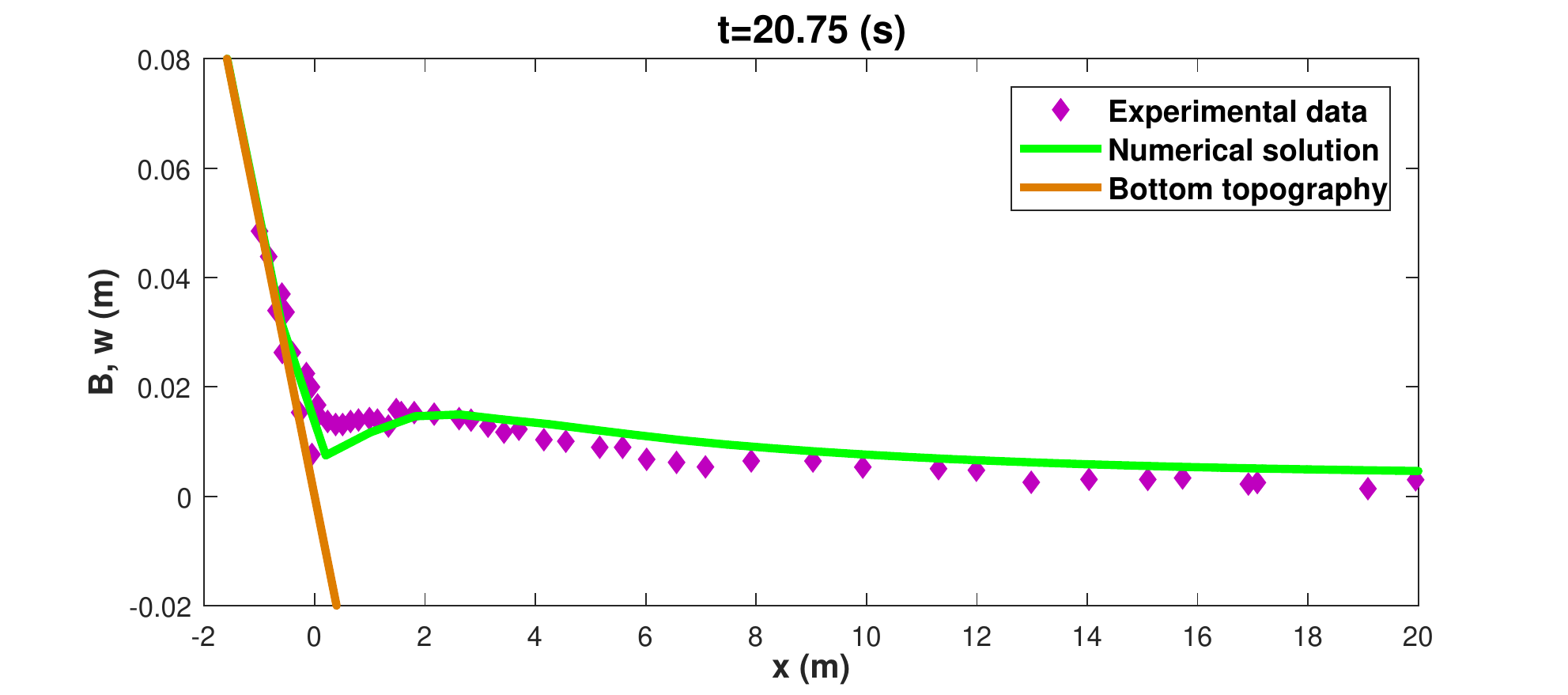}\quad \includegraphics[scale=0.22]{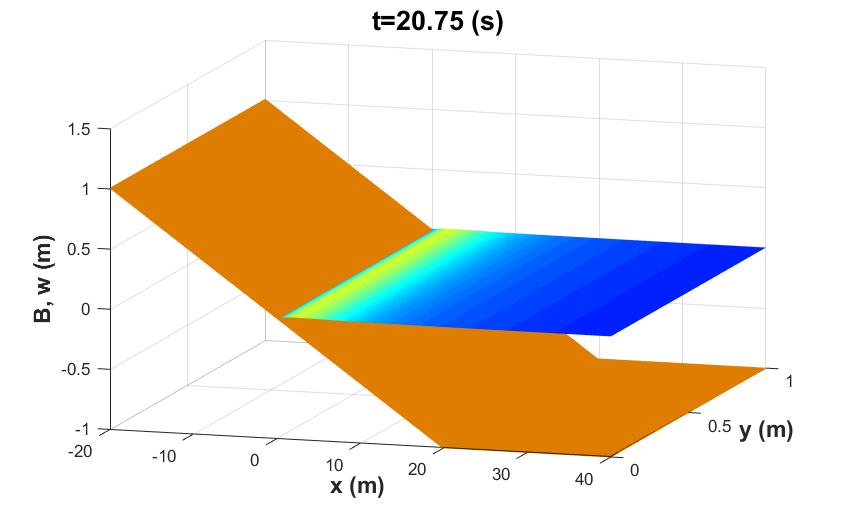}
\caption{Modeling of non-breaking solitary wave over a sloping beach:  Predicted free surface elevation compared to experimental measurements at different times (left) and the corresponding three-dimensional view (right).}
\label{Fig4}
\end{center}
\end{figure}
\newpage
\subsection{Periodic waves propagation on a sloping beach}\label{Ex3}
Here, we will asses the ability of the numerical model for predicting the evolution and run-up of periodic waves on a sloping beach, where the numerical results are compared with data from the Cox laboratory experiment \cite{cox1995experimental}.
The experiment was conducted in a rectangular flume of $33\,m$ long, $0.6\, m$ wide, and $1.5\, m$ height with  $1:35$ sloping bed.  The periodic waves of height $H=0.115\, m$ and period $T=2.2\ s$ were generated at the left boundary of the flume and propagate over a still water of constant depth $h_0=0.4\,m$. Measurements for free-surface elevation and velocity were taken at four different locations in space as shown in Figure \ref{FigP}.
In our numerical simulations, we consider the computational domain  $[0, 8]\times [0, 0.6]$ which is discretized into  $15445$ triangular cells. The water surface elevation is computed at the locations $L_1$ $(0, 0.3)$, $L_2$ $(1.2, 0.3)$, $L_3$ $(2.4, 0.3)$, and $L_4$ $(3.6, 0.3)$. We use the set of experimental data for free-surface elevation and velocity at $L_1$ in our numerical tests as time-varying inflow condition at the left boundary of the domain to generate incoming wave \cite{duran2014recent,bonneton2003dynamique,marche2005theoretical}. Wall boundary conditions are considered at the lateral sides and outflow condition is imposed at the right side of the computational domain. The Manning friction coefficient is set to $n_f=0.01\, m^{-1/3}s$.

\begin{figure}[!ht]
\begin{center}
\includegraphics[scale=0.55]{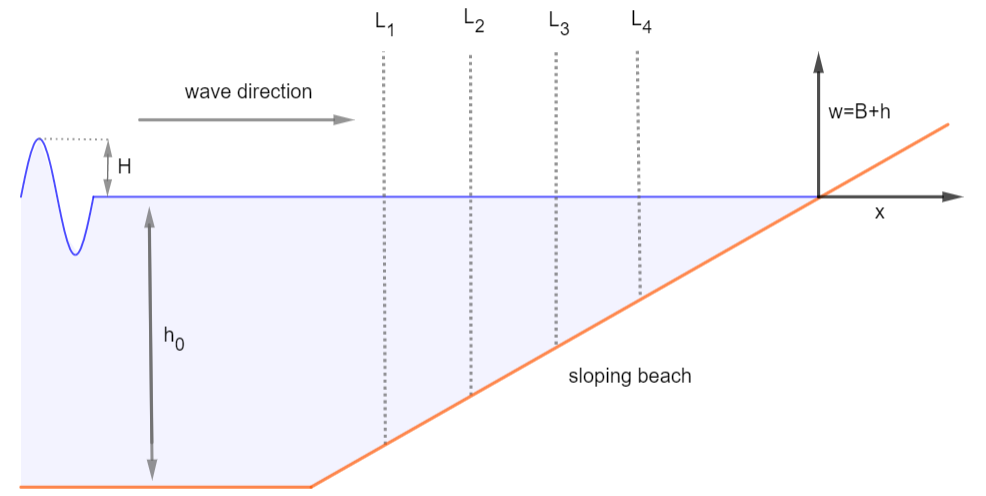}
\caption{A sketch of periodic wave propagation on a sloping beach.}
\label{FigP}
\end{center}
\end{figure}

The time evolution of the computed free-surface elevation and experimental measurements at different locations $L_1$, $L_2$, $L_3$, and $L_4$ are shown in  Figure \ref{Fig5}. A good description of the wave distortion  with sharped  profile is observed when the wave approaches the bed slope shoreline. The numerical results show that the proposed numerical model is accurate in predicting the evolution of periodic waves on the sloping beach. 
\begin{figure}[!ht]
\begin{center}
\includegraphics[scale=0.38]{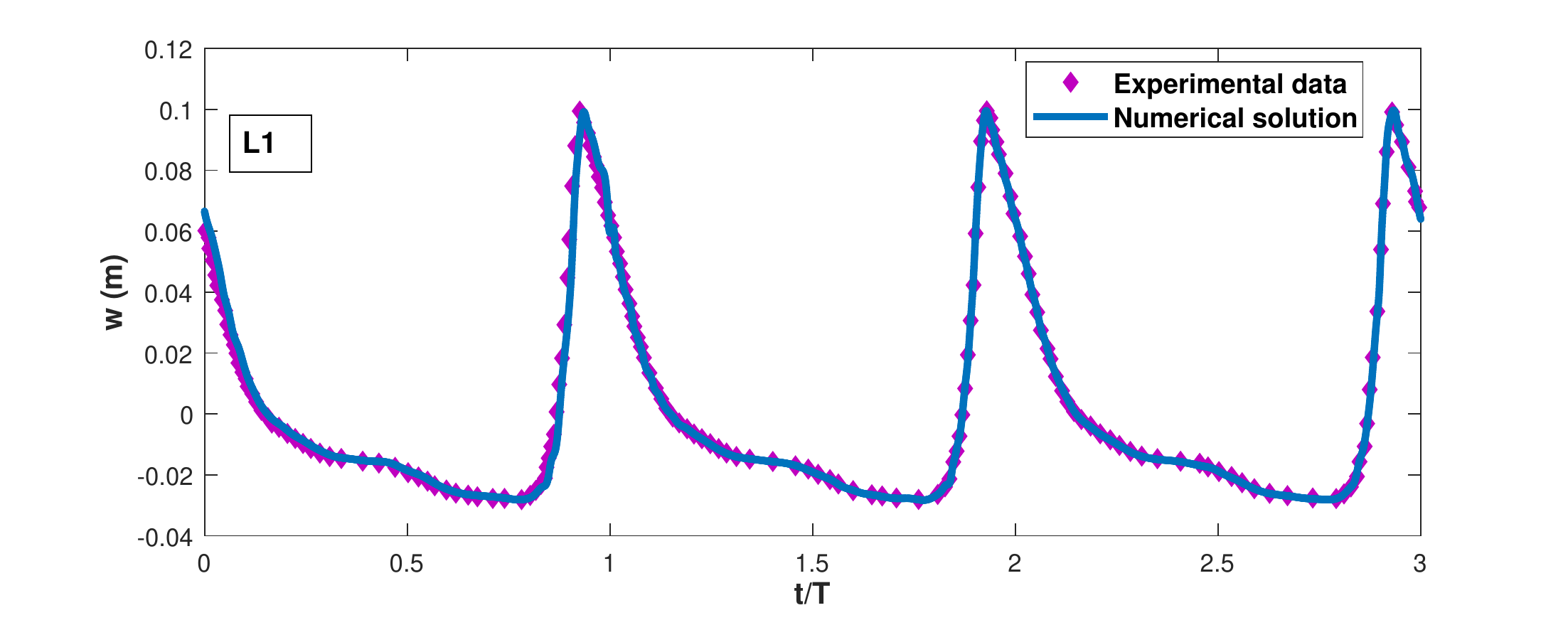}\includegraphics[scale=0.38]{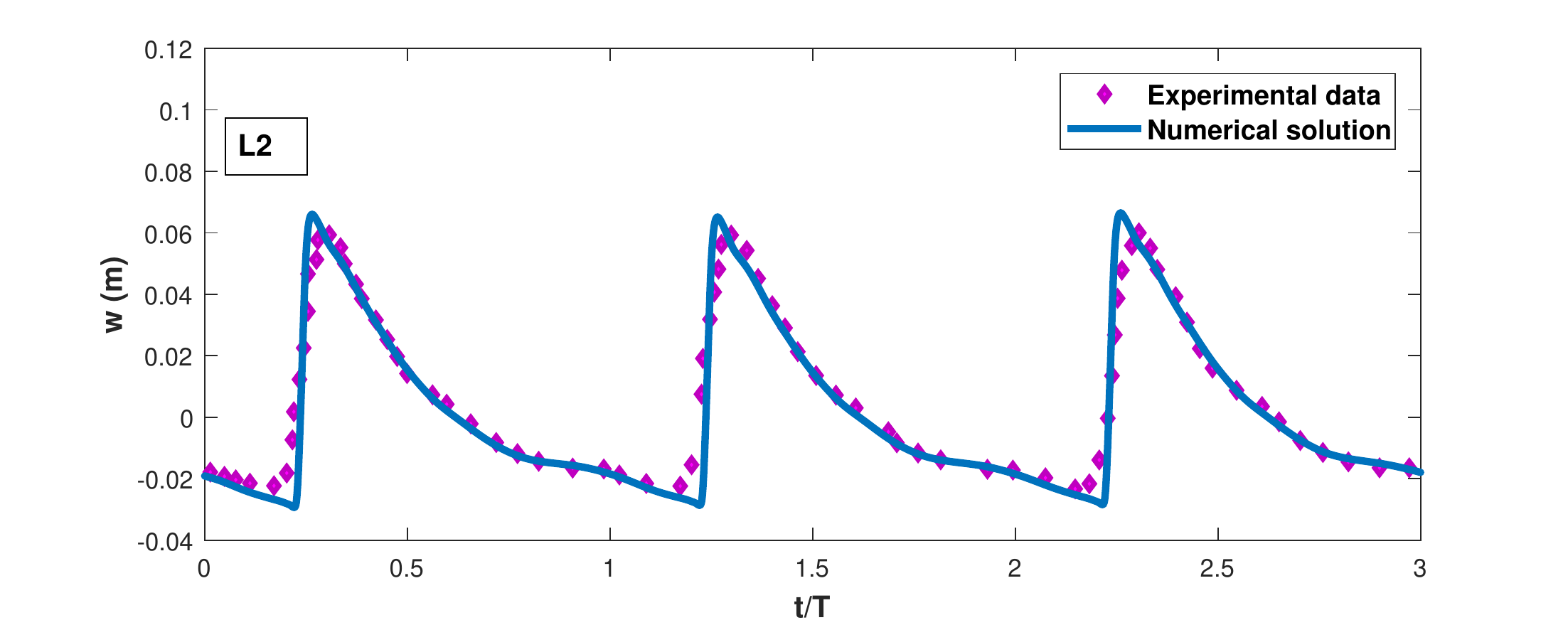}\quad 
\includegraphics[scale=0.38]{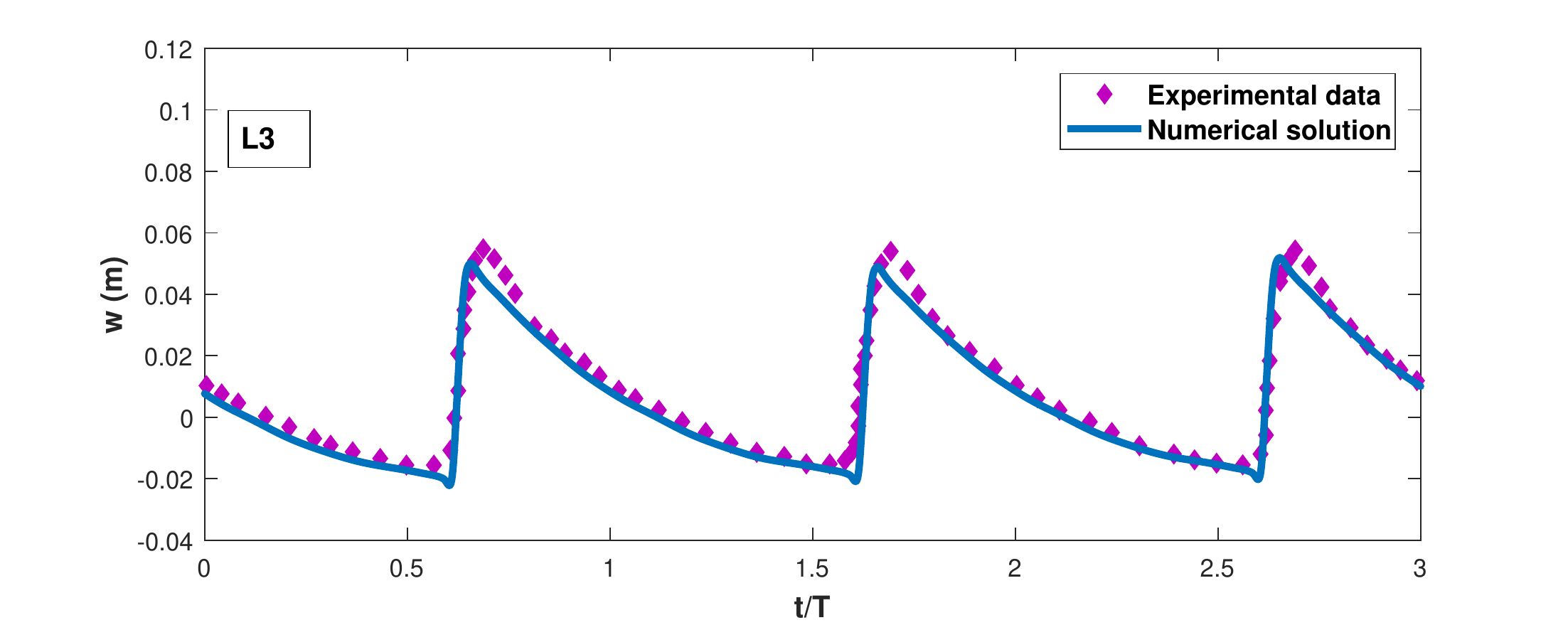}\includegraphics[scale=0.38]{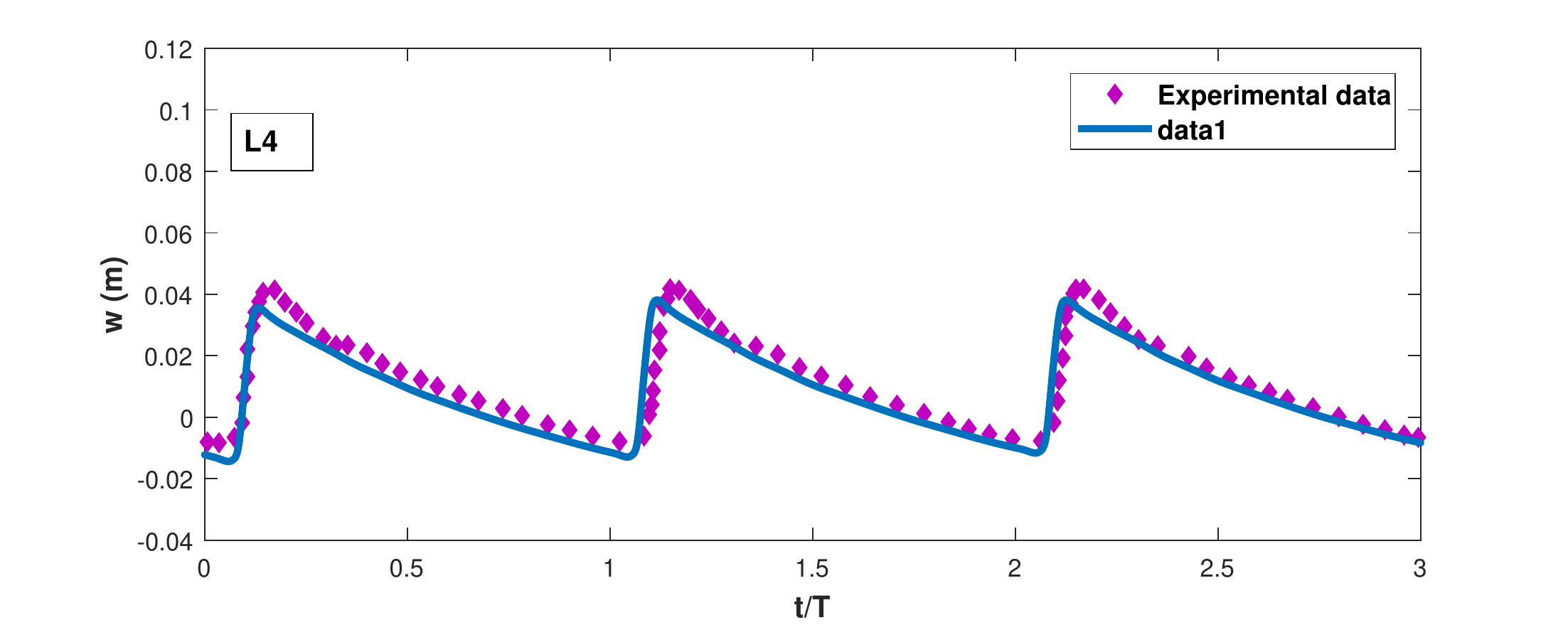}\quad 
\caption{The time evolution of simulated and measured free-surface elevation at the four locations.}
\label{Fig5}
\end{center}
\end{figure}

\newpage
\subsection{Solitary wave run-up on a conical island}\label{Ex4}
In this numerical example, we further test the accuracy of the proposed numerical model where we consider a laboratory experiment of solitary wave on a conical island. The experiment was carried out by Briggs et al. \cite{briggs1995laboratory} and it is a simple representation of the Babi island in the Flores Sea in Indonesia. This laboratory experiment was conducted on a basin of $25\, m$ long, $30\, m$ wide with a circular island of base diameter $7.2\, m$, top diameter  $2.2\, m$, height
$0.625\, m$, and side slope $1:4$  centered  at $\bm{x}=(12.98,13.80)$ as shown in Figure \ref{Figbr}. A directional spectral wave-maker was used to generate waves at the left boundary of the basin. In the numerical simulations, the computational domain $[0, 25]\times [0,30]$ is discretized into 90879 triangular cells. Initially,  the water is at rest and incoming solitary waves of height $H$ and constant depth $h_0=0.32\,m$  are generated at the left boundary of the domain, given by \cite{nikolos2009unstructured,bradford2002finite}:
\begin{equation}
w(t)=H\sech^2\left[\sqrt{\dfrac{3 H}{4h_0^3}}ct \right],\,\, {u}(t)=\dfrac{c w(t)}{w(t)+h_0},
\end{equation}
where we study the case for $H/h_0=0.1$ and $H/h_0=0.2$ solitary waves. The Manning friction coefficient $n_f=0.016\, m^{-1/3} s$ is used,  and wall boundary conditions are imposed at the lateral parts while inflow and outflow boundary conditions are respectively used at the left and right sides of the domain. The numerical simulations are performed until the final time $t=20\, s$ and the computed free-surface elevation is compared to experimental data at four locations $S_1$ $(9.36, 13.80)$, $S_2$ $(10.36, 13.8)$, $S_3$ $(12.96, 11.22)$, and $S_4$ $(15.56, 13.80)$ distributed around the island.
\begin{figure}[!ht]
\begin{center}
\includegraphics[scale=0.55]{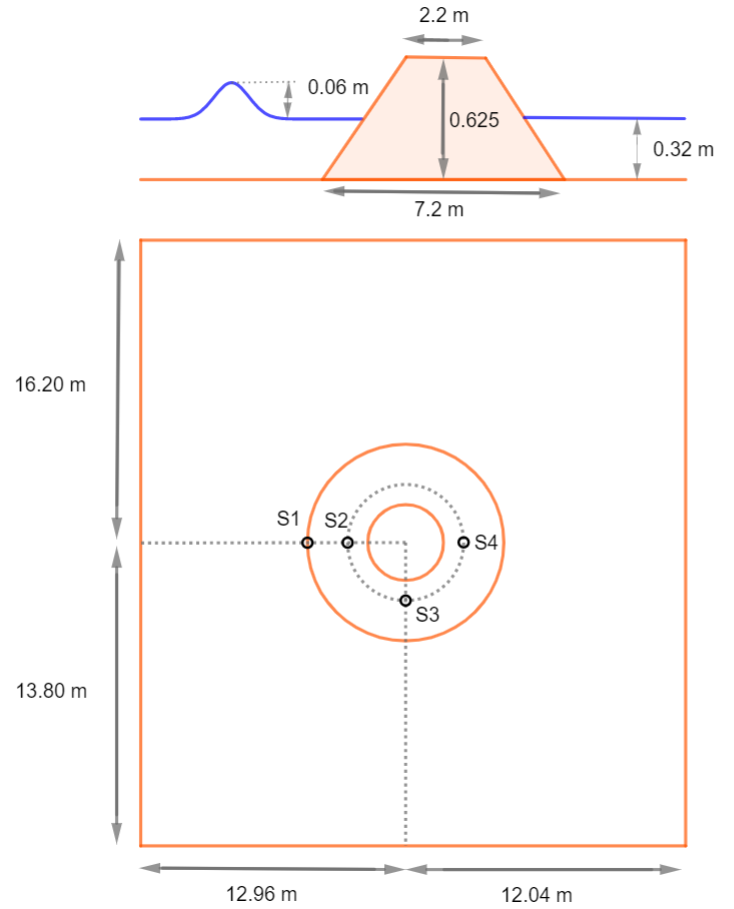}
\caption{A sketch for the solitary wave on a conical island. }
\label{Figbr}
\end{center}
\end{figure}
The predicted and measured free-surface elevation of the solitary waves are presented in Figure \ref{Fig31}. We observe a good agreement between  the computed numerical solution and laboratory experimental data where the maximum amplitude, arrival time, and phase of the leading waves are well predicted.  Furthermore, as emphasized in the previous studies \cite{liu1995runup,yamazaki2009depth,wei2006well}, the results show that the breaking appears everywhere around the island for a solitary wave with $H/h_0=0.2$ which is characterized by a steeper front as it passes around the conical island.
Figure \ref{Fig41} shows the three-dimensional view of the predicted solitary waves profile over the domain at times $t=4$, $6$, $8$, and $10\, s$. The waves advance first over the domain pushing a large amount of water towards to the conical island.  In both two case studies, once the wave arrives the shoreline it climbs up the island to achieve a maximum run-up and then recedes. As expected, when the wave impinges the conical island it reflects and splits into two secondary waves surrounding the island and colliding behind it \cite{liu1995runup,yamazaki2009depth,wei2006well}.


\begin{figure}[!ht] 
\begin{center}
\includegraphics[scale=0.38]{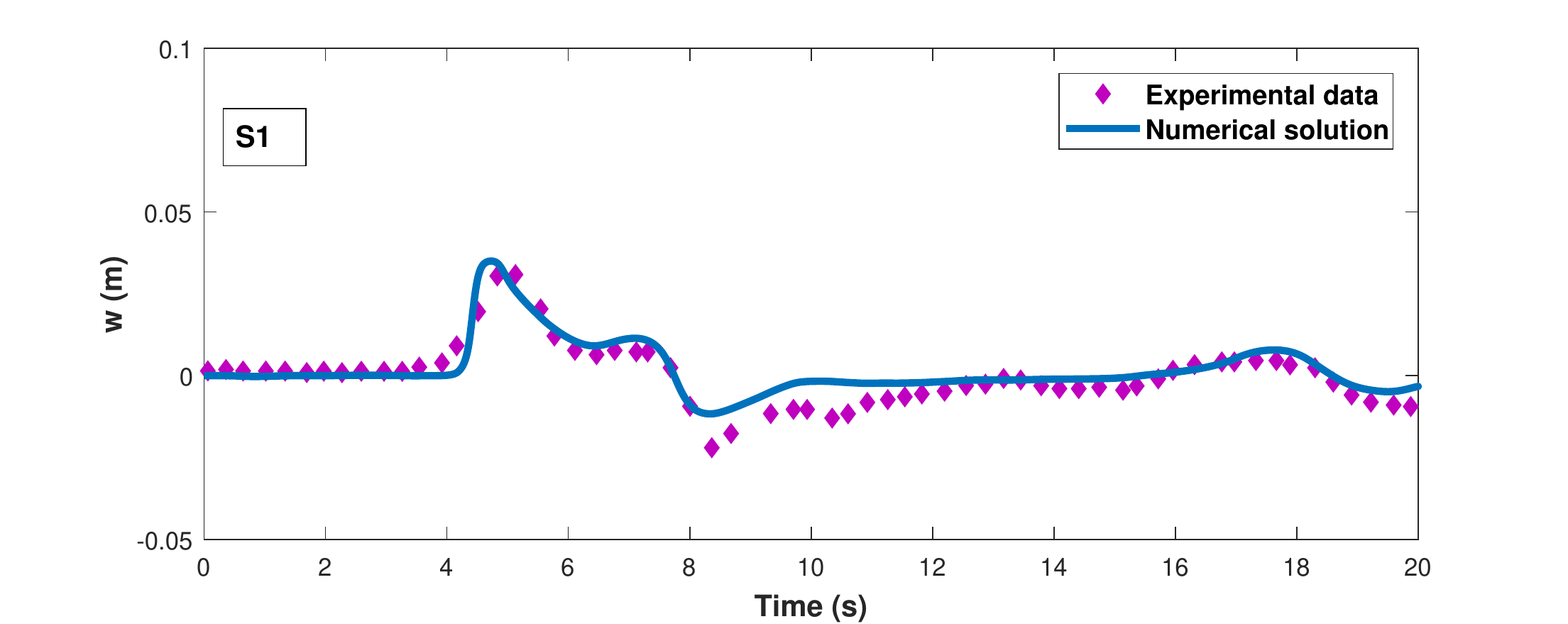}\includegraphics[scale=0.38]{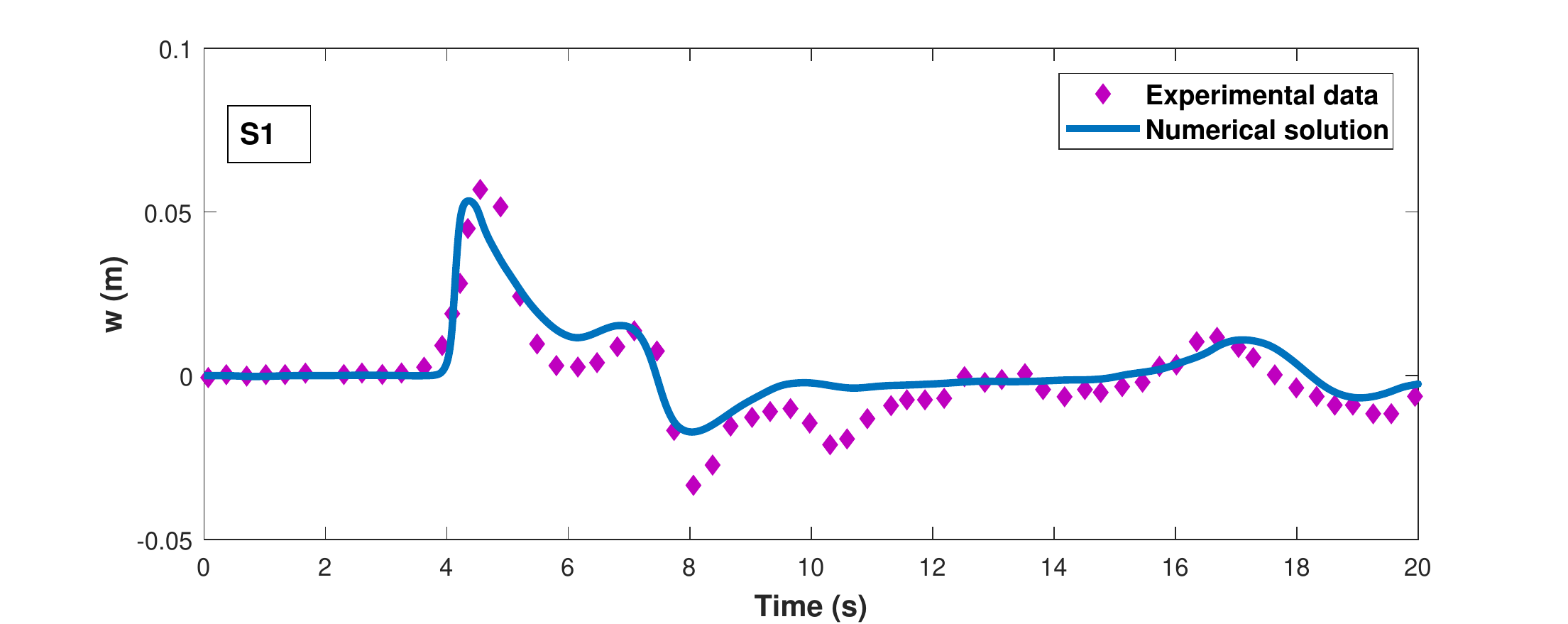}\quad 
\includegraphics[scale=0.38]{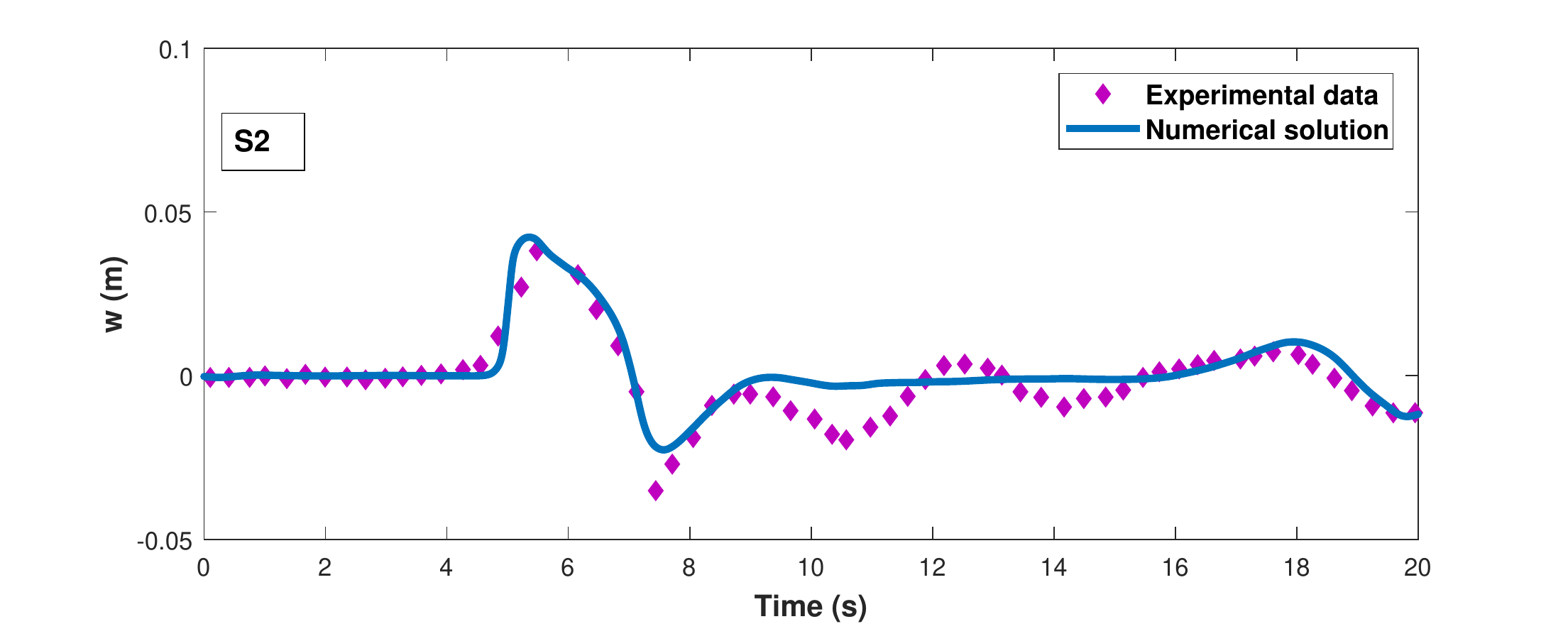}\includegraphics[scale=0.38]{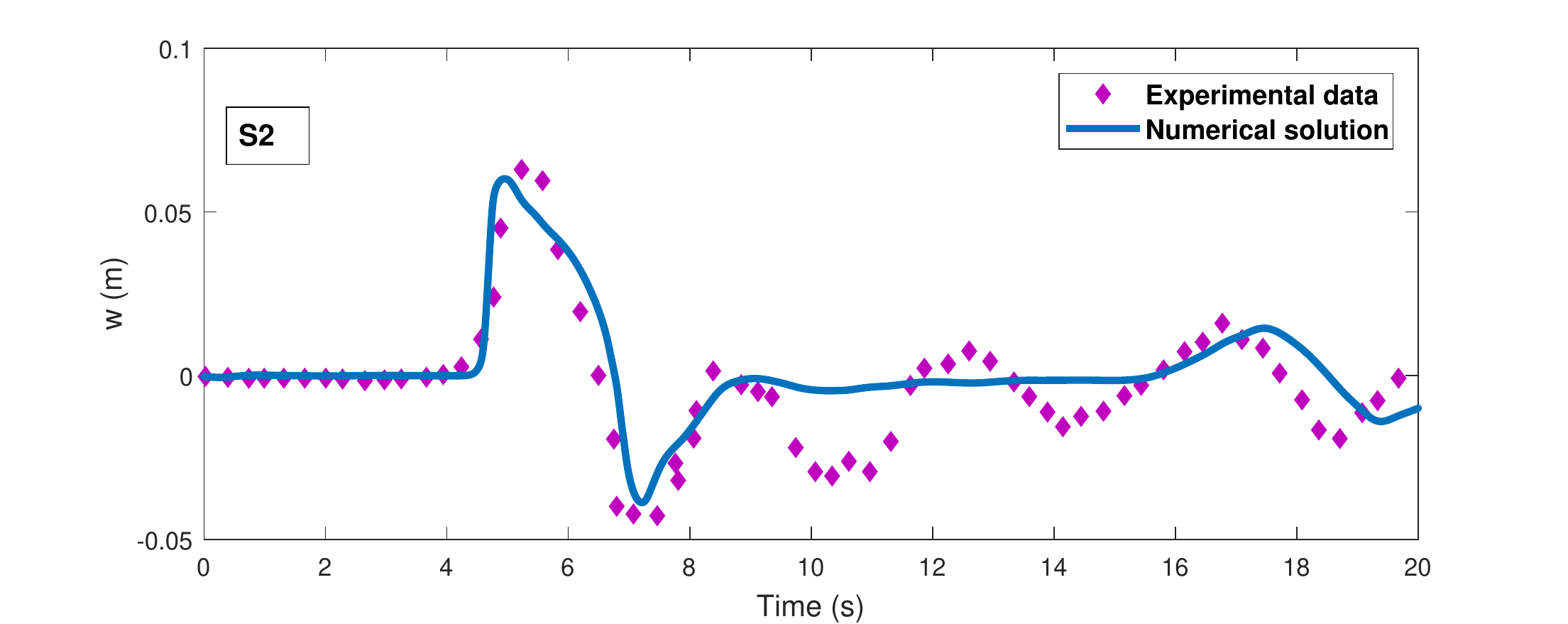}\quad 
\includegraphics[scale=0.38]{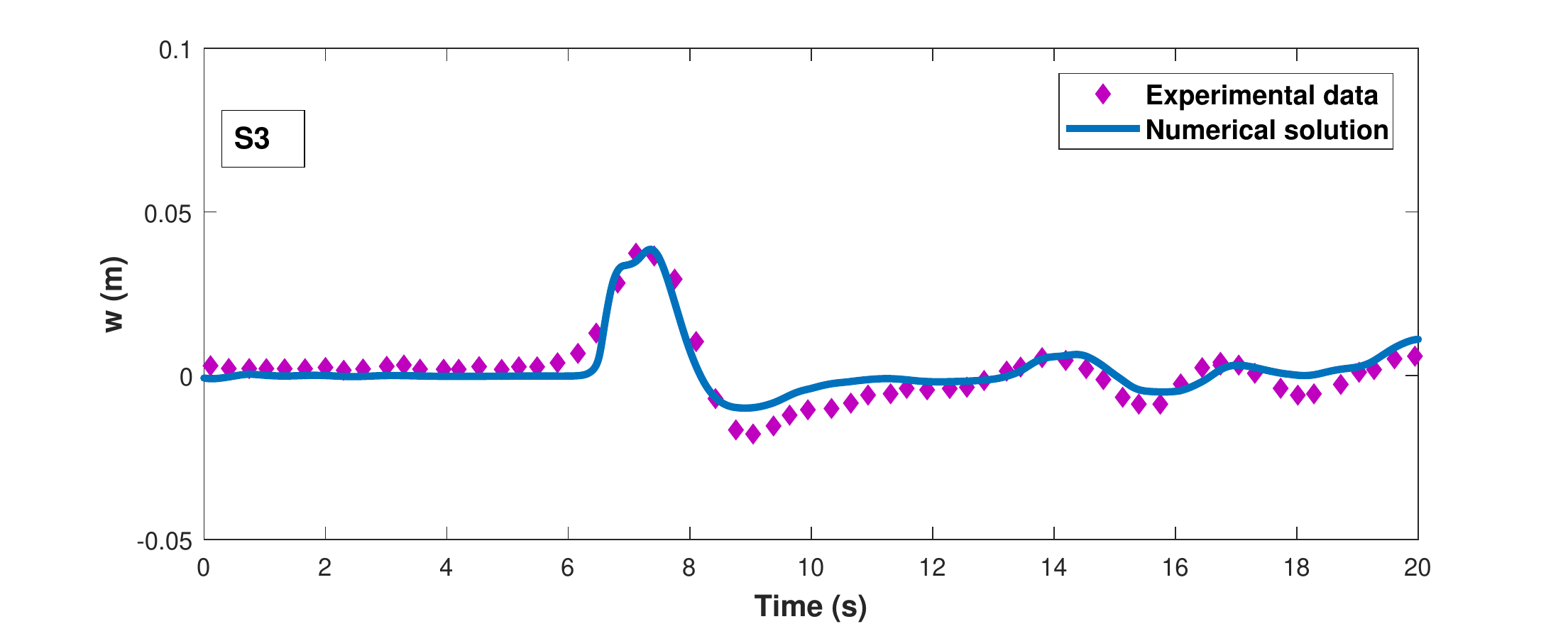}\includegraphics[scale=0.38]{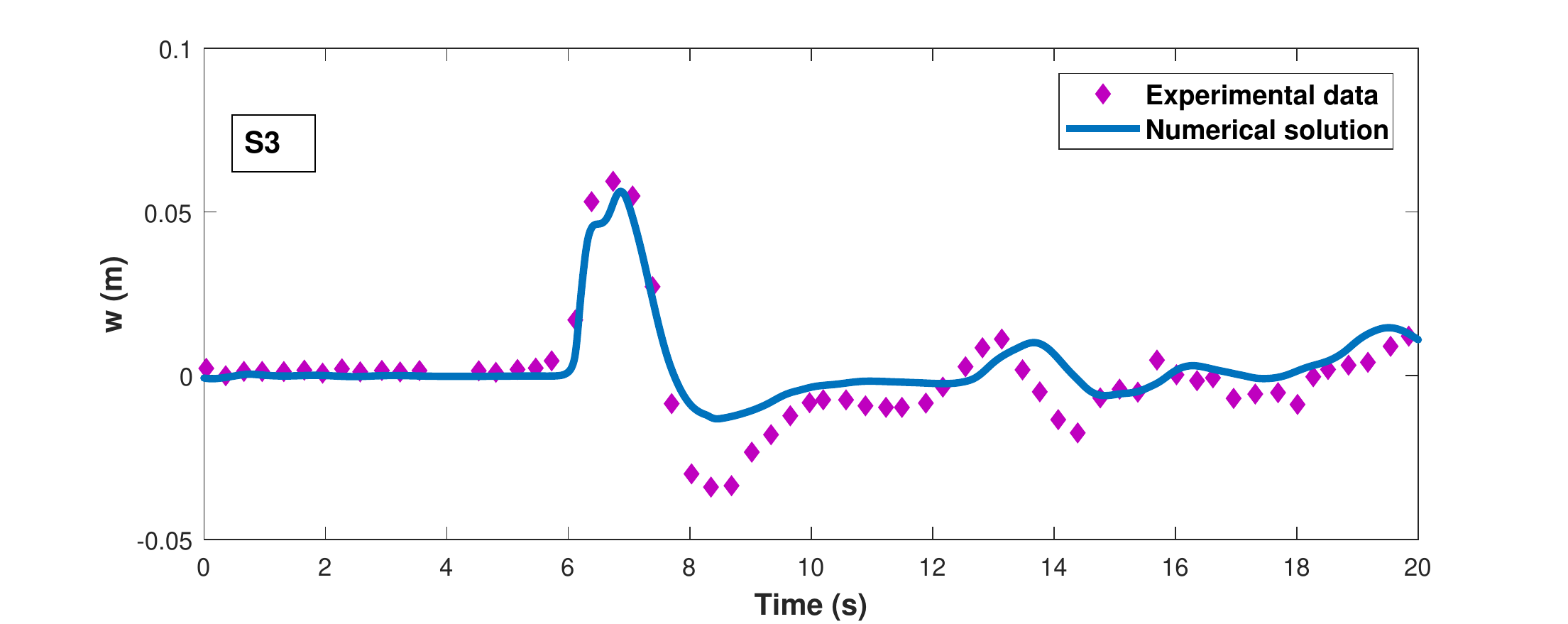}\quad 
\includegraphics[scale=0.38]{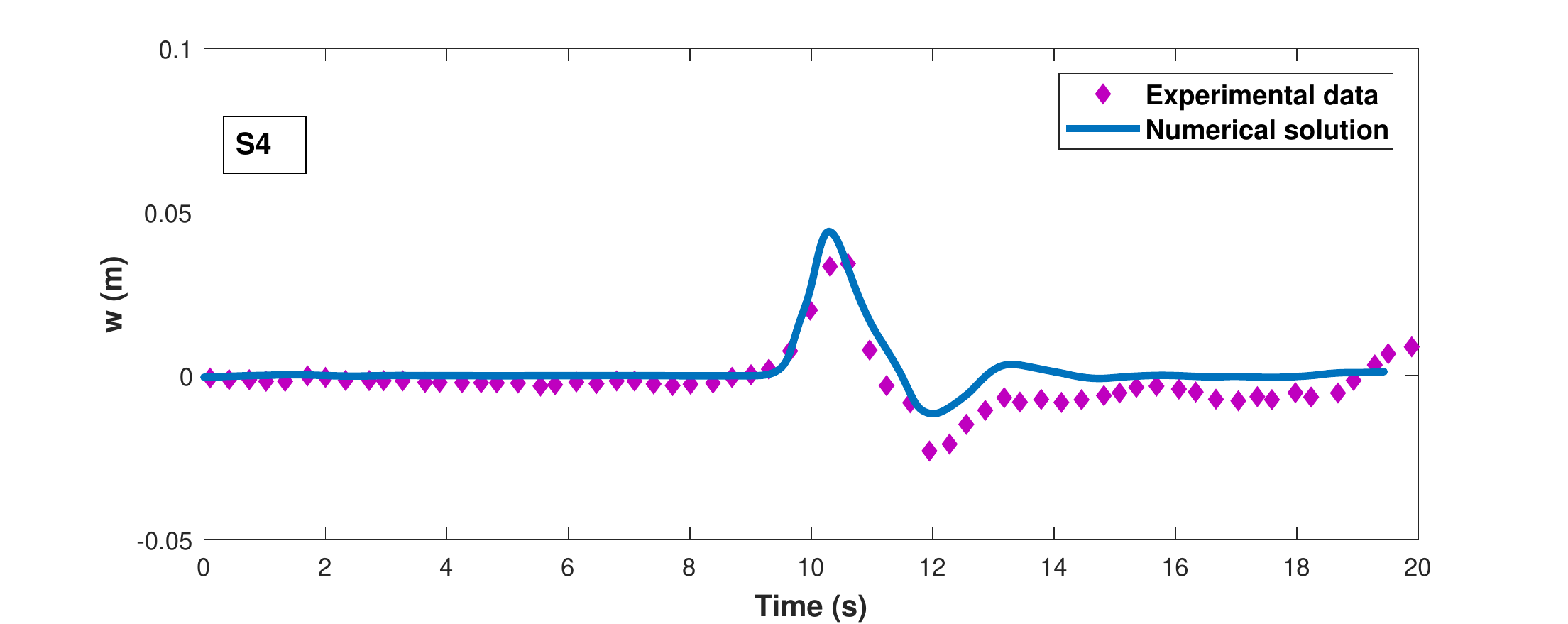}\includegraphics[scale=0.38]{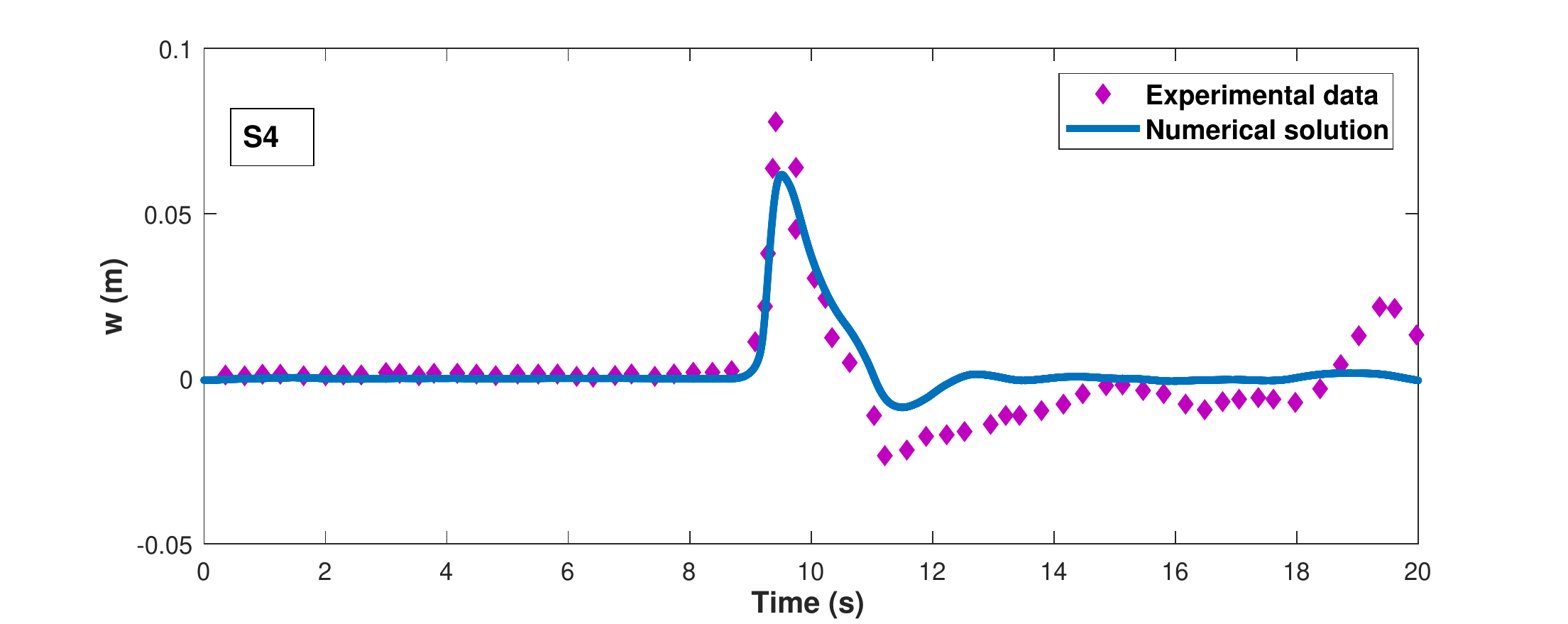}\quad  
\caption{The time evolution of the free-surface elevation computed at the four gauges: Numerical solution compared with experimental measurements for $H/h_0=0.1$ (left) and $H/h_0=0.2$ (right) solitary waves.}
\label{Fig31}
\end{center}
\end{figure}

\begin{figure}[!ht]
\begin{center}

\includegraphics[scale=0.18]{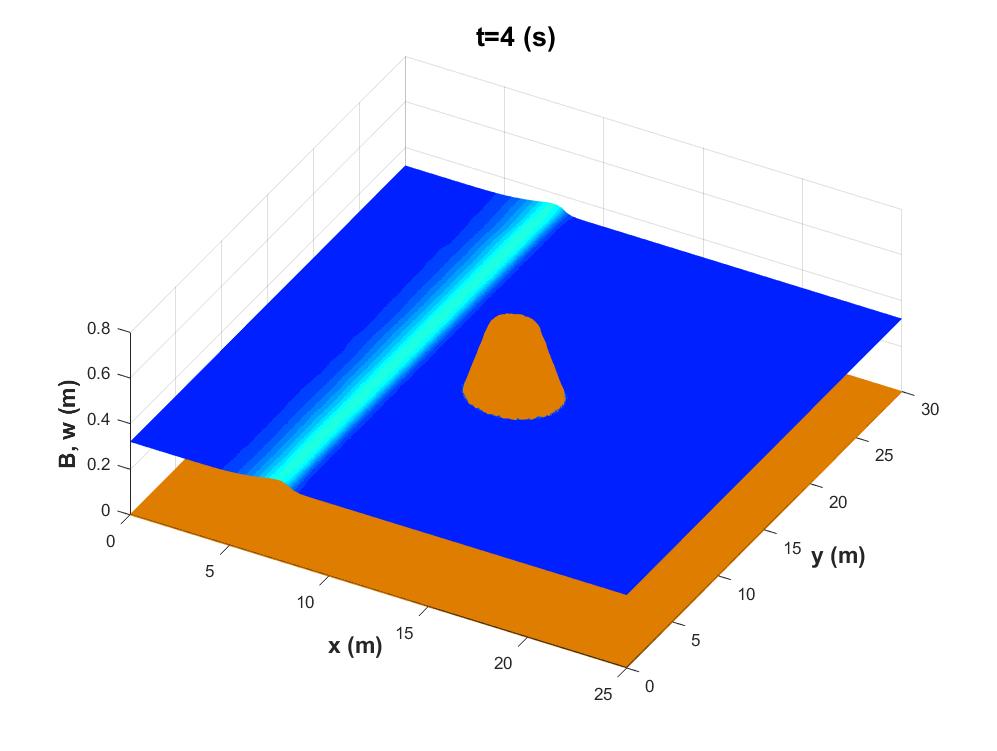}\quad
 \includegraphics[scale=0.18]{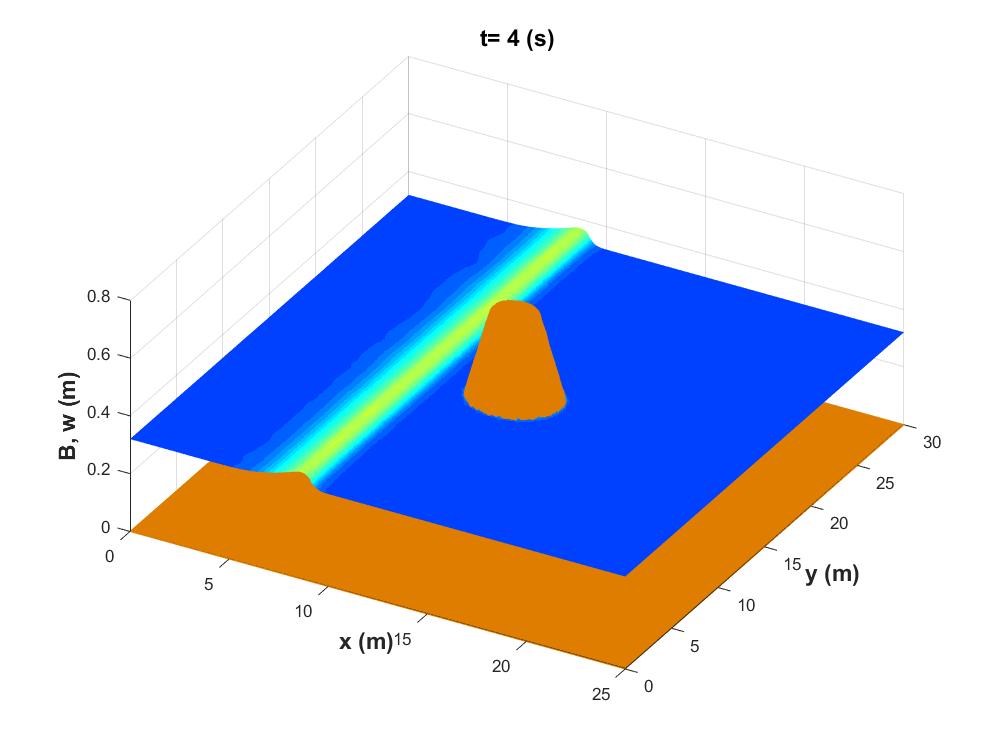}
 
\includegraphics[scale=0.18]{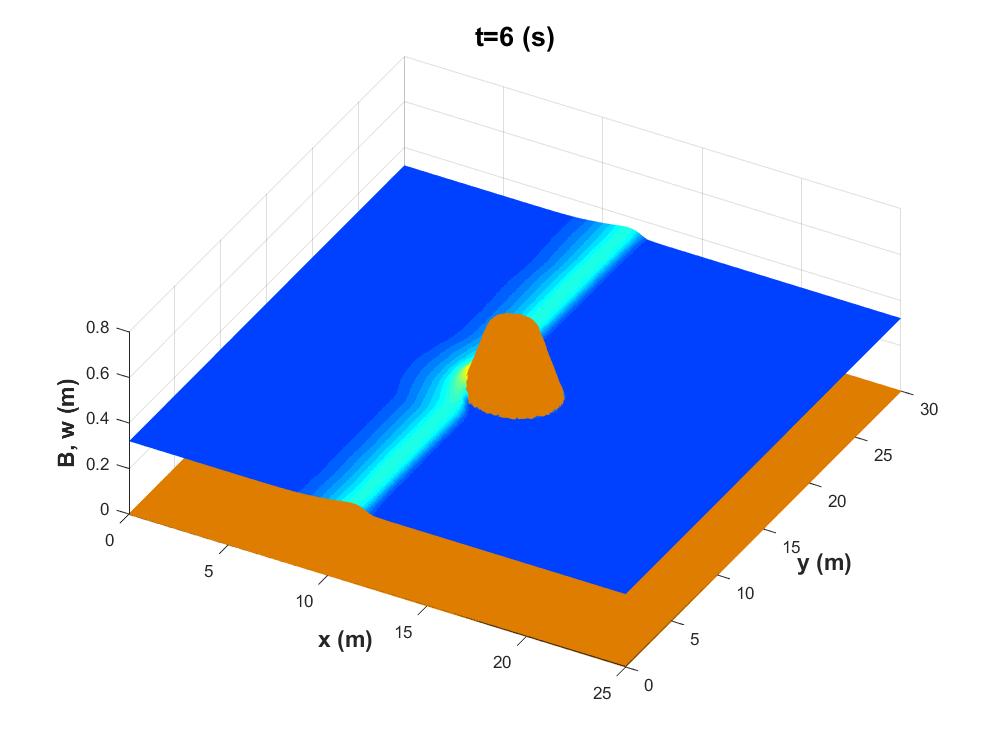}\quad \includegraphics[scale=0.18]{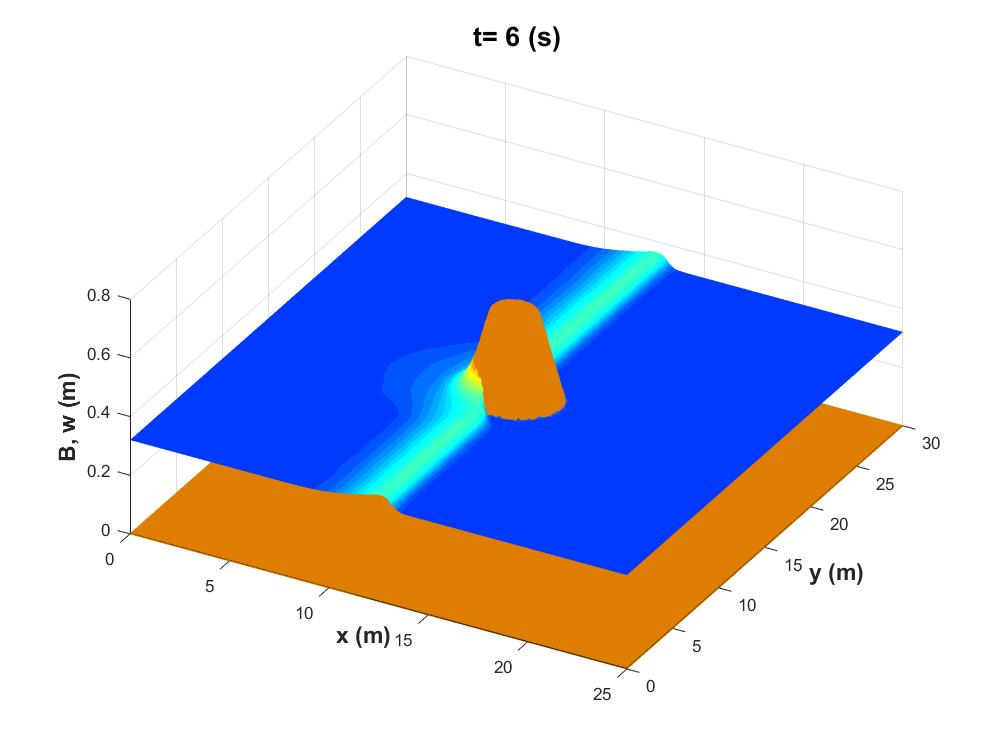}\\
\includegraphics[scale=0.18]{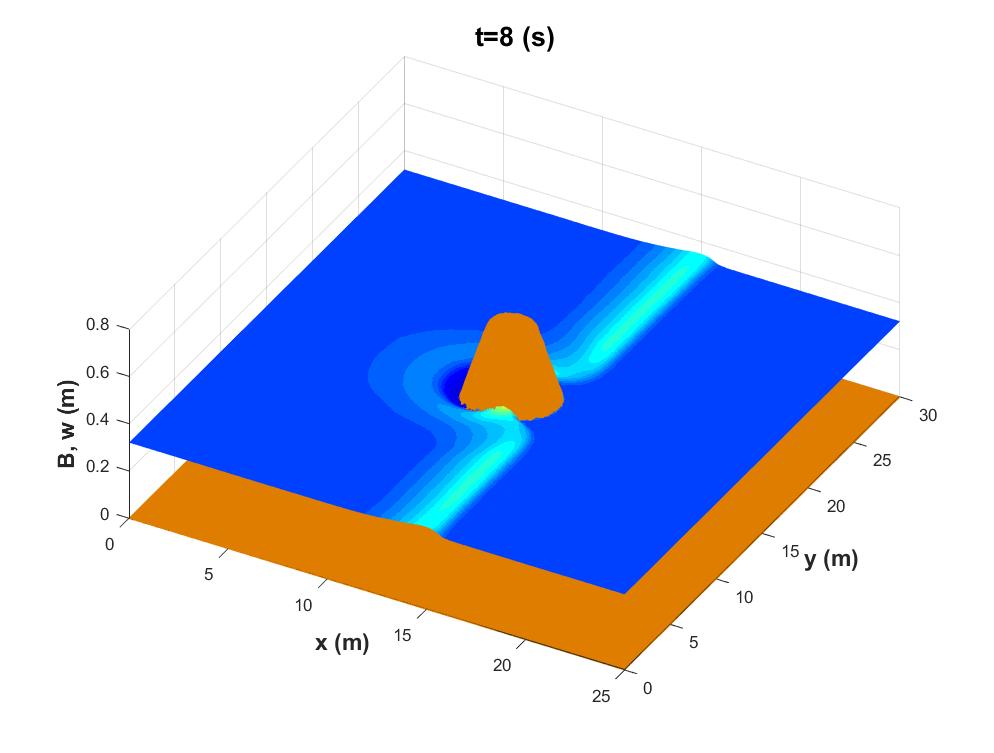}\quad
 \includegraphics[scale=0.18]{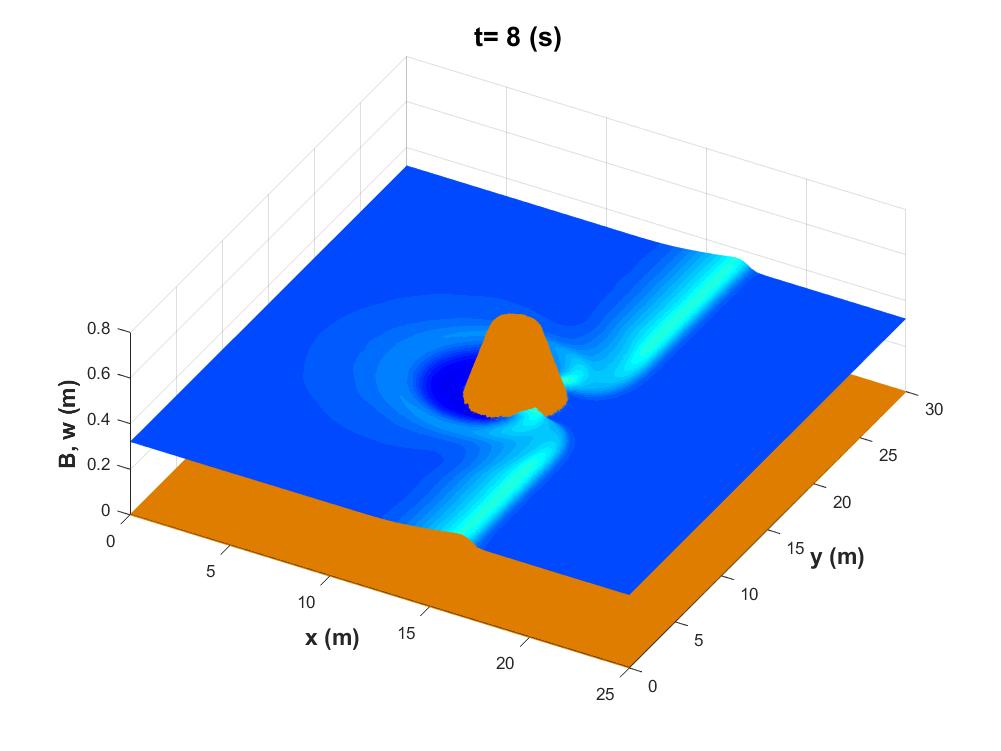}
 
\includegraphics[scale=0.18]{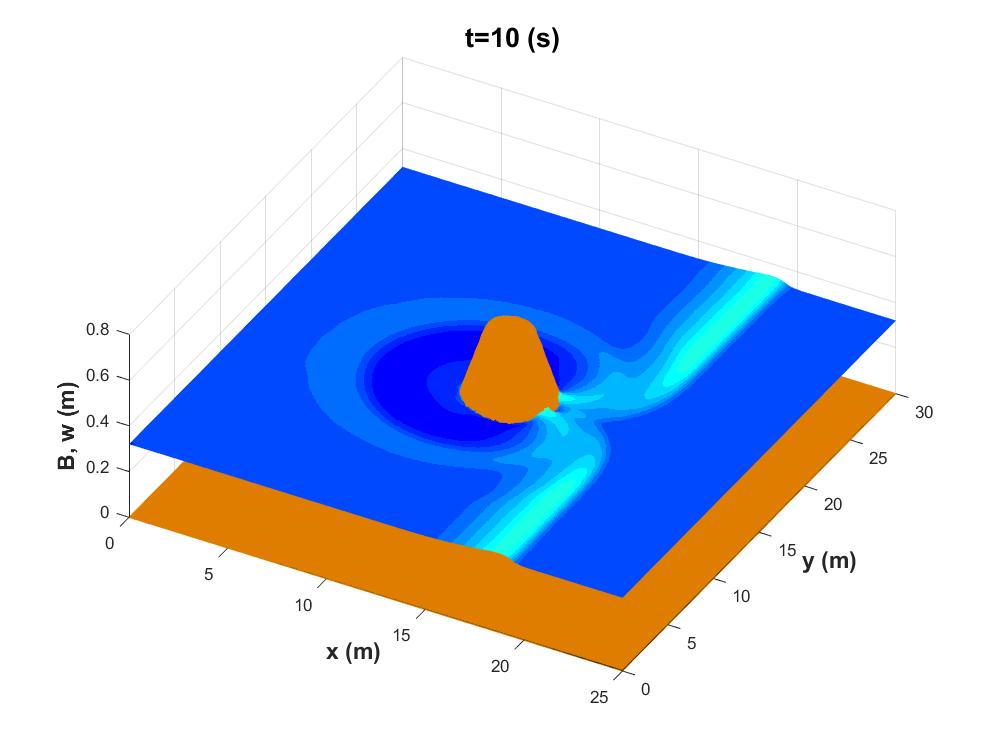}\quad \includegraphics[scale=0.18]{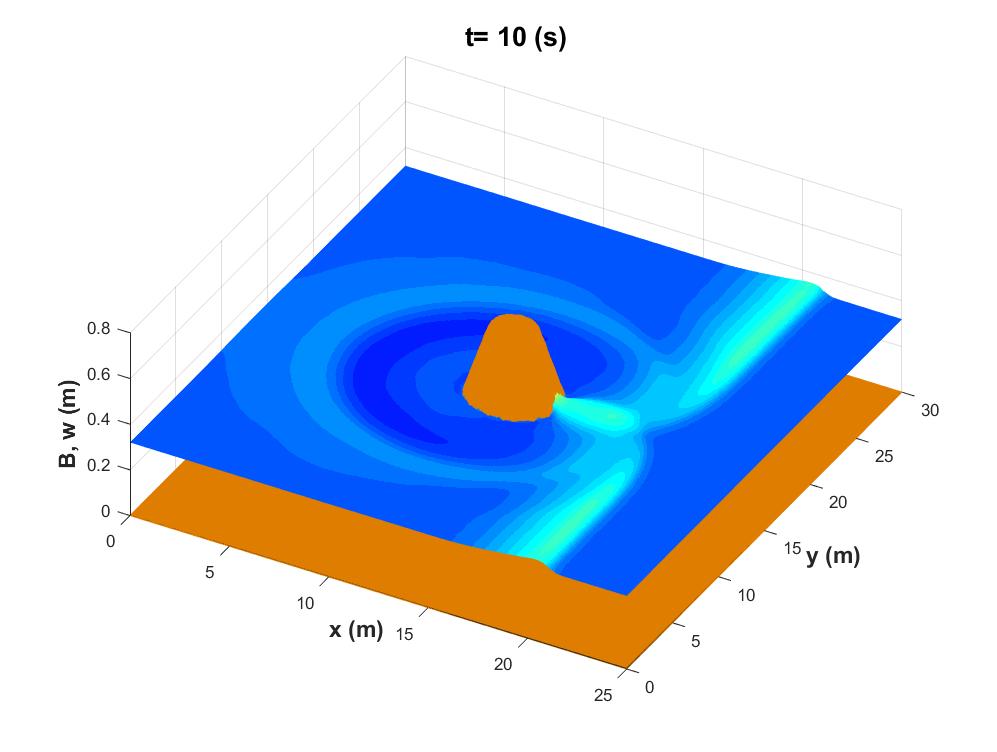}
\caption{Modeling of solitary waves for $H/h_0=0.1$ (left) and $H/h_0=0.2$ (right) over the conical island: Three-dimensional view of the computed free-surface elevation at times $t=4\, ,t=6\, ,t=8\,$,  and $t=10\, s$.}
\label{Fig41}
\end{center}
\end{figure}
\newpage
\subsection{Waves propagation over a complex sloping beach }\label{Ex5}
In this numerical example, we consider a modified test used in  \cite{brocchini2001efficient} to predict the evolution of waves propagating along a sloping beach with complex bottom topography. In our study, we use the bottom topography instead of still water depth \cite{brocchini2001efficient}. The sloping beach geometry consists of a curved topography connected to a constant depth region
\cite{brocchini2001efficient,ozkan1997fourier}. The bottom topography is defined as follows:
\begin{equation}
B(x,y)=\left\{\begin{aligned}
&-h_s,\,\,\quad\quad\quad\quad\quad\quad \quad\quad\text{if}\, x<L_s\\\
& -h_s+\dfrac{0.4(x-L_s)}{3+\cos\left(\dfrac{\pi y}{L_s}\right)},\quad\text{if}\, x\geq L_s.
\end{aligned}\right.
\label{bay}
\end{equation} 
with a sill water depth $h_s=1.02\, m$ and a half-wide of the sloping beach $L_s=8\, m$. The free-surface elevation is initially constant and incoming wave is generated at the offshore boundary using the following solitary wave \cite{brocchini2001efficient}:
\begin{equation}
w(t)=\alpha h_s sech^2\left[ \sqrt{\dfrac{3g}{4h_s}\alpha(\alpha+1)}\cdot t\right],
\end{equation}
where $\alpha=H/h_s=0.2$ with $H$ is the wave height. In our simulations, the computational domain $[0,30]\times [-8,8]$ is discretized using $37334$ triangular cells. Wall conditions are applied at the side boundaries $y=L_s$ and $y=-L_s$, while inflow and outflow conditions are used at the right $x=0$ and left $x=30$ boundaries, respectively. The Manning friction coefficient is set to $n_f=0.01\, m^{-1/3} s$.

Figure \ref{Fig51} shows the evolution of wave propagating along the sloping beach simulated at different times $t=0,2,4,6,8,10,11$, and $14\, s$. The simulation results show that the wave propagates inside the domain from the offshore boundary towards to the shoreline zone generating wave run-up and run-down processes over the sloping bed. Initially, the water is at rest over the domain and at time $t=0\, s$, the incoming wave is generated at the offshore boundary of the sloping beach. At time $t=2 \,s$, the wave advances over the beach pushing a large amount of water towards to the coast. At time $t=5\,s$, we observe a wave run-up where the wave starts to climb first the lateral borders and then in the middle of the beach at about time $t=7\,s$. The wave recedes from the lateral borders towards the inside the beach while it still propagates and climbs up in the middle of the beach. The wave run-up reaches its maximum in the middle of the beach at about $t= 9\, s$ and then recedes. Finally, from time $t=15\, s$ to $t=23\, s$ the wave run-up and run-down motions occur several times and the surface water becomes progressively elongated and steady for long time due to the bed friction effects. The numerical simulations using the proposed numerical model confirm that the model is suitable in the prediction of the propagation of waves along a complex sloping beach.

\begin{figure}[!ht]
\begin{center}
\includegraphics[scale=0.19]{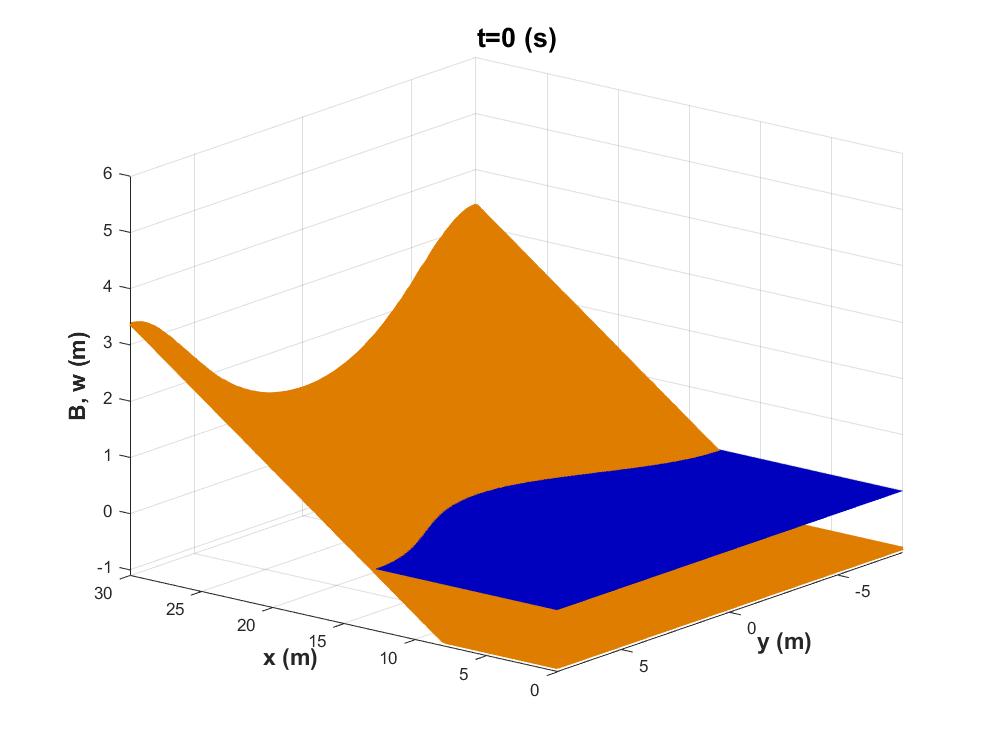}\quad
 \includegraphics[scale=0.19]{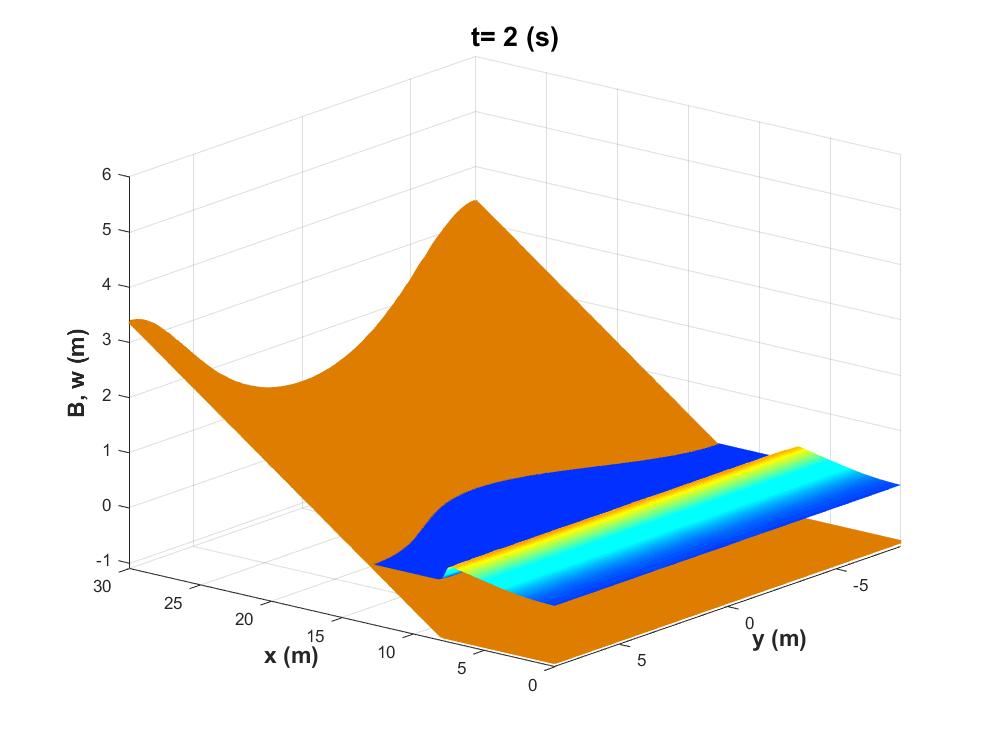}
 
 \includegraphics[scale=0.19]{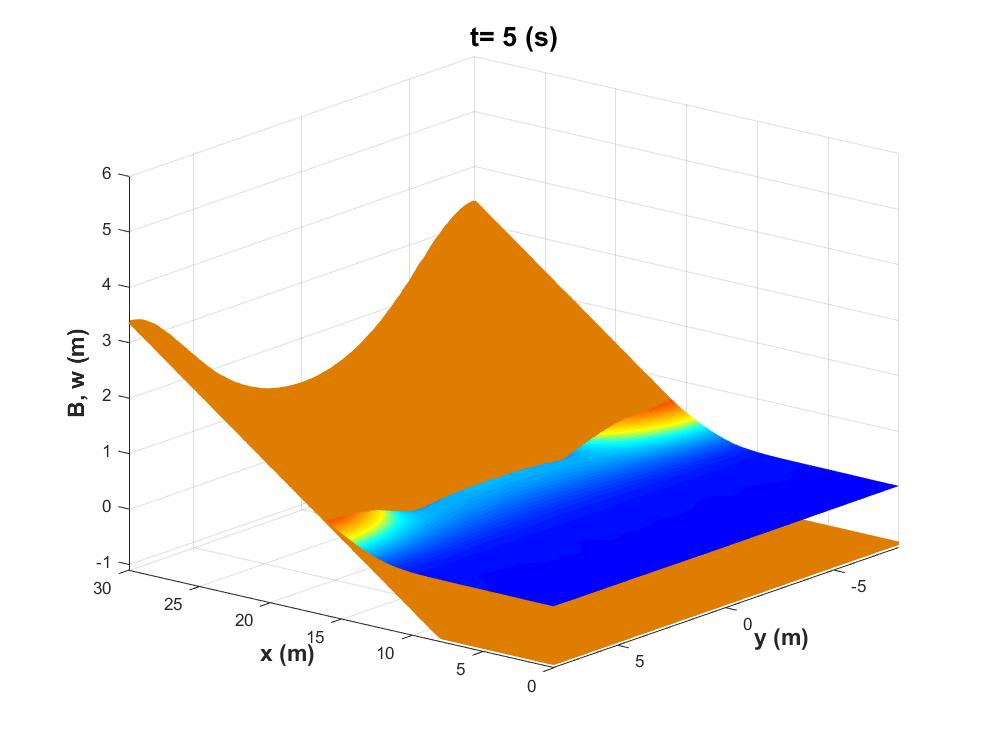}\quad \includegraphics[scale=0.19]{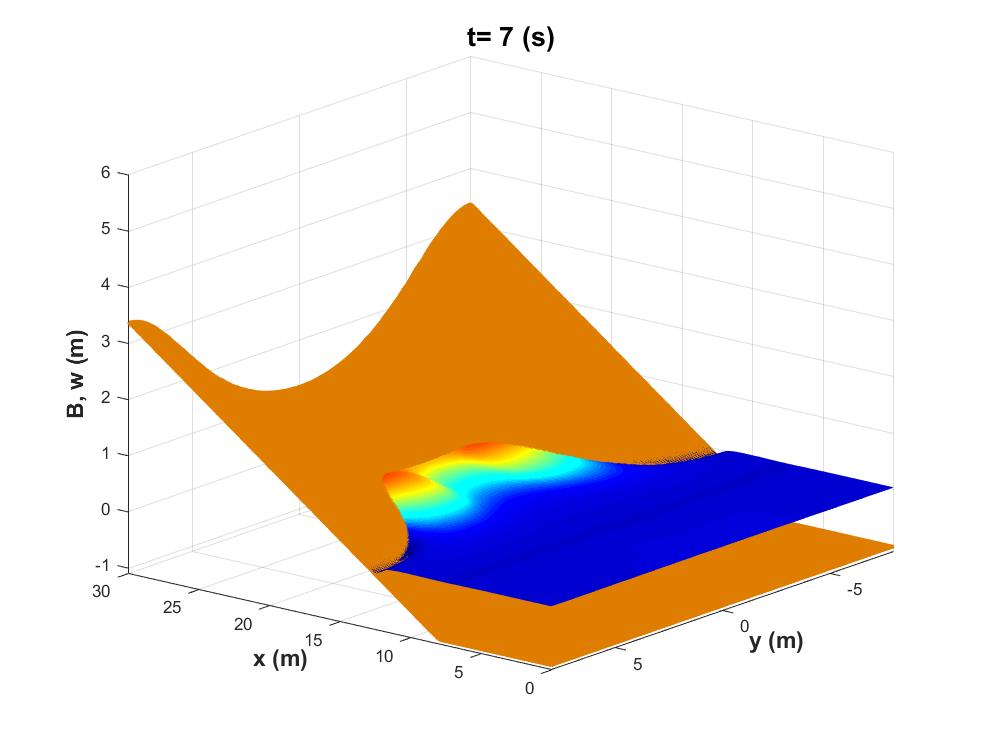}
 
\includegraphics[scale=0.19]{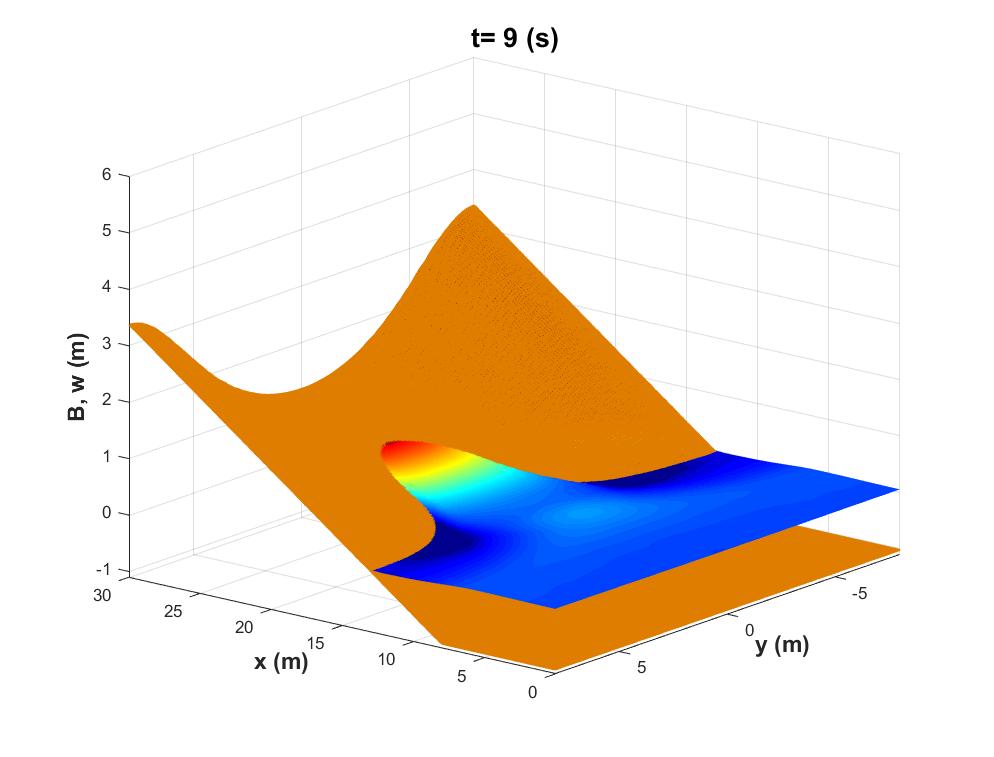}\quad \includegraphics[scale=0.19]{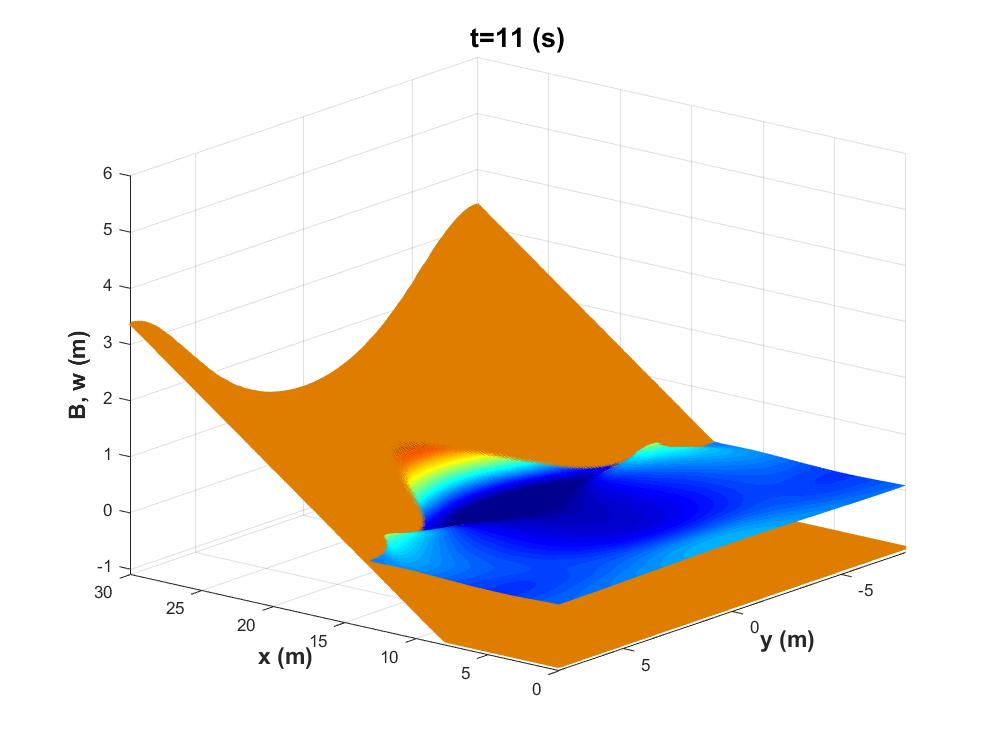}

\includegraphics[scale=0.19]{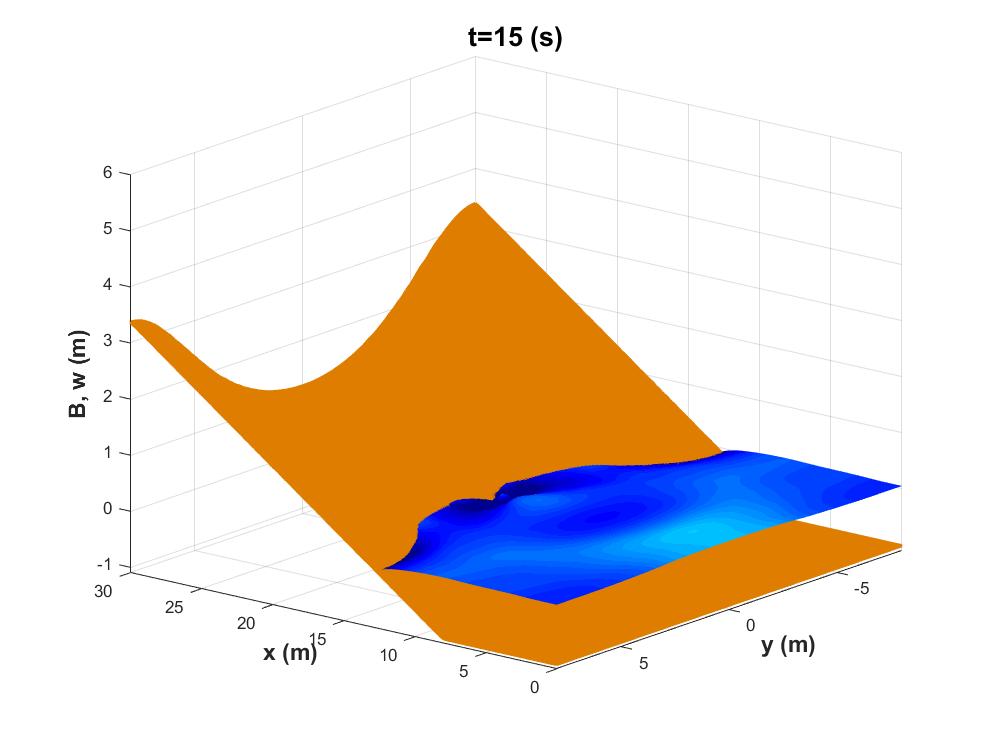}\quad \includegraphics[scale=0.19]{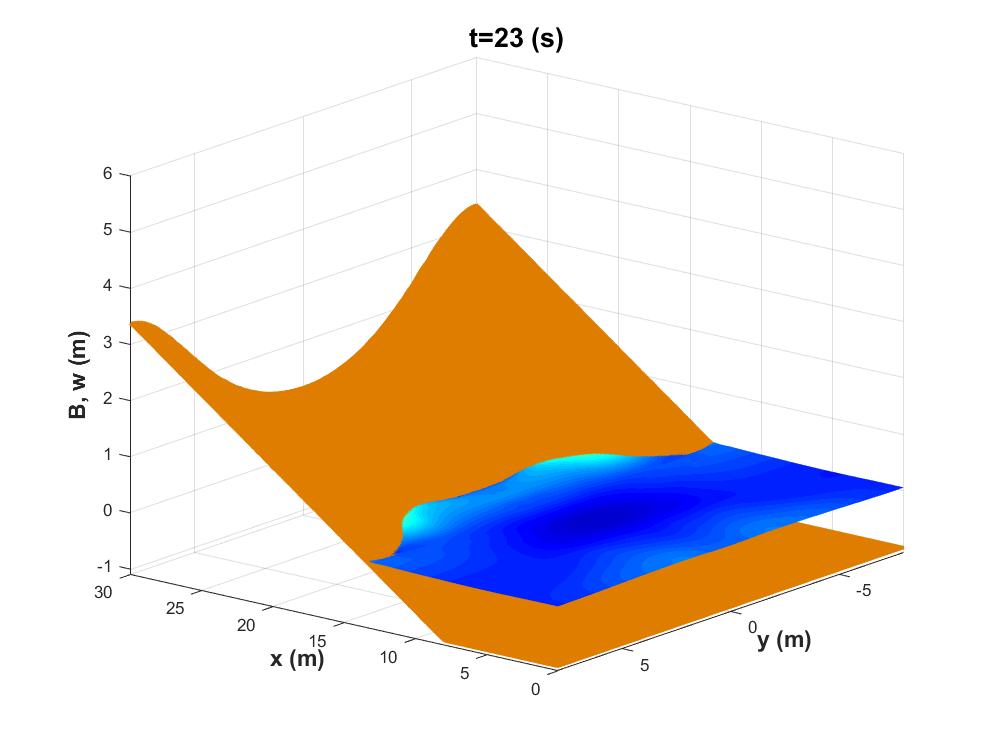}
\caption{Modeling of solitary wave over a sloping beach with complex topography: Three-dimensional  view of the wave propagation along the beach at different times.}
\label{Fig51}
\end{center}
\end{figure}

\newpage
\section{Conclusion}\label{S5}
In this study, we proposed a numerical model preserving a class of nontrivial steady-state solutions for predicting waves run-up on coastal areas. In our approach, unstructured central-upwind scheme is used for solving the shallow water flow model over variable bottom topography with bed friction effects. The numerical scheme preserves the stationary steady-states and guarantees the positivity of the water depth. To ensure the accuracy of the numerical model, we used a linear reconstruction for the variables of the system  where a positivity correction technique is used for the reconstructed values of the free-surface elevation variable to maintain the non-negativity of the computed water depth. Furthermore, we used a special semi-implicit scheme for the friction source term and a well-balanced formulation for the bottom topography to satisfy the balance between numerical fluxes and source terms at the discrete level. We proved that, the numerical model preserves the nontrivial well-balanced property of steady flow over a slanted surface. The robustness and accuracy of the proposed numerical model are tested performing numerical experiments to simulate waves run-up and run-down processes where the results are compared with laboratory experimental data. The simulation results have shown that the wave profile, the moving shoreline and the waves maximum run-up are well predicted and are in good agreement with experimental data. Finally, the proposed numerical model is validated to predict the evolution of waves propagating along a sloping beach with complex bottom topography.

\section*{\textbf{Acknowledgments}}
AB gratefully acknowledges funding form UM6P-OCP to support the PhD program of HK.
\bibliographystyle{elsarticle-num}
\bibliography{ref}


\end{document}